\documentclass[10pt,a4paper]{article_ET_m}
\usepackage{pgfplots}
\pgfplotsset{width=7cm,compat=1.8}
\usetikzlibrary{patterns}
\usepackage[latin1]{inputenc}
\usepackage{latexsym}
\usepackage{amsmath,amssymb,amsfonts,amsthm}
\usepackage{xcolor}
\usepackage{mathrsfs}
\DeclareFontFamily{OT1}{pzc}{}
\DeclareFontShape{OT1}{pzc}{m}{it}{<-> s * [1.10] pzcmi7t}{}
\DeclareMathAlphabet{\mathpzc}{OT1}{pzc}{m}{it}
\usepackage{comment}
\usepackage{bbm}
\usepackage{paralist}
\DeclareRobustCommand\SDE{SDE($b$,$X_0$)}

\usepackage{stmaryrd}

\usepackage{tikz}

\usetikzlibrary{calc}
\usetikzlibrary{decorations.pathreplacing}
\usetikzlibrary{arrows}
\usepgfplotslibrary{fillbetween}
\usepackage{mathtools}

\usepackage[toc,page]{appendix}
\usepackage{enumitem}

\usepackage[numbers,comma,sort]{natbib}

\definecolor{db}{RGB}{0, 0, 130}
\usepackage[colorlinks=true,citecolor=db,linkcolor=black,urlcolor=blue,pdfstartview=FitH]{hyperref}
\usepackage{pdfsync}

\newcommand{\cond}[3]{\|#3\|_{L^{#1}|\mathcal{F}_{#2}}}
\newcommand{\simp}[2]{[#1,#2]^2_{\leq}}
\newcommand{\simptwo}[2]{[#1,#2]^3_{\leq}}
\newcommand{\Dminus}[2]{\overline{[#1,#2]}^2_{\leq}}
\newcommand{\DDminus}[2]{\overline{[#1,#2]}^3_{\leq}}
\definecolor{rp}{rgb}{0.25, 0, 0.75}
\definecolor{dg}{rgb}{0, 0.5, 0}

\newcommand{\R}{\mathbb{R}}

\newcommand{\N}{\mathbb{N}}
\newcommand{\EE}{\mathbb{E}}

\newcommand{\PP}{\mathbb{P}}

\newcommand{\cA}{\mathcal{A}}
\newcommand{\cB}{\mathcal{B}}
\newcommand{\cC}{\mathcal{C}}
\newcommand{\cE}{\mathcal{E}}
\newcommand{\cF}{\mathcal{F}}
\newcommand{\cL}{\mathscr{L}}
\newcommand{\cS}{\mathcal{S}}
\newcommand{\cT}{\mathcal{T}}
\newcommand{\fkb}{\mathfrak{b}}
\newcommand{\vep}{\varepsilon}
\newcommand{\eps}{\varepsilon}
\newcommand{\dd}{{\rm d}}   

\newcommand{\bqn}{\begin{equation}}
\newcommand{\eqn}{\end{equation}}
\newcommand{\bqne}{\begin{equation*}}
\newcommand{\eqne}{\end{equation*}}

\newcommand{\Cat}{\star}

\renewcommand{\leq}{\leqslant}
\renewcommand{\geq}{\geqslant}

\newcommand{\Volt}{\mathfrak{G}}

\DeclareMathOperator{\var}{Var}

\DeclareMathOperator{\law}{\mathcal{L}}

\numberwithin{equation}{section}

\usepackage{todonotes}

\makeatletter
\newcommand{\customlabel}[2]{%
   \protected@write \@auxout {}{\string\newlabel {#1}{{#2}{\thepage}{#2}{#1}{}}}%
   \hypertarget{#1}{#2\hspace{-0.14cm}}
}
\makeatother

\newtheorem*{acknowledgement}{Acknowledgements}
\newtheorem*{funding}{Funding}

\theoremstyle{definition}
\newtheorem{definition}{Definition}[section]

\theoremstyle{remark}
\newtheorem{remark}[definition]{Remark}

\theoremstyle{plain}

\newtheorem{corollary}[definition]{Corollary}
\newtheorem{lemma}[definition]{Lemma}
\newtheorem{proposition}[definition]{Proposition}
\newtheorem{theorem}[definition]{Theorem}
\newtheorem{assumption}[definition]{Assumption}

\author{Lukas Anzeletti\footnote{University of Technology Vienna;  
 \texttt{lukas.anzeletti@asc.tuwien.ac.at}} \and
Lucio Galeati \footnote{Universit\`a degli Studi dell'Aquila; \texttt{lucio.galeati@univaq.it}}
 \and 
Alexandre Richard\footnote{Universit\'e Paris-Saclay, CentraleSup\'elec and CNRS FR-3487;~   \texttt{alexandre.richard@centralesupelec.fr}.} 
\and
Etienne Tanr\'e \footnote{Universit\'e C\^ote d'Azur, Inria, CNRS, LJAD, France; 
\texttt{Etienne.Tanre@inria.fr}}
}

\title{On the density of singular SDEs with fractional noise and applications to McKean--Vlasov equations}
\begin{document}

\maketitle

\begin{abstract}
We investigate properties of the (conditional) law of the solution to SDEs driven by fractional Brownian noise with a singular, possibly distributional, drift. Our results on the law are twofold: i) we quantify the spatial regularity of the law, while keeping track of integrability in time, and ii) we prove that it has a density with Gaussian tails. Then the former result is used to establish novel results on existence and uniqueness of solutions to McKean--Vlasov equations of convolutional type.
\end{abstract}


\medskip

\noindent\textbf{MSC2020 subject classification: } 60H50, 60H10, 60G22, 34A06.

\setcounter{tocdepth}{2}
\renewcommand\contentsname{}
\vspace{-1cm}

\tableofcontents

\section{Introduction}
We are interested in the properties of the law of multidimensional SDEs of the form 
\begin{align} \label{eq:SDEintro}
X_t=X_0+\int_0^t b_s(X_s)\, \dd s + B_t, \tag*{\SDE}
\end{align}
driven by an $\R^d$-valued fractional Brownian motion (fBm) $B$ with an irregular, possibly distributional, time-dependent drift $b$ and a possibly random initial condition $X_0$.
When $b$ is a genuine distribution, or lacks good integrability properties, even making sense of \ref{eq:SDEintro} is challenging, and at this stage the integral in the equation above must be interpreted in a purely formal way. Nevertheless, leveraging on the statistical and pathwise properties of $B$, it is possible to show existence and uniqueness of solutions to \ref{eq:SDEintro}, for a large class of irregular drifts.
This phenomenon is commonly referred to as ``regularization by noise'', see e.g. the monograph by Flandoli~\cite{Flandoli}.
When the driving noise $B$ is sampled as fBm, early contributions in this direction are due to Nualart and Ouknine~\cite{NualartOuknine}; a far-reaching generalization was then established by Catellier and Gubinelli~\cite{CatellierGubinelli} employing nonlinear Young integrals and allowing for distributional drifts in some cases.
With the introduction of the stochastic sewing lemma (SSL) by L\^e~\cite{Le}, the field enjoyed lots of progress in recent years; we refer to \cite{Atetal, GaleatiGerencser, Gerencser, MatPer2024} for several variants of the SSL and to \cite{AnRiTa, GaleatiGerencser, Gerencser, ButLeMyt,ButGal2023} for applications to singular SDEs.
Investigating \ref{eq:SDEintro} is still an active field of study, in particular because existing well-posedness results do not fully match the predictions coming from scaling arguments -- see \cite{GaleatiGerencser,ButLeMyt,butkovsky2024weak} and the discussion in Remark~\ref{rk:scalingMcKV} below.
SSL techniques have also been used to establish regularization by noise results with multiplicative noise, see \cite{DareiotisGerencser,CatellierDuboscq,MatsudaMayorcas,dareiotis2024regularisation}.

Compared to the above references, in this work we restrict ourselves to (ir)regularity regimes where well-posedness for \ref{eq:SDEintro} is already known to hold, and we rather focus on properties of the law $\mathcal{L}(X_t)$ of the unique solution $X$.
This is a non-trivial problem, as the non-Markovianity of the driving noise (fBm) and the singularity of the drift $b$ prevent us from applying classical techniques. Specifically:
\begin{itemize}
\item[a)] When $B$ is sampled as a Brownian motion or L\'evy process, regularity estimates for $\mathcal{L}(X_t)$ (or more generally for the associated Markov transition semigroup) can be obtained by studying the associated Fokker--Planck PDE, which is a purely analytical problem.
Since fBm is not a semimartingale nor a Markov process, a similar strategy is not possible in our setting.
\item[b)] When $b$ is regular, thanks to the Gaussian nature of fBm, one can employ Malliavin calculus techniques to deduce the existence of smooth densities for $\mathcal{L}(X_t)$ with $t>0$; see \cite{CassFriz,HairerPillai} for far-reaching results in the context of rough differential equations. Establishing Gaussian or (sub-)Gaussian upper and lower bounds is also possible, see \cite{tindel,BNOT}.
However, in order to obtain regularity of $\mathcal{L}(X_t)$, Malliavin calculus requires an iterated application of the integration-by-parts formula, which in turns requires higher regularity of $b$ at each step to make sense of the resulting linear SDE associated to the $n$-th order Malliavin derivative. 
\end{itemize} 

The goal of the present work is to develop alternative techniques, compared to the aforementioned ones from a) and b), which are sufficiently robust to cover \ref{eq:SDEintro} with singular drift.
Our main results can be summarized as follows:
\begin{enumerate}[label=\roman*)]
\item two results on the spatial regularity of 
a conditional version of $\mathcal{L}(X_t)$, i.e. conditioning on an earlier time point, while keeping track of its behavior/integrability in time, see Theorem~\ref{mainresult:regularity}; 

\item upper and lower Gaussian bounds for the density of $\mathcal{L}(X_t)$, see Theorem \ref{mainresult:gaussianbound};

\item as a byproduct of i), we obtain novel existence and uniqueness results for McKean--Vlasov equations of convolutional type  (Theorems \ref{Mainresult-McKV} and \ref{Mainresult-McKV-uniqueness}).

\end{enumerate}
A more precise explanation of our main results, their contextualization in the existing literature, and the main ideas behind their proofs is provided in the upcoming subsections.

\subsection{Regularity of the law for ``linear" SDEs}

We consider here \ref{eq:SDEintro} for a given prescribed function (or distribution) $b$ of limited regularity. Following the vocabulary from the Markovian/PDE literature, we refer to this as a ``linear'' SDE, as opposed to the ``nonlinear'' case where $b$ may depend on the law of the solution itself, like in the upcoming Section~\ref{subsec:McKeanVlasov}.
Already in this simpler setting, when $b$ is irregular and the noise is a fractional Brownian motion with $H\neq 1/2$, the literature on regularity of the law of the solution is rather sparse.
In \cite{Oliv}, existence of a density of positive Besov regularity for $\mathcal{L}(X_t)$ is shown for merely H\"older continuous $b$, but its quantitative behavior in time is left unspecified. 
For a class of singular, possibly distributional $b$, \cite[Prop. 3.20]{GHM} proves some joint time-integrability, space-regularity results for $t\mapsto \mathcal{L}(X_t)$, by employing duality arguments and Girsanov transform.
The upcoming Theorem~\ref{mainresult:regularity} sharpens both these references.

We will consider drifts as in~\cite{GaleatiGerencser}, namely time-dependent velocity fields in Lebesgue-Besov spaces fullfilling the following assumption.

\begin{assumption} \label{ass:strong}
Let $b\in L^q \left([0,T]; \mathcal{B}^\gamma_{\infty}\right)$ with parameters satisfying:
	\begin{align*}
		H\in (0,\infty)\setminus \N, \quad q \in (1,\infty], \quad 1-\frac{1}{H (q^\prime\vee 2)} < \gamma<1,
	\end{align*}
where \(\frac{1}{q} + \frac{1}{q^\prime} = 1\).
\end{assumption}

In the above, $\mathcal{B}^\gamma_\infty$ refers to a Besov space, which coincides with the space of $\gamma$-H\"older continuous functions for $\gamma \in (0,1)$ and which contains genuine distributions when $\gamma<0$.
Similarly to \cite{Gerencser, GaleatiGerencser}, we allow $B$ to be sampled as a fBm of Hurst parameter $H\in (0,\infty)\setminus \N$, rather than just the more classical $H\in (0,1)$; this allows to consider driving signals which are differentiable in time.
Moreover, we consider an underlying filtration $\mathbb{F} = (\mathcal{F}_{t})_{t\geq0}$ such that $B$ is an $\mathbb{F}$-fBm; this allows for instance for random, $\cF_0$-measurable initial data (thus necessarily independent of $B$).
We refer to Section~\ref{subsec:notdef} for more details on both Besov spaces and fBm.

Our first main results can be summarized as follows.

\begin{theorem} \label{mainresult:regularity}
Let $b \in L^q([0,T];\mathcal{B}^\gamma_\infty)$ fulfilling Assumption~\ref{ass:strong} and let $X$ be the unique solution to \ref{eq:SDEintro}, starting from an $\mathcal{F}_{0}$-measurable random variable $X_{0}$. 
\begin{enumerate}[label=(\alph*)]
\item \label{en:rega}For any $0<\eta<\gamma-1+\frac{1}{Hq'}$, any $0\leq u < t\leq T$, then almost surely the conditional law $\law(X_t \vert \mathcal{F}_{u})$ admits a density, which satisfies
\begin{equation}\label{eq:intro_main_pointwise_reg}
\Big\lVert \lVert \law(X_t \vert \mathcal{F}_{u})\rVert_{\cB^\eta_1} \Big\rVert_{L^\infty_{\Omega}} \leqslant C( 1 + (t-u)^{-\eta H}). 
\end{equation}
\item \label{en:regb}Moreover, let $(\tilde q, \eta)$ be parameters satisfying
\begin{equation}\label{eq:intro_main_spacetime_reg}
\tilde{q} \in (1,2],\qquad 0<\eta< \min \left\{\gamma-1 + \frac{1}{H},\frac{1}{\tilde{q}H}\right\};
\end{equation}
then for any $u\in [0,T)$, the map $t\mapsto \law(X_t \vert \mathcal{F}_{u})$ belongs almost surely to $L^{\tilde q}([u,T]; \cB^\eta_1)$ and satisfies
\begin{equation}\label{eq:intro_main_spacetime_reg2}
	\Big\lVert \lVert \law(X_\cdot \vert \mathcal{F}_{u})\rVert_{L^{\tilde q}([u,T];\cB^\eta_1)} \Big\rVert_{L^\infty_{\Omega}} \leqslant C (T-u)^{\frac{1}{\tilde q}-\eta H}. 
\end{equation}
\end{enumerate}
\end{theorem}

For more precise and detailed statements, partly with looser assumptions on $b$, see Theorems~\ref{thm:Besovregularity} and~\ref{thm:Besovregularity_better}, corresponding respectively to parts~\ref{en:rega} and~\ref{en:regb} of Theorem~\ref{mainresult:regularity}.

Let us comment on the result.
Estimates~\eqref{eq:intro_main_pointwise_reg}-\eqref{eq:intro_main_spacetime_reg2} are almost surely uniform in $\omega\in\Omega$: even though the conditional law $\law(X_t \vert \mathcal{F}_{u})$ is a random measure, we obtain deterministic upper bounds on its norm.
For any $\eta>0$, estimate \eqref{eq:intro_main_spacetime_reg} explodes as $t-u\to 0$, as expected since $\law(X_t \vert \mathcal{F}_{t})=\delta_{X_t}\notin \mathcal{B}^\eta_1$. If $X$ were Markovian, then one would have
\begin{align*}
	\law(X_t \vert \mathcal{F}_{u})(\dd y)
	= \law(X_t \vert X_u)(\dd y)
	= p_{u,t}(X_u,\dd y)
\end{align*}
where $\{p_{u,t}(x,\dd y)\}_{u\leq t,x\in\R^d}$ is the Markov transition semigroup associated to $X$.
Since in our case the solution to \ref{eq:SDEintro} needs not to be Markovian, whenever conditioning we need to keep track of the whole history $\mathcal{F}_{u}$. Still, our results can be interpreted as a non-Markovian analogue of Besov estimates for $p_{u,t}(x,\dd y)$, uniformly with respect to $x\in\R^d$.

In relation to this, let us mention the recent work \cite{dareiotis2024regularisation}.
Therein, motivated by singular SDEs with multiplicative noise, the authors need to develop estimates on the conditional law of the solution to driftless rough differential equations with smooth coefficients; they do so by means of partial Malliavin calculus.
Our paper provides (in a different context) alternative methods to get statements on conditional laws, without resorting to these techniques.

\begin{remark}
Whenever the pointwise estimate \eqref{eq:intro_main_pointwise_reg} is available, Minkowski's inequality readily implies the validity of \eqref{eq:intro_main_spacetime_reg2} for any $\tilde q$ such that $\eta<\frac{1}{\tilde q H}$. In particular, when $q=\infty$, \ref{en:rega} applies for any $\eta<\gamma-1+\frac{1}{H}$ and thus implies \ref{en:regb} for any $\tilde{q}$ satisfying \eqref{eq:intro_main_spacetime_reg}.
However, when $q<\infty$, the values $\eta$ allowed in \ref{en:regb} can be significantly larger than those in \ref{en:rega}.
To illustrate this, note that when $q\in (1,2]$, as $\gamma$ approaches the threshold $1-\frac{1}{Hq^\prime}$, the condition on $\eta$ in \ref{en:rega} deteriorates to $\eta\sim 0$.
In contrast, we can apply \ref{en:regb} for $\tilde{q}=q\wedge 2$ to deduce that
\begin{equation*}
	\big(\mathcal{L}(X_t)\big)_{t\in [0,T]} \in L^1([0,T];\mathcal{B}^{\frac{1}{H(q\wedge 2)}}_1)\cap L^2([0,T];\mathcal{B}^{\frac{1}{2H}-}_1)
\end{equation*}
even when $\gamma$ is very close to $1-\frac{1}{Hq^\prime}$.
Roughly speaking, \ref{en:rega} is most effective when $q$ is large ($q\sim \infty$), while \ref{en:regb} performs better for $q$ small ($q\sim 1$).
\end{remark}

\begin{remark}
Parts \ref{en:rega} and \ref{en:regb} in Theorem~\ref{mainresult:regularity} are sharp, in the following sense: for fixed $\eta$ as therein, setting ${b\equiv 0}$, using self-similarity of fractional Brownian motion and scaling in Besov spaces, 
one can check that the power of $t-u$ in the estimates \eqref{eq:intro_main_pointwise_reg}-\eqref{eq:intro_main_spacetime_reg2} is optimal; see Remark \ref{rem:optimality_sec4} for more details. However, we do not know whether the allowed range of values for $\eta$ can be further extended.
\end{remark}

The proofs of Theorem~\ref{mainresult:regularity}\ref{en:rega} and Theorem~\ref{mainresult:regularity}\ref{en:regb} are respectively developed in Sections~\ref{sec:Besovregularity} and~\ref{sec:regularity.density}. For both, we interpret the solution to \ref{eq:SDEintro} as a perturbed fractional Brownian motion, $X=\theta+B$, where the perturbation $\theta$ (the drift part of the equation) satisfies some a priori estimates.
This viewpoint naturally leads us to consider a larger class of processes, and we actually prove more general results on the conditional law of processes of the form $\theta + \Volt$, where $\Volt$ is a Gaussian-Volterra process which is $H$-locally nondeterministic and $\theta$ fulfills certain a priori estimates.
In this setting, the abstract version of Theorem~\ref{mainresult:regularity}\ref{en:rega} for $\theta+\Volt$ is given by Proposition~\ref{prop:E1E2}; its proof follows along the lines of Olivera and Tudor \cite{Oliv}, making use of a new conditional version of Romito's Lemma \cite{Romito}, see Lemma~\ref{lem:smoothinglemma}.
As for Theorem~\ref{mainresult:regularity}\ref{en:regb}, its counterpart for $\theta+\Volt$ is given by Corollary~\ref{cor:besov_regularity_general_2}; the proof is based on an alternative viewpoint via a duality argument. 
To do so most efficiently, we introduce a shifted version of the (deterministic) sewing lemma, see Lemma~\ref{lem:deterministic_shifted_sewing}. Both the core of the duality argument and this version of the sewing lemma are outlined in Section~\ref{sec:shiftedsewing}.

Eventually, to apply these abstract results to solutions to~\ref{eq:SDEintro}, we need a priori estimates of the drift part of the solution; these are presented in Section~\ref{sec:apriori} and are adaptations of similar estimates from \cite{GaleatiGerencser}, based on SSL techniques. 
By relying on them, even without explicitly using it, the stochastic sewing lemma is also crucial in our arguments.

\subsection{Gaussian bounds}

In the Brownian case, Gaussian bounds for the density of solutions to SDEs with smooth drifts and additive noise (or more generally, non-degenerate diffusion coefficients) are classical; see moreover \cite{HaoZhang2023,PerkowskiVZ,ZhangZhao} for extensions to singular and distributional drifts.

As for the fractional Brownian case, Besal\'{u} et al. \cite{tindel} proved upper and lower Gaussian bounds for fBm with Hurst parameter $H \in (0,1)$, for one-dimensional equations with additive noise and smooth drift, using the Nourdin-Viens formula \cite{NourdinViens} derived from Malliavin calculus. Similar results are available for multiplicative noise if a Doss-Sussman-type transformation is possible~\cite{tindel} or by considering the equation as a rough differential equation for $H>1/4$, see \cite{BNOT} for upper sub-Gaussian bounds. 

In the case of~\ref{eq:SDEintro} with irregular drift, Li et al.~\cite{LiPanloupSieber} obtained Gaussian upper and lower bounds for $\mathcal{L}(X_t)$ whenever $b\in\mathcal{C}^\gamma(\R^d)$ with $\gamma>\max\{0,1-1/(2H)\}$; when $H<1/2$, they can even allow for drifts $b$ which are merely measurable functions with at most linear growth.
See also~\cite{ButhenhoffSonmez} for similar results for mixed SDEs driven by correlated fBms.
The contribution of our paper for~\ref{eq:SDEintro} is to reconcile condition $\gamma>1-1/(2H)$ with the regime $H \in (0,1/2)$; i.e. replacing the condition on $b$ being locally bounded by $b$ being a distribution in a Besov space $\mathcal{B}^\gamma_\infty$ of negative regularity $\gamma$.
For the fBm-case $H\neq 1/2$, to our knowledge this is the first result on Gaussian bounds for genuine distributional drifts.

\begin{theorem} \label{mainresult:gaussianbound}
Let $T>0$, $H\in (0,1/2)$ and $b \in L^\infty([0,T];\mathcal{B}^\gamma_\infty)$ for $\gamma>1-1/(2H)$. There exists $C=C_T>0$ such that
 at any time $t\in (0,T]$, the law of the unique solution $X$ to \ref{eq:SDEintro}, starting from any deterministic initial condition $X_{0} = x_{0}\in \R^d$, has a density $p(t,\cdot)$ w.r.t. Lebesgue measure and 
\begin{equation}\label{eq:intro_main_gaussian_tails}
C^{-1} t^{-Hd} \exp\{-C t^{-2H}(x-x_0)^2\} \leqslant p(t,x)\leqslant C t^{-Hd} \exp\{-C^{-1}  t^{-2H}(x-x_0)^2\},\quad \forall x\in \R^d.
\end{equation}
\end{theorem}

The double-sided estimate~\eqref{eq:intro_main_gaussian_tails} is clearly sharp. 
The strategy of its proof is similar to~\cite{LiPanloupSieber}. First, using Girsanov's Theorem and conditioning on the endpoint of $B$, one finds an explicit expression of $p(t,x)$ in terms of a fractional Brownian bridge, see Proposition~\ref{prop:reformulation}.

However, following the rest of the argument from~\cite{LiPanloupSieber} leads to the following major obstacle. One needs to make sense and investigate an integral of a distribution $b$ along a fractional Brownian bridge $P$; i.e. treat a term of the form $\smallint_0^\cdot b_r(P_r) \dd r$. Such integrals are well understood via stochastic sewing when being evaluated along fBm due to its local nondeterminism property, giving a lower bound for $Var(B_r\vert \mathcal{F}_s)$ for $s<r$. However, a fractional Brownian bridge becomes ``more and more deterministic" as it approaches its terminal time $t$; hence, such a lower bound is expected to degenerate as $r$ approaches $t$. 
Section~\ref{sec:Gaussianbridges} is dedicated to quantifying this intuition, by proving a variant of local nondeterminism for a class of Gaussian-Volterra bridges with the fractional Brownian bridge as a special case thereof. In particular, this enables us to make sense of the
process $\smallint_0^\cdot b_r(P_r) \dd r$ and estimate its Cameron-Martin norm, 
which proves to be sufficient to conclude.

\subsection{McKean--Vlasov equations} \label{subsec:McKeanVlasov}

As an application of Theorem~\ref{mainresult:regularity}(b),
we establish new well-posedness results for McKean--Vlasov equations of convolutional type: 
\begin{equation}\label{eq:McKV}
\begin{cases}
{\displaystyle Y_{t} = Y_{0} + \int_{0}^t \mathfrak{b}_s\ast \mu_{s}(Y_{s})\, \dd s + B_{t}},\\
\mu_{t} = \law(Y_{t}),~t\geq 0,
\end{cases}
\end{equation}
where the one-time marginal law $\mu_t$ of the solution $Y$ itself enters the equation via a convolution with the drift $\mathfrak{b}$.
Such equations are often called McKean--Vlasov, or nonlinear SDEs, and they arise as mean field limits of systems of particles with pairwise interaction kernel given by $\mathfrak{b}_s$. See for instance the lecture notes of Sznitman~\cite{Sznitman} in the Brownian case, the work \cite{CoghiDeuschelFrizMaurelli} for general additive noise, \cite{galeati2024quantitative} for additive fractional noise with singular interaction kernel and \cite{BailleulCatellierDelarue,BailleulCatellierDelarue2} in the context of rough paths.

In the case of irregular $\fkb$, there are several existence and uniqueness results for \eqref{eq:McKV}. We refer to \cite{Pascucci_Rondelli_2025,pascucci2024existenceuniquenessresultsstrongly,PECDR_2020,Veretennikov2024,RocknerZhang} and references therein in the Brownian case; and to \cite{GHM,GaleatiGerencser,Han2025,HaoRocknerZhang} in the fBm case. Still, most available well-posedness results with fBm can only handle drifts $\mathfrak{b}$ of spatial regularity at least $1-1/(2H)$, see in particular \cite{galeati2024quantitative}. In terms of (strong) existence, we are able to lower this threshold to (almost) $1-1/H$.

We now present our first result on McKean--Vlasov equations driven by fBm, which states an existence criterion. 

\begin{theorem}\label{Mainresult-McKV}
Let $H\in (0,+\infty)\setminus \N$, $Y_{0}$ be an $\mathcal{F}_{0}$-measurable random variable and let
\begin{equation}\label{eq:intro_existence_MKV}
	\mathfrak{b} \in L^\infty([0,T];\mathcal{B}^\theta_\infty)\quad \text{ with } \theta>1-\frac{1}{H}.
\end{equation}
Then there exists a strong solution to the McKean--Vlasov equation \eqref{eq:McKV} starting from $Y_{0}$.
\end{theorem}

Condition \eqref{eq:intro_existence_MKV} in Theorem~\ref{Mainresult-McKV} identifies a \emph{subcritical regime} for the McKean--Vlasov equation \eqref{eq:McKV}, which can be derived by a scaling argument, see Remark~\ref{rk:scalingMcKV} for more details.
In view of this fact, we do not expect strong existence and uniqueness results for \eqref{eq:McKV} to hold for general supercritical $b\in L^\infty([0,T];\mathcal{B}^\theta_\infty)$ with $\theta<1-1/H$. 
In the Brownian case ($H=1/2$), it is noteworthy that the threshold found in the theorem above is $\theta>-1$. In terms of general Besov-valued drifts this matches the current literature, see Theorem 1 in \cite{Menozzietal} (where strong uniqueness is  also obtained).

The intuition why Theorem~\ref{mainresult:regularity} is useful to treat \eqref{eq:McKV} is the following, where we ignore integrability in time for sake of simplicity: 
Assume that we know a priori that $\mu_t$ is of positive regularity. Then $\mathfrak{b}_t \ast \mu_t$ is of higher regularity than $\mathfrak{b}_{t}$. 
In particular one can hope for looser conditions on $\mathfrak{b}$ than for $b$ in the linear \ref{eq:SDEintro}. The procedure is to take a smooth approximation of $\mathfrak{b}$, then
a priori estimates can be given using Theorem~\ref{mainresult:regularity} and the limit can be identified as a solution.  This kind of bootstrap argument has already been explored in the Markovian setting; see \cite{Han} for the case of a genuine function $\mathfrak{b}$, by using techniques from \cite{Romito} and Malliavin calculus, and \cite{Menozzietal} for distributional $\mathfrak{b}$ by investigating the corresponding Fokker-Planck equation.

\medskip

Under additional assumptions on $\mathfrak{b}$ and $\law(Y_0)$, we are also able to deduce a uniqueness result for \eqref{eq:McKV}.

\begin{theorem}\label{Mainresult-McKV-uniqueness}
Let $H\in (0,+\infty)\setminus \N$ and $\mathfrak{b}\in L^\infty([0,T];\cB^\theta_p)$ for some $\theta\in (-\infty,1)$, $p\in [1,\infty]$ satisfying
	\begin{equation}\label{eq:condition-uniqueness}
		\theta>1-\frac{1}{2H}, \quad \theta-\frac{d}{p}>1-\frac{1}{H}.
	\end{equation}
Further assume that $\law(Y_0)\in L^\infty_x$.
Then there exists a strong solution to the McKean--Vlasov equation \eqref{eq:McKV} starting from $Y_{0}$, such that
\begin{equation}\label{eq:condition_uniqueness2}
	\mathfrak{b}\ast \mu\in L^1([0,T];\cC^1_b);
\end{equation}
moreover, pathwise uniqueness and uniqueness in law hold in the class of solutions satisfying \eqref{eq:condition_uniqueness2}.
\end{theorem}

Condition \eqref{eq:condition-uniqueness} is a special subcase of \eqref{eq:intro_existence_MKV}, in view of the Besov embedding $\cB^\theta_p\hookrightarrow \cB^{\theta-d/p}_\infty$. In particular, whenever
\begin{align*}
	H \frac{d}{p}>\frac{1}{2},
\end{align*}
condition \eqref{eq:condition-uniqueness} reduces to $\theta-d/p>1-1/H$, which is again scaling subcritical.
 For $H\in (0,1)$, Theorem \ref{Mainresult-McKV-uniqueness} includes a relevant class of Coulomb and Riesz-type kernels, see Remark~\ref{rk:Coulomb-kernel} for details.

In order to achieve uniqueness, we currently need to additionally enforce $\law(Y_0)\in L^\infty_x$. In this case, thanks to the regularity condition \eqref{eq:condition_uniqueness2}, we can show that any $L^p_x$-regularity of $\mathcal{L}(Y_0)$ is propagated at positive times (cf. Lemma \ref{lem:propagation_integrability}); in turn, leveraging on this information, we can use functional inequalities and stability estimates to close a contraction argument in Wasserstein distances.
We do not know whether condition $\law(Y_0)\in L^\infty_x$ is merely technical, and whether the result can be extended to general $Y_0$; we leave this problem for the future.

\paragraph{Organization of the paper.} 
Section~\ref{sec:preliminaries} includes preliminaries and three results that are of use throughout the paper -- namely, an analysis of the behavior of \ref{eq:SDEintro} under rescaling of $b$, a priori estimates of the increments of its solution, and a shifted deterministic sewing lemma. 
Sections~\ref{sec:Besovregularity} and~\ref{sec:regularity.density} are dedicated to the proof of Theorem~\ref{mainresult:regularity}.
In the former, we give quantitative estimates at any fixed positive time; in the latter, we provide estimates measuring jointly integrability in time and regularity in space. 
In Section~\ref{sec:McKeanVlasov}, we exploit the previous results
to prove Theorems ~\ref{Mainresult-McKV} and~\ref{Mainresult-McKV-uniqueness}.
Section~\ref{sec:GaussianBounds} is devoted to the proof of the Gaussian bounds from Theorem~\ref{mainresult:gaussianbound}; to derive them, we first investigate the local nondeterminism property of Gaussian bridges in Section~\ref{sec:Gaussianbridges}.
In Appendices~\ref{app:besov} and~\ref{app:fBm} we state and prove some tailor-made results on Besov spaces and fBm that are needed throughout the paper. Appendix~\ref{app:sewing} proves the deterministic shifted sewing lemma (Lemma~\ref{lem:deterministic_shifted_sewing}).
Appendix~\ref{app:bridge} gives a representation of Gaussian bridges which is used in Section~\ref{sec:GaussianBounds}.

\begin{acknowledgement}
We warmly thank Khoa L\^e for fruitful discussions and insightful suggestions, especially concerning the use of a priori estimates from rough path theory, that allowed us to strongly improve the results of the paper.
\end{acknowledgement}

\begin{funding}
LA acknowledges the support of the Labex de Math\'ematique Hadamard and of the Austrian Science Fund (FWF) via program P34992. LG was supported by the SNSF Grant 182565 and by the Swiss State Secretariat for Education,
Research and Innovation (SERI) under contract number MB22.00034, and by the Istituto
Nazionale di Alta Matematica (INdAM) through the project GNAMPA 2025 ``Modelli stocastici
in Fluidodinamica e Turbolenza''. AR acknowledges the financial support of the French National Research Agency through the SDAIM project ANR-22-CE40-0015. This research was funded in whole or in part by the Austrian Science Fund (FWF) \href{https://www.fwf.ac.at/en/research-radar/10.55776/P34992}{10.55776/P34992}. For open access purposes, the authors have applied a CC BY public copyright license to any author-accepted manuscript version arising from this submission.
\end{funding}

\section{Preliminaries}\label{sec:preliminaries}

\subsection{Notations and definitions}\label{subsec:notdef}

\paragraph{Various notations.} 
Throughout the paper, we use the following notations and conventions:
\begin{compactitem}
	\item Constants $C$ might vary from line to line. 
	\item We write  \(U \lesssim V\) if there exists a constant \(C\) such that \( U \leqslant C V\).
\item Denote by \(\cC_{[0,T]}\) the space of continuous functions from $[0,T]$ to $\R^d$, and by $\cC_{b}^k$ the space of $k$ times continuously differentiable and bounded functions from $\R^d$ to $\R^d$. 

\item 
Let \(s < t\) be two real numbers and \(\Pi = (s = t_0 < t_1 < \cdots < t_n =t)\) be a 
partition of \([s,t]\), we denote $|\Pi| = \sup_{i=1,\cdots,n}(t_{i} - t_{i-1})$ the mesh 
of $\Pi$. 
\item For $\cS,\cT \in \mathbb{R}$ with $\cS\leqslant \cT$, we denote $\simp{\cS}{\cT}\coloneqq\{(s,t):\cS\leqslant s \leqslant t \leqslant \cT\}$ and $\simptwo{\cS}{\cT}\coloneqq\{(s,u,t):\cS\leqslant s \leq u \leqslant t \leqslant \cT\}$ .
\item For $\cS,\cT \in \mathbb{R}$ with $\cS\leqslant \cT$, we also define the restricted simplexes
\begin{equation}\label{eq:restricted_simplexes}\begin{split}
	\Dminus{\cS}{\cT} & \coloneqq\{(s,t) \in \simp{\cS}{\cT}:s-(t-s)\geqslant \cS\},\\
	\DDminus{\cS}{\cT} & \coloneqq\{(s,u,t) \in \simptwo{\cS}{\cT}:(s,t)\in \Dminus{\cS}{\cT}, (u-s) \wedge (t-u) \geq (t-s)/3\}.
\end{split}\end{equation}
\item For any function $f$ defined on $[\cS,\cT]$, we denote $f_{s,t}\coloneqq f_t-f_s$ for $(s,t) \in \simp{\cS}{\cT}$. 
\item For any function \(g\) defined on $\simp{\cS}{\cT}$ and 
$(s, u, t) \in \simptwo{\cS}{\cT}$, we denote $\delta 
g_{s,u,t}\coloneqq g_{s,t}-g_{s,u}-g_{u,t}$.

\item For a probability space $\Omega$ and $m \in [1,\infty]$, the norm on $L^m(\Omega)$ is denoted by $\|\cdot\|_{L^m}$ or $\|\cdot\|_{L^m_{\Omega}}$. The conditional $L^m$ norm $(\EE[|\cdot|^m\vert\mathcal{G}])^{1/m}$ w.r.t. a $\sigma$-algebra $\mathcal{G}$ is denoted by $\|\cdot\|_{L^m\vert \mathcal{G}}$.

\item Let $E$ be a Banach space. For $f \colon [\cS,\cT]\rightarrow E$, we denote $\|f\|_{L^q_{[\cS,\cT]}E}\coloneqq  \left(\int_{\cS}^{\cT} \|f_r\|_E^q\, \dd r\right)^{1/q}$ and denote the corresponding Lebesgue space as $L^q([\cS,\cT];E)$. In the case $[\cS,\cT] = [0,T]$, we also use the notation $\|f\|_{L^q_{T}E}$.
\item For $p\in[1,\infty]$, $L^p_{x}$ is used to denote $L^p(\R^d)$.
\item For a filtration $(\mathcal{F}_t)_{t\geqslant 0}$, the conditional expectation $\EE[\cdot\, \vert \mathcal{F}_t]$ is denoted by 
$\EE^t[\cdot]$.
\item For $q \in [1,\infty]$, let $L^q([\cS,\cT];\mathcal{C}^\infty_b) \coloneqq \bigcap_{k \in \mathbb{N}} L^q([\cS,\cT];\mathcal{C}^k_b)$.

\item For $p \in [1,\infty]$, $p^\prime \in [1,\infty]$ is such that $1/p+1/{p^\prime}=1$.

\item All filtrations are assumed to fulfill  the usual conditions. 
\end{compactitem}

\paragraph{Fractional Brownian motion.} We refer to \cite{NualartOuknine} and the references therein for the definition of fractional Brownian motion (fBm) and its basic properties.
It is well-known that, given a one-dimensional fBm $B$ with Hurst parameter $H\in (0,1)$, one can construct a Brownian motion 
$W$ on the same probability space, which generates the same filtration and such that 
\begin{align}\label{eq:fBmrepresentation}
B_{t} = \int_{0}^t K_{H}(t,s)\, \dd W_{s};
\end{align}
$K_{H}$ is a deterministic kernel, whose explicit formula in terms of special functions can be found in \cite[p. 106]{NualartOuknine}.
In the sequel, with a slight abuse of notation, $B$ will refer to a $d$-dimensional fBm, namely an $\R^d$-valued process with 
independent components distributed as one-dimensional fractional Brownian motions. 

For a given filtration $\mathbb{F} = (\mathcal{F}_{t})_{t\geq0}$, we will say that $B$ is an $\mathbb{F}$-fBm if $W$ is an $\mathbb{F}$-Brownian motion.
The following local nondeterminism (LND) property will be useful, see \cite[Lemma 7.1]{Pitt}:
\begin{align}\label{eq:fBm_cond_var}
\var(B_t \vert \mathcal{F}_s)=\int_s^t K_H(t,u)^2\, \dd u ~ \text{I}_{d}=C_H (t-s)^{2H}~ \text{I}_{d}.
\end{align}

The definition of fBm can be extended to $H>1$: by induction, if $B^H$ denotes an fBm defined for some $H\in \R_{+}^* \setminus \N$, one sets
$B^{H+1}_{t} = \int_{0}^t B^H_{s}\, \dd s$ to be an fBm of Hurst parameter $H+1$. Accordingly, the LND property \eqref{eq:fBm_cond_var} extends to $H>1$, see \cite{GaleatiGerencser}.

\paragraph{Gaussian-Volterra processes.}

Some of the results obtained in Sections~\ref{sec:Besovregularity}, \ref{sec:regularity.density} and \ref{sec:GaussianBounds} hold for a broader class of processes than the fractional Brownian motion presented in the previous paragraph. Namely, we consider Gaussian processes $\Volt$ with the following Volterra representation:
\begin{align} \label{eq:defGaussianVolterra}
\Volt_t=\int_0^t K(t,s) \, \dd W_s,
\end{align}
where $W$ is an $\mathbb{F}$-Brownian motion, and the Volterra kernel $K$ is
a function
\begin{align}\label{eq:defK}
\begin{array}{ccc}
K\colon [0,T]^2 & \rightarrow & \mathbb{R}\\
(t,s) & \mapsto & K(t,s)
\end{array}
\end{align}
such that 
\begin{equation}\label{eq:propertiesK}
\begin{cases}\int_{[0,T]} K(t,s)^2 \, \dd s<\infty, ~\text{for all}~ t\in [0,T], \\
K(t,s)= 0, ~\text{for}~ s>t.
\end{cases}
\end{equation}
For $H\in (0,+\infty)$, we will say that $\Volt$ is $H$-locally nondeterministic if there exists $c_{H}>0$ such that for any $s< t$,
\begin{align}\label{eq:LND}
\var(\Volt_t \vert \mathcal{F}_s)=\int_s^t K(t,u)^2\, \dd u ~ \text{I}_{d} \geq c_H (t-s)^{2H}~ \text{I}_{d}.
\end{align}

\paragraph{Besov spaces.}

For $s \in \mathbb{R}$ and $1\leqslant p,q \leqslant \infty$, we denote the nonhomogeneous Besov space with these parameters by $\mathcal{B}_{p,q}^s(\R^d)$, or  simply $\mathcal{B}_{p,q}^s$. We refer to \cite[Section 2.7]{BaDaCh} for the precise definition.
If $q=\infty$, we simply write $\mathcal{B}_p^s$ instead of $\mathcal{B}_{p,\infty}^s$. For various properties including important embeddings, see Appendix~\ref{app:besov}.

In addition to the Besov norms, the H\"older seminorms will be convenient to work with when $\gamma\in (0,1)$:
\begin{align}\label{eq:defHolder}
	\llbracket f \rrbracket_{\cC^\gamma} = \sup_{x\neq y} \frac{|f(x)-f(y)|}{|x-y|^\gamma}.
\end{align}
We denote by $\cC^\gamma_b$, or simply $\cC^\gamma$ whenever clear, the space of bounded H\"older functions with norm $\| f\|_{\cC^\gamma_b}=\sup_{x\in\R^d} |f(x)|+\llbracket f \rrbracket_{\cC^\gamma}$; as usual, the definition can be extended to $\cC^\gamma_b$ with $\gamma \in \mathbb{R}_{+}\setminus \mathbb{N}$ (note the slight difference with $\cC^k_b$ for $k\in\N$).
For $s \in \mathbb{R}_{+}\setminus \mathbb{N}$ and $p=q=\infty$, Besov and H\"older spaces coincide: $\cB^s_\infty=\cC^s_b$, see \cite[p.99]{BaDaCh}.

\paragraph{Gaussian semigroup.} 
For any $t>0$ and $x \in \mathbb{R}^d$, let $g_t(x)\coloneqq (2\pi 
t)^{-d/2} e^{-\tfrac{|x|^2}{2t}}$. For any tempered distribution $\phi$ on 
\(\mathbb{R}^d\),
denote
\begin{align} \label{Gaussiansemigroup}
    G_t\phi(x)\coloneqq g_t * \phi (x).
\end{align}

\begin{definition} \label{def:beta-}
Let $\gamma \in \mathbb{R}$ and $q \in [1,\infty]$. We say that $(f^n)_{n \in \mathbb{N}}$ converges to $f$ in $L^q([0,T];\mathcal{B}_{\infty}^{\gamma-})$ as $n \rightarrow \infty$ if $\sup_{n \in \mathbb{N}} \|f^n\|_{L^q_{[0,T]} \mathcal{B}_{\infty}^\gamma}<\infty$ and 
\begin{align} \label{eq:approximation}
\forall \gamma^\prime<\gamma, \quad    \lim_{n \rightarrow 
\infty}\|f^n-f\|_{L^q_{[0,T]} \mathcal{B}_{\infty}^{\gamma^\prime}}=0.
\end{align}
\end{definition}

For a given $f\in L^q([0,T];\mathcal{B}_{\infty}^{\gamma})$, one can always find smooth approximations $f^n\in L^q([0,T]; \cC^\infty_b)$ such that $f^n\to f$ in $L^q([0,T];\cB^{\gamma-}_\infty)$; for instance, one can take $f^{n}_{t}=G_{1/n} f(t)$, as guaranteed by Lemma \ref{A.3}\ref{A.3.2}-\ref{A.3.3} in Appendix \ref{app:besov}.

\paragraph{Control functions.} A continuous function $w: \simp{0}{T}\rightarrow [0,\infty)$ is a control function if for all $(s,u,t) \in \simptwo{0}{T}$,
\begin{align*}
w(s,u)+w(u,t)\leqslant w(s,t) \text{ and } w(s,s)=0.
\end{align*} 
For $q \in [1,\infty)$, $E$ a Banach space and $\phi \in L^q([0,T];E)$, the mapping $(s,t) \mapsto \|\phi\|^q_{L^q_{[s,t]} E}$ is a control function.
If $w_{1}$ and $w_{2}$ are control functions and $a,b>0$ are such that $a+b\geqslant 1$, then $w_{1}^a w_{2}^b$ is also a control function, see e.g. \cite[Exercise 1.9]{FrizVictoir}.

\paragraph{Solution concepts.} 
\begin{definition} \label{def:solution}
Let $T>0$, $\gamma \in \mathbb{R}$, $q \in [1,\infty]$, $b \in L^q([0,T];\mathcal{B}_{\infty}^\gamma)$. Let $B$ be an $\mathbb{F}$-fBm and $X_0$ be an $\mathcal{F}_{0}$-measurable random variable in $\mathbb{R}^d$.

\begin{itemize}
\item \emph{Admissible solution}: We call $(X_t)_{t \in [0,T]}$ an admissible solution to \ref{eq:SDEintro} if there exists a sequence $(b^n)_{n\in \N}$ in $L^q([0,T];\mathcal{C}^\infty_b)$ that converges to $b$ in $L^q([0,T];\mathcal{B}_{\infty}^{\gamma-})$ such that $X^n$ converges in $L^2(\Omega;\mathcal{C}_{[0,T]})$ to $X$, where $X^n$ denotes the classical unique strong solution of {\renewcommand\SDE{SDE($b^n$,$X_0$)}\ref{eq:SDEintro}\renewcommand\SDE{SDE($b$,$X_0$)}}.

\item \emph{Uniqueness}: we say that uniqueness holds among admissible solutions if whenever $(b^{i,n})_{n\in \N}$ in $L^q([0,T];\mathcal{C}^\infty_b)$, $i=1,2$, converge to $b$ in $L^q([0,T];\mathcal{B}_{\infty}^{\gamma-})$ and $X^{i,n}$ converges in $L^2(\Omega;\mathcal{C}_{[0,T]})$ to a limit $X^{i}$, then $X^1$ and $X^2$ are indistinguishable.
\end{itemize}
\end{definition}

\begin{remark}
Let us provide a few important comments on this notion of solution and what is already known from the literature.
\begin{itemize}
\item Under Assumption~\ref{ass:strong}, \ref{eq:SDEintro} is known to have a pathwise unique strong solution \cite[Theorem 5.6]{GaleatiGerencser} in the sense of nonlinear Young equations introduced in \cite{CatellierGubinelli}. Since we will not need nonlinear Young integrals in this paper, we introduced the above notion of solution by regularization of the drift, which turns out to produce the same solutions as via nonlinear Young integrals under Assumption~\ref{ass:strong}. Hence, as a consequence of the well-posedness and stability result from \cite{GaleatiGerencser}, we get existence and uniqueness in the sense of the previous definition when Assumption~\ref{ass:strong} is fulfilled. The previous discussion holds when the initial condition $X_{0}$ is deterministic, but extends without difficulty when $X_{0}$ is an $\mathcal{F}_0$-measurable random variable, see e.g. \cite[Remark 4.7]{GaleatiGerencser}.

\item From its definition, it is clear that any admissible solution is $\mathbb{F}$-adapted.

\item When $\gamma>0$, the drift is a genuine function and any admissible solution is a classical solution, in the sense that both sides in the equality \ref{eq:SDEintro} must be equal and defined without regularization.
\end{itemize}
\end{remark}

\subsection{Scaling} \label{sec:scaling}

We recall here the scaling argument from \cite{GaleatiGerencser}.
In multiple situations throughout the paper, this simplifies the problem of studying the density of $\mathcal{L}(X_t)$ by looking at a fixed time point, say $t=T$; the precise dependence in time of the estimates then follows by rescaling.

\begin{lemma} \label{lem:scaling}
Let $b$ satisfy Assumption~\ref{ass:strong} and let $X$ be the unique admissible solution to \ref{eq:SDEintro} driven by a fractional Brownian motion $B$.
Then for any $\lambda \in (0,1]$ the process $X^\lambda_t = \lambda^{-H} X_{\lambda t}$ solves the SDE with the fractional Brownian motion $B^\lambda = (\lambda^{-H}B_{\lambda t})_{t\geq0}$ and drift
	\begin{align*}
		b^\lambda_{s}(x) = \lambda^{1-H} b_{\lambda s}(\lambda^H x).
	\end{align*}
	Moreover,
		\begin{align}
	\label{eq:scalingdrift}
	\text{if}\ \gamma<0, \ \text{then}\ \quad \|b^{\lambda}\|_{L^q_{[0,T]}\mathcal{B}^\gamma_\infty}&\lesssim \|b\|_{L^q_{[0,\lambda T]} \mathcal{B}^\gamma_\infty};\\
	\label{eq:scalingdrift>0}
	\text{if}\ \gamma\in (0,1), \ \text{then}\ \quad \left(\int_0^T \llbracket b_r^{\lambda} \rrbracket_{\cC^\gamma}^q\, \dd r \right)^{\frac{1}{q}} 
	&\lesssim \left(\int_0^{\lambda T} \llbracket b_r \rrbracket_{\cC^\gamma}^q\, \dd r \right)^{\frac{1}{q}}.
\end{align}
\end{lemma}

\begin{remark}
For any $t \in (0,T]$, choosing $\lambda=t/T$ yields $X_t= \left(\frac{t}{T}\right)^{H} X^{t/T}_T$. We will see that $\law(X_t)$ is absolutely continuous w.r.t the Lebesgue measure, from which 
it follows that
\begin{equation}\label{eq:scaling}
\frac{\dd \law(X_t)}{\dd x}(x) 
		= \left(\frac{t}{T}\right)^{-Hd} \frac{\dd \law(X^{t/T}_T)}{\dd x}\Big( \Big(\frac{t}{T}\Big)^{-H} x\Big).
\end{equation}
\end{remark}

\begin{proof}[Proof of Lemma~\ref{lem:scaling}]
The first part of the statement follows by a simple change of variables and by self-similarity of fractional Brownian motion.
To see \eqref{eq:scalingdrift}, note that by scaling in Besov spaces (see Lemma~\ref{lem:besov_scaling}),
\begin{align*}
\|b^{\lambda}\|_{L^q_{[0,T]}\mathcal{B}^\gamma_\infty}\leqslant C \big(\lambda^{1-H-\frac{1}{q}}+\lambda^{1-H-\frac{1}{q}+\gamma H}\Big) \|b\|_{L^q_{[0,\lambda T]} \mathcal{B}^\gamma_\infty} .
\end{align*}
Under Assumption~\ref{ass:strong} when $\gamma<0$, it holds $1-H-\frac{1}{q}> 1-H-\frac{1}{q}+\gamma H> 0$, hence we deduce \eqref{eq:scalingdrift}. 
As for \eqref{eq:scalingdrift>0} when $\gamma>0$, it holds $\llbracket b_r^{\lambda} \rrbracket_{\cC^\gamma} \leq \lambda^{1-H+\gamma H} \llbracket b_{\lambda r} \rrbracket_{\cC^\gamma}$ and a change of variables yields the result.
\end{proof}

\begin{remark}\label{rk:scalingMcKV}
We present a heuristic scaling argument which identifies the subcritical regime for the drift $\mathfrak{b} \in L^q([0,T];\mathcal{B}^\theta_{\infty})$ in the nonlinear SDE \eqref{eq:McKV}, readapted from \cite{GaleatiGerencser} to the McKean-Vlasov setting. With similar notations to the previous lemma, consider $Y^\lambda_{t} = \lambda^{-H} Y_{\lambda t}$. Then by change of variables,
\begin{align*}
Y_{t}^\lambda &=  Y_{0}^\lambda +  \lambda^{-H} \int_{0}^{\lambda t} \mathfrak{b}_{s}\ast \mu_{s}(Y_{s}) \, \dd s + B^\lambda_{t}\\
&= Y_{0}^\lambda +  \lambda^{1-H} \int_{0}^{t} \mathfrak{b}_{\lambda s}\ast \mu_{\lambda s}(Y_{\lambda s}) \, \dd s + B^\lambda_{t} .
\end{align*}
Since $\mu_{\lambda s}$ is the law of $Y_{\lambda s}$, the same scaling as \eqref{eq:scaling} gives
\begin{align*}
\lambda^{1-H} \mathfrak{b}_{\lambda s}\ast \mu_{\lambda s}(Y_{\lambda s}) = \lambda^{1-H} \int_{\R^d} \mathfrak{b}_{\lambda s}(\lambda^H Y^\lambda_{s}- \lambda^H z) \, \mu^\lambda_{s}(\dd z),
\end{align*}
for $\mu^\lambda_{s}$ the law of $Y^\lambda_{s}$. Hence for $\mathfrak{b}^\lambda_{s}(x) \coloneqq \lambda^{1-H} \mathfrak{b}_{\lambda s}(\lambda^H x)$, $Y^\lambda$ solves \eqref{eq:McKV} with drift $\mathfrak{b}^\lambda$ and fractional Brownian motion $B^\lambda$.
The same scaling argument as in Lemma~\ref{lem:scaling} then shows that a control of the form $ \|\mathfrak{b}^\lambda\|_{L^q_{[0,T]}\mathcal{B}^\gamma_\infty}\lesssim \|\mathfrak{b}\|_{L^q_{[0,\lambda T]} \mathcal{B}^\gamma_\infty}$, uniformly over $\lambda\in (0,1]$, is possible if and only if
\begin{equation}\label{eq:critical+subcritical}
\theta \geq 1-\frac{1}{Hq'}.
\end{equation}
As in the linear case, this suggests the necessity of condition \eqref{eq:critical+subcritical} to guarantee well-posedness for the SDE (with possibly more restrictions in the case of the critical equality).
For $q=\infty$, the subcritical regime (i.e. the strict inequality in \eqref{eq:critical+subcritical}) coincides with the condition $\theta>1-1/H$ found in Theorem~\ref{Mainresult-McKV}.
\end{remark}

\subsection{Shifted sewing lemma} \label{sec:shiftedsewing}

The main idea used in Section~\ref{sec:regularity.density} to show regularity of the law $t\mapsto \mathcal{L}(X_t)$  is the following: by duality, it suffices to estimate
\begin{align*}
\|\mathcal{L}(X_\cdot)\|_{L^{2}_{[s,t]}\mathcal{B}^{1/(2H)}_{1,1}}
\lesssim \sup_{f} \Big| \int_s^t \langle f_r, \mathcal{L}(X_r) \rangle\, \dd r \Big|
= \sup_{f} \Big|\EE\int_s^t f_{r}(X_r)\, \dd r \Big|,
\end{align*}
where the supremum is taken over a set of ``nice" test functions $f$, which are uniformly bounded in the norm of the dual Besov space $L^{2}_{[s,t]}\mathcal{B}^{-1/(2H)}_{\infty,\infty}$.
Note that, for fixed $f$, the last term in the above equation can be regarded as a time increment of the deterministic function $t\mapsto \EE\int_0^t f_{r}(X_r)\, \dd r$. To bound it, we use a shifted version of the deterministic sewing lemma.

The idea of combining sewing arguments with ``shifting'' tricks has been first introduced in \cite{Gerencser} and further developed in  \cite{MatPer2024,GaleatiGerencser}. It is useful in the stochastic framework, in order to deal with singularities that naturally arise -- considering a shifted filtration $\mathcal{F}_{s-(t-s)}$ helps to avoid restrictions arising from integrability conditions when taking conditional expectation with respect to $\mathcal{F}_s$, see \cite[Remark 2.3]{Gerencser}.
Interestingly, it turns out that in a similar way shifting can be employed effectively in the deterministic setup. To the best of our knowledge, this version of the sewing lemma is new in the literature. Its proof can be found in Appendix~\ref{app:sewing}.

We use the notations $\Dminus{\cS}{\cT}$ and $\DDminus{\cS}{\cT}$ as defined in \eqref{eq:restricted_simplexes}.

\begin{lemma}[Shifted deterministic sewing]\label{lem:deterministic_shifted_sewing}
Let $E$ be a Banach space, $\cS\leq \cT$ and $A:\simp{\cS}{\cT}\to E$ be a deterministic function. Suppose that there exist controls $w_1$, $w_2$ and nonnegative parameters $\alpha_1$, $\alpha_2$, $\beta_1$, $\beta_2$ satisfying
\begin{align*}
	\beta_1>0, \quad \beta_2>0,\quad \alpha_2+\beta_2>1
\end{align*}
such that
\begin{align}
	\label{eq:deterministic_shifted_assumption_1} & \| A_{s,t} \|_E \leq w_1(s-(t-s),t)^{\alpha_1} (t-s)^{\beta_1}, \quad \forall\, (s,t)\in \Dminus{\cS}{\cT},\\
	\label{eq:deterministic_shifted_assumption_2} & \| \delta A_{s,u,t} \|_E \leq w_2(s-(t-s),t)^{\alpha_2} (t-s)^{\beta_2}, \quad \forall\, (s,u,t)\in \DDminus{\cS}{\cT}.
\end{align}
Further assume that there exists a continuous map $\mathcal{A}:[\cS,\cT]\to E$ such that, for any $t \in [\cS,\cT]$ and any sequence of partitions $\Pi_k=\{t_i^k\}_{i=0}^{N_k}$ of $[\cS,t]$ with mesh size going to zero, we have
\begin{align*}
    \mathcal{A}_t-\cA_\cS=\lim_{k\rightarrow \infty}\sum_{i=0}^{N_k}A_{t_i^k,t_{i+1}^k} \text{ in } E. 
\end{align*}
Define
\begin{align*}
	C_1 \coloneqq \frac{2^{\alpha_2}}{1-2^{\alpha_2+\beta_2-1}}, \quad
	C_2 \coloneqq \frac{2^{-\beta_1}}{1-2^{-\beta_1}}, \quad
	C_3 \coloneqq C_1 \, \frac{2^{-\beta_2}}{1-2^{-\beta_2}};
\end{align*}
then it holds
\begin{align}
\label{eq:deterministic_shifted_conclusion_1}
	& \| \cA_t-\cA_s-A_{s,t}\|_E \leq C_1 w_2(s-(t-s),t)^{\alpha_2} (t-s)^{\beta_2}, \quad \forall\, (s,t)\in \Dminus{\cS}{\cT}, \\
\label{eq:deterministic_shifted_conclusion_2}
	& \| \cA_t-\cA_s\|_E \leq C_2 w_1(s,t)^{\alpha_1} (t-s)^{\beta_1} + C_3 w_2 (s,t)^{\alpha_2} (t-s)^{\beta_2}, \quad \forall\, (s,t)\in \simp{\cS}{\cT}.
\end{align}
\end{lemma}

\subsection{A priori estimates on solutions to SDEs} \label{sec:apriori}

We present here the relevant a priori estimates on solutions that we will need throughout the paper. Due to their structural differences, we need to give separate statements in the cases where the drift $b$ has spatial regularity $\cB^\gamma_\infty$ with $\gamma>0$ or $\gamma<0$.
Without loss of generality we can exclude the case $\gamma=0$, since our condition on $\gamma$ from Assumption \ref{ass:strong} is a strict inequality, so that we can use the embedding $\cB^0_\infty \hookrightarrow \cB^{-\vep}_\infty$ for $\vep>0$ small enough.
It suffices to prove the estimates for drifts $b\in L^q([0,T];\mathcal{C}^\infty_b)$, as the result for distributional $b$ then follows by the definition of an admissible solution.

\begin{lemma}\label{lem:apriori-gamma<0}
	Let $m\in [1,\infty)$, $H\in (0,1/2)$, $(\gamma,q)$ with $\gamma<0$ satisfying Assumption~\ref{ass:strong} and define
	\begin{equation}\label{eq:eta_apriori_gamma<0}
	\zeta \coloneqq \frac{-\gamma H}{\gamma H + 1 -H}.
	\end{equation}
	Let $b\in L^q([0,T];\cB^\gamma_\infty)$, $X_0$ be $\cF_0$-measurable 
 and let $X$ denote the unique admissible solution to \ref{eq:SDEintro}. Set $\varphi\coloneqq X-B$.
	Then there exists a constant $C=C(H,d,\gamma,q,m)$ such that for all $(s,t) \in \simp{0}{T}$ it holds
	\begin{equation}\label{eq:apriori-gamma<0}
		\| \cond{m}{s}{\varphi_t-\varphi_s}\|_{L^\infty_\Omega}
		\leq C \Big[ \| b\|_{L^q_{[s,t]} \cB^\gamma_\infty} (t-s)^{\gamma H + \frac{1}{q'}} + \| b\|_{L^q_{[s,t]} \cB^\gamma_\infty}^{1+\zeta} (t-s)^{\frac{1}{q'}-\frac{\zeta}{q}}\Big].
	\end{equation}
\end{lemma}

\begin{remark}\label{rem:eta_apriori_gamma<0}
For $q\in (1,2]$, Assumption~\ref{ass:strong} reduces to $\gamma>1-1/(Hq')$ and implies that
\begin{align*}
	\zeta< q(1-H)-1\leqslant 1-2H<1, \quad \frac{1}{q'}-\frac{\zeta}{q} > \gamma H + \frac{1}{q'} > H.
\end{align*}
In the case $q>2$, we have $\gamma>1-1/(2H)$ and we can run the computation with $q$ replaced by $2$ to obtain the same conclusion.
Hence, for $|t-s|\leq 1$, also using the estimate $a^{1+\zeta}\leq a+a^2$, we may write the bound \eqref{eq:apriori-gamma<0} in a simpler way as
\begin{align}\label{eq:apriori-gamma<0simple}
	\| \cond{m}{s}{\varphi_t-\varphi_s}\|_{L^\infty_\Omega} \lesssim \| b\|_{L^q_{[s,t]} \cB^\gamma_\infty}(1 + \| b\|_{L^q_{[s,t]} \cB^\gamma_\infty}) (t-s)^{\gamma H + \frac{1}{q'}}.
\end{align}
\end{remark}

\begin{proof}[Proof of Lemma~\ref{lem:apriori-gamma<0}]
	The statement is a more refined version of \cite[Lemma 2.4]{GaleatiGerencser}, 
	based on keeping more carefully track of the constants. The statement therein is given for deterministic $X_0$, but the proof for $\cF_0$-measurable initial data proceeds identically.
	
	We start by considering the case $q\in (1,2]$, in which case Assumption~\ref{ass:strong} becomes $\gamma>1-1/(H q')$.
	Set $\tilde{w}_b(s,t)=\| b\|_{L^q_{[s,t]}\cB^\gamma_\infty}^q$ 
	and for $s<t$ fixed, define
	\begin{equation}\label{eq:def-seminorm-apriori}
		\llbracket \varphi \rrbracket_{[s,t]} \coloneqq \sup_{s\leqslant u<v\leqslant t} \frac{\| \cond{m}{u}{\varphi_v-\varphi_u}\|_{L^\infty_\Omega}}{\tilde{w}_b(u,v)^{\frac{1}{q}} (v-u)^{\gamma H + \frac{1}{q'}}}.
	\end{equation}
	Following the computations up to \cite[eq. (2.6)]{GaleatiGerencser}, one finds
	\begin{equation}\label{eq:apriori-proof-eq1}
		\llbracket \varphi \rrbracket_{[s,t]} \lesssim 1 + \tilde{w}_b(s,t)^{\frac{1}{q}} (t-s)^{\gamma H + \frac{1}{q'}-H} \llbracket \varphi \rrbracket_{[s,t]}.
	\end{equation}
	Define another control $w_\ast$ by $w_\ast(s,t)^{\gamma H +1-H} = \tilde{w}_b(s,t)^{1/q} (t-s)^{\gamma H + 1/q'-H}$. It follows from \eqref{eq:apriori-proof-eq1} that, for small enough $\Delta>0$ (depending only on $H,\,d,\,\gamma,\,q$), $\llbracket \varphi\rrbracket_{[s,t]}\lesssim 1$ whenever $w_\ast(s,t)\leq \Delta$; for such $(s,t)$, \eqref{eq:apriori-gamma<0} automatically holds.
	Consider now $(s,t)$ with $w_\ast(s,t)> \Delta$. Then we can find $n\geq 2$ and an increasing sequence $s=t_0<t_1<\dots<t_n=t$ such that $w_\ast(t_i,t_{i+1})=\Delta$ for all $i=0,\ldots, n-2$ and $w_\ast(t_{n-1},t)\leq \Delta$.
	Arguing as in the proof of \cite[Lemma 2.4]{GaleatiGerencser}, it then holds
	\begin{align*}
		\| \cond{m}{s}{\varphi_t-\varphi_s} \|_{L^\infty_\Omega}
		& \leq \sum_{i=1}^n \| \cond{m}{t_{i-1}}{\varphi_{t_i}-\varphi_{t_{i-1}}} \|_{L^\infty_\Omega}\\
		& \lesssim \sum_{i=1}^n \tilde{w}_b(t_{i-1},t_i)^{\frac{1}{q}} (t_i-t_{i-1})^{\gamma H + \frac{1}{q'}}
		\leq n^{-\gamma H}  \tilde{w}_b(s,t)^{\frac{1}{q}} (t-s)^{\gamma H + \frac{1}{q'}}.
	\end{align*}
	By the superadditivity of $w_\ast$, it holds $(n-1) \Delta \leq w_\ast(s,t)$. Hence, as $\gamma<0$ by assumption, we conclude that in this case
	\begin{align*}
		\| \cond{m}{s}{\varphi_t-\varphi_s} \|_{L^\infty_\Omega} \lesssim \tilde{w}_b(s,t)^{\frac{1}{q}} (t-s)^{\gamma H + \frac{1}{q'}} w_\ast(s,t)^{-\gamma H},
	\end{align*}
	which by the definitions of $\tilde{w}_b$ and $w_\ast$ again implies \eqref{eq:apriori-gamma<0}.
	
	Now assume $q>2$. Then $b\in L^2([0,T];\cB^\gamma_\infty)$ and estimate \eqref{eq:apriori-gamma<0} holds with $q$ replaced by $2$. Moreover, by H\"older's inequality $\| b\|_{L^2_{[s,t]} \cB^\gamma_\infty} \leq \| b\|_{L^q_{[s,t]} \cB^\gamma_\infty} (t-s)^{1/2-1/q}$. Therefore
	\begin{align*}
	\| \cond{m}{s}{\varphi_t-\varphi_s}\|_{L^\infty_\Omega}
	& \lesssim \| b\|_{L^2_{[s,t]} \cB^\gamma_\infty} (t-s)^{\gamma H + \frac{1}{2}} + \| b\|_{L^2_{[s,t]} \cB^\gamma_\infty}^{1+\zeta} (t-s)^{\frac{1}{2}-\frac{\zeta}{2}}\\
	& \lesssim \| b\|_{L^q_{[s,t]} \cB^\gamma_\infty} (t-s)^{\gamma H + 1 - \frac{1}{q}} + \| b\|_{L^q_{[s,t]} \cB^\gamma_\infty}^{1+\zeta} (t-s)^{1-(1+\zeta)\frac{1}{q}},
	\end{align*}
	which upon rearranging yields \eqref{eq:apriori-gamma<0} also in this case.
\end{proof}

\begin{remark}\label{rem:apriori-borderline-gamma=0}
	In the borderline case $\gamma=0$, up to replacing $\cB^\gamma_\infty$ with the slightly better $L^\infty_x$ (in the sense that $L^\infty_x \subsetneq \cB^0_\infty$),  it is straightforward by H\"older's inequality to show that
	\begin{equation}\label{eq:apriori-borderline-gamma=0}
		\| \cond{m}{s}{\varphi_t-\varphi_s}\|_{L^\infty_\Omega}
		\leq 2 \| b\|_{L^q_{[s,t]} L^\infty_x} (t-s)^{\frac{1}{q'}}.
	\end{equation}
	In this case the inequality holds true for any $q\in [1,\infty]$.
\end{remark}

\begin{lemma}\label{lem:apriori-gamma>0}
	Let $m\in [1,\infty)$, $H\in (0,+\infty)\setminus \N$, $(\gamma,q)$ with $\gamma \in (0,1)$ satisfy Assumption~\ref{ass:strong}
	and let $b\in L^q([0,T];\cB^\gamma_\infty)$. Let $X_0$ be $\cF_0$-measurable and $X$ denote the unique admissible solution to \ref{eq:SDEintro}. Recall the notation for the H\"older seminorm introduced in \eqref{eq:defHolder}. Set $\varphi\coloneqq X-B$ and define the control
	\begin{align}\label{eq:wb}
		w_b(s,t)\coloneqq \int_s^t \llbracket b_r \rrbracket_{\cC^\gamma}^q \, \dd r.
	\end{align}
Additionally define the parameters
	\begin{equation}
		\tilde \zeta \coloneqq \frac{\gamma}{1-\gamma}, \quad \tilde{\varepsilon}\coloneqq \gamma H+\frac{1}{q^\prime}-H>0.
	\end{equation}
	Then there exists a constant $C=C(H,d,\gamma,q,m)$ such that for all $(s,t) \in \simp{0}{T}$, it holds
	\begin{equation}\label{eq:apriori-gamma>0}
		\big\| \|\varphi_t-\EE^s\varphi_t\|_{L^m|\mathcal{F}_s} \big\|_{L^\infty_\Omega}
		\leq C w_b(s,t)^{\frac{1}{q}} (t-s)^{\gamma H + \frac{1}{q^\prime}}\Big(1 + w_b(s,t)^{\frac{\tilde\zeta}{q}} (t-s)^{\tilde{\zeta}\tilde{\varepsilon}}\Big).
	\end{equation}
\end{lemma}

\begin{proof}
	As before, we start with the case $q\in (1,2]$.
	For any $s<t$, define $\llbracket \varphi\rrbracket_{[s,t]}$ as follows:
	\begin{align*}
	\llbracket \varphi \rrbracket_{[s,t]} \coloneqq \sup_{s<u<v<t} \frac{\| \cond{m}{u}{\varphi_v- \EE^u \varphi_v}\|_{L^\infty_\Omega}}{w_b(u,v)^{\frac{1}{q}} (v-u)^{\gamma H + \frac{1}{q'}}}.
	\end{align*}
	By \cite[Lemma 2.1]{GaleatiGerencser}, we know that $\llbracket \varphi \rrbracket_{[s,t]}$ is a well-defined, finite quantity. As therein, using the random ODE solved by $\varphi$ and the properties of conditional expectation, for any $(u,v) \in \simp{s}{t}$ it holds
\begin{align*}
	\|\varphi_v-\EE^u \varphi_v\|_{L^m|\mathcal{F}_u}
	& \leq \Big\|\varphi_v- \varphi_u - \int_u^v b_r(\EE^u\varphi_r + \EE^u B_r) \, \dd r\Big\|_{L^m|\mathcal{F}_u} \\
	&\quad + \Big\|\EE^u \big[ \varphi_u -\varphi_{v} + \int_u^v b_r(\EE^u\varphi_r + \EE^u B_r)\, \dd r \big]\Big\|_{L^m|\mathcal{F}_u} \\
	& \leq 2 \Big\|\varphi_v- \varphi_u - \int_u^v b_r(\EE^u\varphi_r + \EE^u B_r)\, \dd r\Big\|_{L^m|\mathcal{F}_u}\\
	& \leq 2  \int_u^v \|b_r(\varphi_r+B_r) - b_r(\EE^u\varphi_r + \EE^u B_r)\|_{L^m|\mathcal{F}_u} \dd r\\
	& \lesssim \int_u^v \llbracket b_r\rrbracket_{\cC^\gamma} \big[ \|\varphi_r - \EE^u\varphi_r\|_{L^m|\mathcal{F}_u}^\gamma + \|B_r - \EE^u B_r\|_{L^m|\mathcal{F}_u}^\gamma \big] \dd r\\
	& \lesssim \Big( \int_u^v \llbracket b_r\rrbracket_{\cC^\gamma}^q\, \dd r\Big)^{\frac{1}{q}} \big( \llbracket \varphi \rrbracket_{[s,t]}^\gamma |u-v|^{\gamma (\gamma H +\frac{1}{q'})+\frac{1}{q^\prime}} w_b(u,v)^{\frac{\gamma}{q}} + |u-v|^{\gamma H+\frac{1}{q'}}\big),
\end{align*}
where we applied H\"older's inequality, finiteness of  $\llbracket \varphi \rrbracket_{[s,t]}$ and the LND property \eqref{eq:fBm_cond_var} in the last step. Recall the definition of $\tilde\vep$. Then dividing both sides by $|u-v|^{\gamma H+1/q'} w_b(u,v)^{1/q}$ and taking the supremum we arrive at
\begin{align*}
	\llbracket \varphi \rrbracket_{[s,t]}
	\lesssim \llbracket \varphi \rrbracket_{[s,t]}^\gamma |t-s|^{\tilde{\vep} \gamma} w_b(s,t)^{\frac{\gamma}{q}} + 1 .
\end{align*}
Applying Young's inequality, we deduce the existence of a constant $C$ (depending on $\gamma$ and the previous hidden constant) such that
\begin{align*}
	\llbracket \varphi \rrbracket_{[s,t]} \leq \frac{1}{2} \llbracket \varphi \rrbracket_{[s,t]} + C |t-s|^{ \tilde\vep \tilde \zeta} w_b(s,t)^{\frac{\tilde\zeta}{q}} + C.
\end{align*}
By the definitions of $\llbracket \varphi \rrbracket_{[s,t]}$ and $\tilde \vep$, this yields the desired \eqref{eq:apriori-gamma>0}.

In the case $q>2$, as before we can first apply estimate \eqref{eq:apriori-gamma>0} with $q$ replaced by $2$ and then invoke H\"older's inequality, so that
\begin{align*}
	\tilde w_b(s,t)^{\frac{1}{2}}
	\coloneqq \Big( \int_s^t \llbracket b_r \rrbracket_{\cC^\gamma}^2 \dd r\Big)^{\frac{1}{2}}
	\leq \Big( \int_s^t \llbracket b_r \rrbracket_{\cC^\gamma}^q \dd r\Big)^{\frac{1}{q}} (t-s)^{\frac{1}{2}-\frac{1}{q}}
	= w_b(s,t)^{\frac{1}{q}} (t-s)^{\frac{1}{2}-\frac{1}{q}}.
\end{align*}
Inserting this inequality into the estimate for $q$ replaced by $2$ again yields \eqref{eq:apriori-gamma>0}.
\end{proof}

\begin{remark}\label{rem:apriori-borderline-gamma=1}
In the borderline case $\gamma=1$, up to replacing $\cB^1_\infty$ with the slightly better $W^{1,\infty}_x$ (in the sense that $W^{1,\infty}_x \subsetneq \cB^1_\infty$), going through similar computations one finds
\begin{align*}
	\| \EE^s |\varphi_t-\EE^s \varphi_t| \|_{L^\infty_\Omega}
	\leq 2 \int_s^t \| \nabla b_r\|_{L^\infty_x} \Big( \| \EE^s |\varphi_r-\EE^s \varphi_r| \|_{L^\infty_\Omega} + c_H (r-s)^{H} \Big) \, \dd r
\end{align*}
which upon an application of Gr\"onwall's lemma yields
\begin{equation}\label{eq:apriori-borderline-gamma=1}
	\| \|\varphi_t-\EE^s\varphi_t\|_{L^1|\mathcal{F}_s}\|_{L^\infty_\Omega}
	\lesssim \exp\Big( 2 \int_s^t \| \nabla b_r\|_{L^\infty_x} \dd r\Big) \Big(\int_s^t \| \nabla b_r\|_{L^\infty_x} \dd r\Big) (t-s)^H.
\end{equation}
\end{remark}

\begin{remark}\label{rem:apriorisimple>0}
Although estimates \eqref{eq:apriori-gamma<0} and \eqref{eq:apriori-gamma>0} may appear convoluted at first sight, their importance lies in the constant $C$ not depending on the time horizon $T$ nor $\| b\|_{L^q_{[0,T]} \cB^\gamma_\infty}$. Moreover, in \eqref{eq:apriori-gamma>0} only the $\cC^\gamma$-seminorm of $b$ appears, which will be important in Sections~\ref{sec:Besovregularity} and~\ref{sec:GaussianBounds} when performing rescaling arguments.
Observe however that, whenever the full generality of \eqref{eq:apriori-gamma<0} is not needed, we may just write
\begin{equation}\label{eq:apriori-simpler}
	\| \|\varphi_t-\varphi_s\|_{L^m|\mathcal{F}_s} \|_{L^\infty_\Omega} \lesssim_{T, \| b\|_{L^q_{[0,T]} \cB^\gamma_\infty}}(t-s)^{\gamma H + \frac{1}{q'}}, \quad \forall\, (s,t) \in \simp{0}{T},
\end{equation}
and similarly for \eqref{eq:apriori-gamma>0}. Let us finally point out that \eqref{eq:apriori-gamma<0} is stronger than \eqref{eq:apriori-gamma>0}, since $\varphi_s\in \cF_s$ and by properties of conditional expectation it holds
\begin{align}\label{eq:compcondexp}
	\|\|\varphi_t-\EE^s \varphi_t\|_{L^m|\mathcal{F}_s}\|_{L^\infty_\Omega} \leq 2 \|\|\varphi_t-\varphi_s\|_{L^m|\mathcal{F}_s}\|_{L^\infty_\Omega}.
\end{align}
\end{remark}

\section{Regularity for fixed time of the conditional law} \label{sec:Besovregularity}

In the main theorem of this section, Theorem~\ref{thm:Besovregularity}, we prove that under Assumption~\ref{ass:strong}, the conditional law of the solution $X$ to \ref{eq:SDEintro} has a density which lies in a Besov space of positive regularity, and track its dependence in time.

\begin{theorem} \label{thm:Besovregularity}
Let $X_0$ be $\mathcal{F}_0$-measurable and let $b\in L^q \left([0,T];\mathcal{B}^\gamma_{\infty}\right)$ with $q$ and $\gamma$ satisfying Assumption~\ref{ass:strong}.
Let $X$ be the unique admissible solution to \ref{eq:SDEintro} and let 
\[0<\eta<\gamma-1+\frac{1}{Hq'}.\]
Then, for any $0\leq u< t \leq T$, the conditional law of $X_{t}$ given $\cF_u$ admits  $\PP$-a.s. a density w.r.t. the Lebesgue measure. Moreover there exists a constant $C=C(\|b\|_{L^q_{[0,T]}\mathcal{B}^\gamma_\infty},H,d,\gamma,q,\eta)$ such that
\begin{equation*}
 	\Big\lVert \lVert \law(X_t \vert \mathcal{F}_u)\|_{\cB^\eta_1} \Big\lVert_{L^\infty_{\Omega}} \leq C(1 + (t-u)^{-\eta H}).
\end{equation*}
In particular, an unconditional version of the previous estimate holds as well:
\begin{equation*}
	\lVert \law(X_t)\|_{\cB^\eta_1} \leqslant C(1 + t^{-\eta H}) , \quad \forall\, t\in (0,T].
\end{equation*}
\end{theorem}

\begin{remark}
The unconditional version of the previous result is consistent with the literature on drifts in Besov spaces with positive regularity: see \cite[Theorem 1]{Oliv} 
for a similar (not quantitative) result when $\gamma>0$ and $q=\infty$.
\end{remark}

\begin{remark}
The constant $C$ in Theorem \ref{thm:Besovregularity} 
depends on the time horizon $T$ only through $\| b\|_{L^q_{[0,T]} \cB^\gamma_\infty}$.
In fact, as the proof shows, it suffices to have uniform bounds on $\| b\|_{L^q_{[t,t+1]} \cB^\gamma_\infty}$, over $t\in [0,T-1)$, which allows to cover infinite time intervals as well. 
\end{remark}

\smallskip

The rest of this section is organised as follows: 
first in Lemma~\ref{lem:smoothinglemma}, we provide an extension of Romito's Lemma (see \cite[Lemma A.1]{Romito}) for conditional laws, which gives a sufficient condition on a random variable to imply that its law conditioned on some $\sigma$-algebra admits a density with positive Besov regularity.
Then we apply this lemma to Gaussian-Volterra processes with drift having conditional H\"older continuity, to deduce in Proposition~\ref{prop:E1E2} that the conditional law of such processes is absolutely continuous w.r.t. the Lebesgue measure. Finally, as the solution to \ref{eq:SDEintro} can be seen as a perturbed fBm, combining the a priori estimates from Section~\ref{sec:apriori}, Proposition~\ref{prop:E1E2} and a scaling argument, we prove Theorem~\ref{thm:Besovregularity}.

To formulate Lemma~\ref{lem:smoothinglemma}, we need the following definition.

\begin{definition}
Let $h \in \mathbb{R}^d$ and $f\colon \mathbb{R}^d \rightarrow \mathbb{R}$. We define $\Delta^1_h f\colon \mathbb{R}^d\rightarrow \mathbb{R}$ by 
\begin{equation*}
(\Delta^1_h f)(\cdot)\coloneqq f(\cdot+h)-f(\cdot).
\end{equation*}
We define iteratively $\Delta^n_h f\colon \mathbb{R}^d\rightarrow \mathbb{R}$ for $n \in \mathbb{N}^*$ by 
\begin{equation}\label{eq:higher.order.increments}
(\Delta^n_h f)(\cdot)\coloneqq(\Delta^1_h(\Delta^{n-1}_h f))(\cdot)=\sum_{j=0}^n (-1)^{n-j} \binom{n}{j} f(x+jh).
\end{equation}
\end{definition}

\begin{lemma}\label{lem:smoothinglemma}
Let $Z$ be an $\mathbb{R}^d$-valued random variable on $(\Omega,\mathcal{F},\PP)$ and $\mathcal{G}\subseteq \mathcal{F}$ a $\sigma$-algebra. Assume there exist parameters $s,\delta,C>0$, $h_{0}\in (0,1]$ and an integer $m\geqslant 1$ with $\delta<s<m$ such that, for any $\phi \in \mathcal{C}^\delta_b(\mathbb{R}^d)$ and $h \in \mathbb{R}^d$ with $|h|<h_{0}$, for any $A\in \mathcal{G}$,
\begin{equation}\label{eq:smoothing.lemma.assumption}
\Big|\EE\big[\Delta_h^m \phi(Z)\, \mathbbm{1}_{A} \big]\Big|\leq C |h|^s \|\phi\|_{\mathcal{C}^\delta_b}\, \PP(A).
\end{equation}
Then $\PP$-almost surely, the (regular version of) the conditional law of $Z$ given $\mathcal{G}$, denoted by ${\mathcal{L}(Z \vert \mathcal{G})}$, is absolutely continuous w.r.t. the Lebesgue measure and belongs to $\cB^{s-\delta}_1$. Moreover, there exists a constant $\tilde{C}>0$ independent of $Z$ and $\mathcal{G}$ such that 
\begin{equation}\label{eq:smoothing.lemma}
\Big\lVert \lVert \law(Z \vert \mathcal{G})\lVert_{\cB^{s-\delta}_1} \Big\lVert_{L^\infty_{\Omega}} \leq \tilde{C}(1+C+h_{0}^{-s+\delta}).
\end{equation}
\end{lemma}

\begin{proof}
The proof mostly follows the one from \cite[Lemma A.1]{Romito}, up to taking into account conditioning w.r.t. $\mathcal{G}$ and the smallness parameter $h_0$.
We will use throughout several properties of Besov spaces, mostly found in \cite[Appendix A]{Romito}; we refer to \cite{Triebel} and the references therein for a more complete account.

First notice that, for $|h|\geq h_0$, since $s>\delta$ we can perform the estimate
\begin{align*}
	\Big|\EE\big[\Delta_h^m \phi(Z)\, \mathbbm{1}_{A} \big]\Big|
	\leq \| \Delta_h^m \phi\|_{L^\infty_x} \, \PP(A)
	\lesssim_{m,\delta} |h|^\delta \| \phi\|_{\cC^\delta_b} \, \PP(A)
	\leq h_0^{-s+\delta}\, |h|^s \| \phi\|_{\cC^\delta_b} \, \PP(A).
\end{align*}
In particular, up to replacing the constant $C$ by (a multiple of) $C'\coloneqq C+h_0^{-s+\delta}$, in the following we may assume w.l.o.g. $h_0=1$.

Set $\mu=\mathcal{L}(Z \vert \mathcal{G})$ and let $\mu^\varepsilon=\chi^\varepsilon\ast \mu$, for $\{\chi^\eps\}_{\eps>0}$ standard symmetric mollifiers.
Since $\mu^\eps$ converges $\PP$-a.s. to $\mu$ in the sense of measures as $\eps\to 0$, by lower semicontinuity of the $\cB^{s-\delta}_1$-norm, it suffices to prove estimate \eqref{eq:smoothing.lemma} for $\mu$ replaced by $\mu^\eps$.
For $h$ and $\phi$ fixed, by \eqref{eq:smoothing.lemma.assumption} and the definition of conditional law it holds
\begin{align*}
	| \EE[\langle \Delta^m_h \phi, \mu^\varepsilon\rangle \mathbbm{1}_A]|
	& = | \EE[\langle \Delta^m_h \phi^\eps, \mu\rangle \mathbbm{1}_A]|
	= | \EE[ \Delta^m_h \phi^\eps(Z)\mathbbm{1}_A]|\\
	& \leq C |h|^s \| \phi^\eps\|_{\mathcal{C}^\delta_b} \mathbb{P}(A)
	\leq C |h|^s \| \phi\|_{\mathcal{C}^\delta_b} \mathbb{P}(A) ,
\end{align*}
where in the above we used the fact that convolution by $\chi^\eps$ is a self-adjoint operation which commutes with $\Delta^m_h$, and set $\phi^\eps=\phi\ast\chi^\eps$.
As the estimate works for any $A\in\mathcal{G}$ and $\langle \Delta^m_h \phi,\mu^\eps\rangle$ is $\mathcal{G}$-adapted, we conclude that
\begin{equation}\label{eq:smoothing.lemma.proof1}
	| \langle \Delta^m_h \phi, \mu^\varepsilon\rangle |
	= |\langle \phi, \Delta^m_{-h} \mu^\varepsilon\rangle| \leq C |h|^s \| \phi\|_{\cC^\delta_b}\quad \PP\text{-a.s.},
\end{equation}
where we used that $(\Delta^m_h)^\ast = \Delta^m_{-h}$.
The estimate is true for $|h|\leq h_0$ and $\phi\in\cC^\delta$ fixed; in particular, we can take $\phi=(1-\Delta)^{-\delta/2} \tilde \phi$ with $\tilde\phi\in L^\infty_x$ (we ask the reader not the confuse the classical Laplacian $\Delta$ with the discrete increments $\Delta^m_h$).
Since $\mu^\eps$ is smooth, we can now take a countable collection $\{h^n\}_n$ dense in $\{h:|h|< 1\}$, 
a countable collection $\{\tilde\phi^n\}_n$
 weak-$\ast$ dense in $L^\infty_x$ and argue by approximation to upgrade \eqref{eq:smoothing.lemma.proof1} to
\begin{align*}
	\sup_{\tilde\phi\in L^\infty_x\setminus\{0\}, 0<|h|<1} \frac{|\langle \Delta^m_h (1-\Delta)^{-\delta/2}\tilde\phi, \mu^\varepsilon\rangle|}{|h|^s \| \tilde\phi\|_{L^\infty_x}}
	= \sup_{\tilde\phi\in L^\infty_x\setminus\{0\}, 0<|h|<1} \frac{|\langle \tilde\phi, \Delta^m_{-h} (1-\Delta)^{-\delta/2} \mu^\varepsilon\rangle|}{|h|^s \| \tilde\phi\|_{L^\infty_x}}
	\lesssim C \quad \PP\text{-a.s.},
\end{align*}
where we used the fact that $\Delta^m_h$ and $(1-\Delta)^{-\delta/2}$ commute. By duality, we deduce that
\begin{equation}\label{eq:smoothing.lemma.proof2}
	\sup_{0<|h|<1} \frac{\| \Delta^m_{h} (1-\Delta)^{-\delta/2} \mu^\varepsilon\|_{L^1_x}}{|h|^s}
	\lesssim C \quad \PP\text{-a.s.} 
\end{equation}
Noting that
\begin{equation}\label{eq:smoothing.lemma.proof3}
	\| (1-\Delta)^{-\delta/2} \mu^\varepsilon\|_{L^1_x} \leq \|\mu^\eps\|_{L^1_x} \leq 1\quad \PP\text{-a.s.}
\end{equation}
since $\mu^\eps$ is a probability measure, we can combine bounds \eqref{eq:smoothing.lemma.proof2}-\eqref{eq:smoothing.lemma.proof3} with the characterization of Besov spaces via finite differences (cf. \cite[Sec. 2.5.12]{Triebel}) to conclude that $\PP$-a.s.,
\begin{align*}
	\| (1-\Delta)^{-\delta/2} \mu^\eps\|_{\cB^s_1} \lesssim  \| (1-\Delta)^{-\delta/2} \mu^\eps\|_{L^1_x} + \sup_{0<|h|<1} \frac{\| \Delta^m_{h} (1-\Delta)^{-\delta/2} \mu^\varepsilon\rangle\|_{L^1_x}}{|h|^s} \lesssim 1 + C.
\end{align*}
Since $(1-\Delta)^{-\delta/2}$ acts isometrically between $\cB^{s-\delta}_1$ and $\cB^s_1$, the desired $\PP$-a.s. estimate for $\| \mu^\eps\|_{\cB^s_1}$ (with constants independent of $\eps$) follows.
\end{proof}

The scheme of the proof of the following proposition is similar to \cite[Prop. 1]{Oliv}, replacing the use of Romito's lemma by Lemma~\ref{lem:smoothinglemma} to take into account the conditioning. 
Recall that Gaussian-Volterra processes are of the form \eqref{eq:defGaussianVolterra} with a kernel $K$ satisfying \eqref{eq:propertiesK}, and the local nondeterminism property is defined in \eqref{eq:LND}.
As the result below does not depend on a particular time horizon $T$, we work directly with processes and filtrations defined on the time interval $[0,+\infty)$. 

\begin{proposition}\label{prop:E1E2}
Let $X=\theta+\Volt$, where $\theta$ is an $\mathbb{F}$-adapted process and $\Volt$ is a Gaussian-Volterra process which is $H$-locally nondeterministic w.r.t. $\mathbb{F}$. Let $u\geq 0$ and consider a $\sigma$-algebra $\mathcal{G}$ such that $\mathcal{G} \subseteq \mathcal{F}_{u}$. 
Further assume that there exist $K>0$ and $\beta>H$ such that
\begin{equation}\label{eq:E1E2_hypothesis}
	 \Big\| \EE \big[|\theta_t - \EE^s\theta_t| \vert \mathcal{G} \big] \Big\|_{L^\infty_{\Omega}} \leq K |t-s|^\beta , \quad \forall\, s<t \text{ with } |t-s|\leq 1 \text{ and } s\geq u.
\end{equation}
Then for any $\eta<\beta/H-1$ and any $\varepsilon>0$ it holds that, for all  $t\in (u,+\infty)$,
\begin{equation}\label{eq:E1E2_conclusion}
	\Big\lVert \lVert \law(X_t \vert \mathcal{G})\lVert_{\cB^\eta_1} \Big\lVert_{L^\infty_{\Omega}} \lesssim (1+K) (1+(t-u)^{-\eta(H+\varepsilon)}) .
\end{equation}
\end{proposition}

\begin{proof}
For reasons that will become clear later, let us fix an integer $m$ large enough such that
\begin{align*}
	\eta+1<\frac{m\beta}{\beta+mH}<m, \quad
	\eta\Big( H + \frac{\beta}{m}\Big)<\eta H + \varepsilon.
\end{align*}
We plan to apply Lemma~\ref{lem:smoothinglemma} to $X_t$ with $\delta=1$ and $s=\eta+1$.
Fix $t>u$.
Let $h_{0} = 1\wedge (t-u)^{\frac{\beta+mH}{m}}$ and $|h|\leq h_{0}$.
Now let $\varepsilon = |h|^{\frac{m}{\beta+mH}}$, so that $\varepsilon<t$ and $u\leq t-\varepsilon$, and introduce the variable
\begin{equation*}
Y_t^{\varepsilon}\coloneqq \EE^{t-\varepsilon} [\theta_t]+\Volt_t.
\end{equation*}
For $A\in \mathcal{G}$ and \(\phi \in \mathcal{C}^1_b\), we consider the following quantity:
\begin{align}\label{eq:E1E2}
\EE[\Delta^m_h \phi(X_t) \mathbbm{1}_{A}]=\EE[\Delta^m_h\phi(Y_t^\varepsilon) \mathbbm{1}_{A}]+\big(\EE[\Delta^m_h\phi(X_t)\mathbbm{1}_{A}]-\EE[\Delta^m_h\phi(Y_t^\varepsilon)\mathbbm{1}_{A}]\big)\eqqcolon E_1 + E_2.
\end{align}
As will be shown in the following, $\varepsilon$ was chosen in an optimal way to minimize $E_1+E_2$.

\textbf{Bounding $E_1$:}
Note that
\begin{align*}
Y_t^\varepsilon
=\EE^{t-\varepsilon} [\theta_{t}+\Volt_t]+\Volt_t-\EE^{t-\varepsilon}\Volt_t.
\end{align*}
Denote by $g_{\sigma^2_{t-\varepsilon,t}}$ the Gaussian density with variance $\var(\Volt_t-\EE^{t-\varepsilon} \Volt_t) \eqqcolon \sigma_{t-\varepsilon,t}^2 I_d$.
By the tower property, and using both that $\Volt_t-\EE^{t-\varepsilon} \Volt_t$ is Gaussian and independent of $\mathcal{F}_{t-\varepsilon}$ (recall that $\Volt$ has the Wiener integral representation \eqref{eq:defGaussianVolterra} against an $\mathbb{F}$-Brownian motion $W$) and that $A\in \mathcal{F}_{u}\subseteq \mathcal{F}_{t-\varepsilon}$, we have that
\begin{align*}
|E_1|&=|\EE[\mathbbm{1}_{A}\, \EE^{t-\varepsilon}[\Delta_h^m\phi(Y_t^\varepsilon)]]|\\
&= \left|\EE \left[\mathbbm{1}_{A} \int_{\mathbb{R}^d} \Delta_h^m \phi(\EE^{t-\varepsilon} [\theta_{t}+\Volt_t]+x) \, g_{\sigma^2_{t-\varepsilon,t}}(x)\, \dd x\right]\right|\\
&= \left|\EE \left[\mathbbm{1}_{A} \int_{\mathbb{R}^d}  \phi(\EE^{t-\varepsilon} [\theta_{t}+\Volt_t]+x) \, \Delta_{-h}^m g_{\sigma^2_{t-\varepsilon,t}}(x)\, \dd x\right]\right|\\
&\leqslant\|\phi\|_{L^\infty_x} \|\Delta_{-h}^m g_{\sigma^2_{t-\varepsilon,t}}\|_{L^1_x} \PP(A).
\end{align*}
By the $H$-local nondeterminism of $\Volt$, there exists $C>0$ such that
$\sigma_{t-\varepsilon,t}^2\geq C \varepsilon^{2H}$.
Then, as in \cite{Romito} (see two lines after Equation (2.7)),
\begin{align} \label{eq:boundingE1}
|E_1|\leqslant \|\phi\|_{L^\infty_x} \|\Delta_{-h}^m g_{\sigma^2_{t-\varepsilon,t}}\|_{L^1_x}\PP(A)
\leqslant C\|\phi\|_{L^\infty_x} \left(\frac{|h|}{\varepsilon^H}\right)^m \PP(A).
\end{align}

\textbf{Bounding $E_2$:}
Using  $\phi\in\mathcal{C}^1(\mathbb{R}^d)$, we have
\begin{align*}
|E_2|\leqslant C \|\phi\|_{\mathcal{C}^{1}_b} \EE\Big[\mathbbm{1}_{A} |X_t-Y_t^{\varepsilon}| \Big]
&= C\|\phi\|_{\mathcal{C}^{1}_b}\EE[\mathbbm{1}_{A} |\theta_t-\EE^{t-\varepsilon}\theta_t|]\\
&= C\|\phi\|_{\mathcal{C}^{1}_b}\EE\big[\mathbbm{1}_{A} \EE \big[|\theta_t-\EE^{t-\varepsilon}\theta_t| \vert \mathcal{G}\big]\big].
\end{align*}
We now apply \eqref{eq:E1E2_hypothesis} to get
\begin{align*}
|E_2| \lesssim K \|\phi\|_{\mathcal{C}^1_b} \varepsilon^{\beta} \PP(A).
\end{align*}

Combining the above with \eqref{eq:boundingE1}, recalling that $\varepsilon=|h|^{m/(\beta+mH)}$, we have
\begin{align*}
|E_1|+|E_2|\lesssim \|\phi\|_{\mathcal{C}^1_b}(1+K) |h|^{\frac{\beta m}{\beta+mH}} \PP(A).
\end{align*}
By \eqref{eq:E1E2}, overall one finds
\begin{align*}
	| \EE[\Delta^m_h\phi(X_t) \mathbbm{1}_{A}] |
	\lesssim \| \phi\|_{\mathcal{C}^1_b} (1+K) |h|^{1+\eta} \PP(A)
\end{align*}
and Lemma~\ref{lem:smoothinglemma} gives the conclusion.
\end{proof}

We now proceed with the proof of Theorem~\ref{thm:Besovregularity}.

\begin{proof}[Proof of Theorem~\ref{thm:Besovregularity}]
As in the previous paragraphs, since the result does not depend on the time horizon $T$, it is more convenient to work with $u<t \in [0,+\infty)$ rather than on $[0,T]$. We can do this by simply extending $b$ to be identically $0$ on $(T,+\infty)$.

For a given $u\geq 0$, we first consider the case $t-u=1$. 
In view of the simplified bound \eqref{rem:eta_apriori_gamma<0} of Lemma~\ref{lem:apriori-gamma<0} (for $\gamma<0$) or the simplified bound \eqref{eq:apriori-simpler} of Lemma~\ref{lem:apriori-gamma>0} (for $\gamma>0$), $\theta \equiv X-B$ fulfills \eqref{eq:E1E2_hypothesis} and Proposition~\ref{prop:E1E2} yields
\begin{equation}\label{eq:Thm31-1stbound}
\Big\lVert \lVert \law(X_t \vert \mathcal{F}_{u})\|_{\cB^\eta_1} \Big\lVert_{L^\infty_{\Omega}} \leq C .
\end{equation}
In the case $t-u> 1$, by the tower property $\EE^u \law(X_t \vert \mathcal{F}_{t-1})=\law(X_t \vert \mathcal{F}_{u})$, so we can apply \eqref{eq:Thm31-1stbound} with $\tilde u=t-1$ and properties of conditional expectation to find 
\begin{align*}
	\Big\lVert \lVert \law(X_t \vert \mathcal{F}_{u})\|_{\cB^\eta_1} \Big\lVert_{L^\infty_{\Omega}}
	 \leq \Big\lVert \lVert \law(X_t \vert \mathcal{F}_{t-1})\|_{\cB^\eta_1} \Big\lVert_{L^\infty_{\Omega}}
	 \leq C .
\end{align*}

Although the same line of proof works for any $t-u\in (0,1)$ as well, it would lead to the presence of an extra parameter $\varepsilon>0$ as in \eqref{eq:E1E2_conclusion}.
This is not optimal in this situation, as we can get rid of it using the scaling property of fBm. Namely, for any $\lambda>0$, recall the notation $X^\lambda_t = \lambda^{-H} X_{\lambda t}$ from Section~\ref{sec:scaling} and define here $\mathcal{F}_{t}^\lambda \coloneqq \mathcal{F}_{\lambda t}$. The same argument that led to \eqref{eq:scaling} also gives
\begin{equation}\label{eq:Thm31-2ndbound}
\frac{\dd \law(X_t \vert \mathcal{F}_{u})}{\dd x}(x) 
		= \lambda^{-Hd} \frac{\dd \law\big(X^{\lambda}_{\lambda^{-1}t} \vert \mathcal{F}^\lambda_{\lambda^{-1}u}\big)}{\dd x}\big( \lambda^{-H} x\big).
\end{equation}
For $\lambda = t-u\in (0,1)$, applying \eqref{eq:Thm31-1stbound} to $\law(X^\lambda_{\lambda^{-1}t} \vert \mathcal{F}^\lambda_{\lambda^{-1} u})$ reads, as $\lambda^{-1}t - \lambda^{-1}u=1$,
\begin{equation*}
\Big\lVert \lVert \law(X^\lambda_{\lambda^{-1}t} \vert \mathcal{F}^\lambda_{\lambda^{-1} u})\|_{\cB^\eta_1} \Big\lVert_{L^\infty_{\Omega}} \leq C;
\end{equation*}
the constant $C$ above does not depend on $\lambda\in (0,1)$, thanks to the uniform estimates \eqref{eq:scalingdrift}-\eqref{eq:scalingdrift>0} on the rescaled drift $b^\lambda$.
Finally, it follows from \eqref{eq:Thm31-2ndbound} and Lemma~\ref{lem:besov_scaling} that
\begin{equation*}
\Big\lVert \lVert \law(X_{t} \vert \mathcal{F}_{u})\|_{\cB^\eta_1} \Big\lVert_{L^\infty_{\Omega}} \lesssim (1+ \lambda^{-\eta H}) \Big\lVert \lVert \law(X^\lambda_{\lambda^{-1}t} \vert \mathcal{F}^\lambda_{\lambda^{-1} u})\|_{\cB^\eta_1} \Big\lVert_{L^\infty_{\Omega}} \lesssim C (1+ (t-u)^{-\eta H}). \qedhere
\end{equation*}
\end{proof}

\section{Space-time regularity of the conditional law}\label{sec:regularity.density}

This section is dedicated to the proof of Theorem~\ref{thm:Besovregularity_better}, concerning time-space integrability/regularity of the conditional density of solutions to~\ref{eq:SDEintro}.
Similarly to Section~\ref{sec:Besovregularity}, the result follows from a more abstract result, valid for a class of perturbed locally nondeterministic Gaussian processes $\theta+\Volt$;
see Lemma \ref{lem:besov_regularity_general_1} and Corollary \ref{cor:besov_regularity_general_2} below.
Their proof is based on the deterministic sewing lemma with shifts introduced in Section~\ref{sec:shiftedsewing} and a duality argument.

\begin{theorem}\label{thm:Besovregularity_better}
Let $b\in L^q \left([0,T];\mathcal{B}^\gamma_{\infty}\right)$ with $q$ and $\gamma$ satisfying Assumption~\ref{ass:strong} and let $X_0$ be $\cF_0$-measurable.
Let $X$ be the unique admissible solution to \ref{eq:SDEintro}. 
Let $\tilde{q} \in (1,2]$ and $\alpha>0$ satisfy the condition
\begin{equation}\label{eq:cond_besov_regularity}
\alpha < \min\Big\{ \frac{1}{H} + \gamma -1, \frac{1}{\tilde q H} \Big\}.
\end{equation}
Then for any $u\in [0,T)$, $\PP$-a.s. $\law(X_\cdot|\cF_u)\in L^{\tilde{q}}([u,T];\mathcal{B}^\alpha_{1,1})$. More precisely, for any $(u,t)\in [0,T]^2_\leq$ with $|t-u|\leqslant 1$, we have the following $L^\infty_\Omega$-bounds:
\begin{itemize}
\item for $\gamma<0$,
\begin{equation}\label{eq:besovregularity1_conditional}
\begin{split}
\Big\| \lVert \mathcal{L}(X_\cdot \vert & \cF_u)\|_{L^{\tilde{q}}_{[u,t]} \mathcal{B}^\alpha_{1,1}} \Big\|_{L^\infty_\Omega}\\
& \lesssim (t-u)^{\frac{1}{\tilde{q}}-H \alpha}+\|b\|_{L^q_{[u,t]} \mathcal{B}^\gamma_\infty} \left(1+ \|b\|_{L^q_{[u,t]} \mathcal{B}^\gamma_\infty}\right)  (t-u)^{\frac{1}{\tilde{q}}-H \alpha + \frac{1}{q'}+ H (\gamma-1)};
\end{split}
\end{equation}

\item for $\gamma \in (0,1)$,
\begin{align}\label{eq:besovregularity2_conditional}
\hspace{-0.7cm}
\Big\| \lVert \law(X_\cdot\vert \cF_u)\|_{L^{\tilde{q}}_{[u,t]} \mathcal{B}^\alpha_{1,1}} \Big\|_{L^\infty_\Omega}
\lesssim (t-u)^{\frac{1}{\tilde{q}}-H\alpha}+w_b(u,t)^{\frac{1}{q}}\Big(1+ w_b(u,t)^{\frac{\eta}{q}}\Big) (t-u)^{\gamma H +\frac{1}{q^\prime}},
\end{align}
where we recall that $w_{b}(u,t) = \int_u^t \llbracket b_{r}\rrbracket^q_{\mathcal{C}^\gamma} \dd r$, $\tilde \eta = \frac{\gamma}{1-\gamma}$.
\end{itemize}
\end{theorem}

\begin{remark}
Note that one can give a more precise estimate than \eqref{eq:besovregularity1_conditional} by using Lemma~\ref{lem:apriori-gamma<0} instead of its simplifications in Remark~\ref{rem:eta_apriori_gamma<0}, similarly for \eqref{eq:besovregularity2_conditional} using Lemma~\ref{lem:apriori-gamma>0} instead of  Remark~\ref{rem:apriorisimple>0}. In doing so, we do not need to assume $|t-u|\leqslant 1$.
\end{remark}

Upon using the fact that $\law(X_t)=\EE[\law(X_t|\cF_u)]$, from the $\PP$-a.s. estimates \eqref{eq:besovregularity1_conditional}-\eqref{eq:besovregularity2_conditional} and Minskowski's inequality we immediately deduce the following deterministic ones.

\begin{corollary}\label{cor:Besovregularity_better}
	Let $b$, $X_0$, $X$, $\tilde q$ and $\alpha$ be as in Theorem \ref{thm:Besovregularity_better}. Then for any $(u,t)\in [0,T]^2_\leq$ with $|t-u|\leqslant 1$, we have the following bounds:
\begin{itemize}
\item for $\gamma<0$,
\begin{equation}\label{eq:besovregularity1}
	\|\mathcal{L}(X_\cdot)\|_{L^{\tilde{q}}_{[u,t]} \mathcal{B}^\alpha_{1,1}} 
\lesssim (t-u)^{\frac{1}{\tilde{q}}-H \alpha}+\|b\|_{L^q_{[u,t]} \mathcal{B}^\gamma_\infty} \left(1+ \|b\|_{L^q_{[u,t]} \mathcal{B}^\gamma_\infty}\right)  (t-u)^{\frac{1}{\tilde{q}}-H \alpha + \frac{1}{q'}+ H (\gamma-1)};
\end{equation}

\item for $\gamma \in (0,1)$,
\begin{align}\label{eq:besovregularity2}
\hspace{-0.7cm}
\|\law(X_\cdot)\|_{L^{\tilde{q}}_{[u,t]} \mathcal{B}^\alpha_{1,1}}
\lesssim (t-u)^{\frac{1}{\tilde{q}}-H\alpha}+w_b(u,t)^{\frac{1}{q}}\Big(1+ w_b(u,t)^{\frac{\eta}{q}}\Big) (t-u)^{\gamma H +\frac{1}{q^\prime}}.
\end{align}
\end{itemize}
\end{corollary}

\begin{remark} \label{rem:regularity}
Let us point out two important cases of the regularity obtained in Corollary~\ref{cor:Besovregularity_better}:
\begin{enumerate}[label=(\alph*)]
\item \label{en:firstcase}Let $\varepsilon_1>0$ and $b \in L^{1+\varepsilon_1}([0,T];\mathcal{C}_b^1)$. Then for any $\varepsilon_2>0$,
$
\law(X_\cdot) 
\in L^{1}([0,T];\mathcal{B}^{\frac{1}{H}-\varepsilon_2}_{1,1})$.
\item \label{en:secondcase} Let Assumption~\ref{ass:strong} hold. Then, for any $\varepsilon>0$, $
\law(X_\cdot) 
\in L^2([0,T];\mathcal{B}^{\frac{1}{2H}-\varepsilon}_{1,1})$.
\end{enumerate}
Moreover, for $\tilde q>2$, one can obtain further bounds by interpolating between the trivial 
$\law(X_\cdot) \in L^\infty([0,T];\mathcal{B}^{0-}_{1,1})$ 
and the bound for $\tilde q=2$. For instance, in regime \ref{en:secondcase} we deduce that
\begin{align*}
	\law(X_\cdot)
	 \in L^{\tilde q}([0,T];\mathcal{B}^\alpha_{1,1})\quad \text{ for any }\tilde q\in [2,\infty] \text{  and  }\alpha<\frac{1}{\tilde q H}.
\end{align*}
\end{remark}

\begin{remark}[Optimality]\label{rem:optimality_sec4}
Corollary~\ref{cor:Besovregularity_better} (and consequently Theorem \ref{thm:Besovregularity_better}) seems unimprovable in several situations. 
For instance, when $x_0=0$, $b\equiv 0$, we have $X_t=B_t$ and so we know that 
$\law(X_t)= P_{t^{2H}} = t^{-Hd} P_1(t^{-H}\cdot)$; therefore for small $t$ by scaling it holds $\lVert \law(X_t)\|_{\mathcal{B}^\alpha_{1,1}} \sim t^{-\alpha H}$. 
It follows that $(\law(B_t))_{t\in[0,T]}\in L^{\tilde q}([0,T]; \mathcal{B}^{\alpha}_{1,1})$ only if $\alpha<1/(\tilde q H)$, implying necessity of the second condition
in \eqref{eq:cond_besov_regularity}.
Moreover if we take $\tilde q=q$ with $q\in(1,2]$, then Assumption~\ref{ass:strong} implies that
\begin{align*}
	\min\Big\{ \frac{1}{H} + \gamma -1, \frac{1}{\tilde q H} \Big\} = \frac{1}{\tilde q H};
\end{align*}
the same holds whenever $\gamma\sim 1$, since $\tilde q>1$.
Hence, in both regimes \ref{en:firstcase} and \ref{en:secondcase} in Remark~\ref{rem:regularity}, we recover the best case scenario, namely the same result for $X$ as for fBm itself.
\end{remark}

The rest of this section is dedicated to the proof of Theorem~\ref{thm:Besovregularity_better}.
As mentioned, we do so by considering more generaly processes of the form $X=\mathfrak{G} +\theta$; then we can deduce Theorem~\ref{thm:Besovregularity_better} by employing the a priori estimates for $\theta=\varphi$ from Section~\ref{sec:apriori}.

\begin{remark}
The abstract results of this section can be useful in other contexts; for example, in the case of fBm-driven SDEs for which weak existence of solutions $X$ can be shown, with $X-B$ having some regularity, even though uniqueness is open. See \cite[Thm. 2.5]{AnRiTa}, \cite[Thm. 8.2]{GaleatiGerencser} and \cite{ButLeMyt,ButGal2023} for such situations. In particular, one can still deduce some Besov regularity for the conditional laws $t\mapsto \mathcal{L}(X_t\vert \cF_s)$.
\end{remark}

\begin{remark}[The borderline case $\gamma=1$]\label{rem:lipschitz_case}
	Let $X$ solves \ref{eq:SDEintro} with Lipschitz drift, in the sense that $\nabla b\in L^1([0,T];L^\infty_x)$; by estimate \eqref{eq:apriori-borderline-gamma=1}, the assumptions of Lemma \ref{lem:besov_regularity_general_1} and Corollary~\ref{cor:besov_regularity_general_2} below are satisfied with $\kappa_1=1$, $\kappa_2=H$. Therefore we can conclude that
\[\law(X_\cdot\vert \cF_u)\in L^\infty_\Omega L^{1}_{[u,T]} \mathcal{B}^\alpha_{1,1}\quad \text{ for any } \alpha<\frac{1}{H}; \]	
to the best of our knowledge, even this result is novel in the fBm literature. Note that here $b$ is allowed to grow linearly at infinity as we only require boundedness of $\nabla b$.
\end{remark}

\begin{lemma}\label{lem:besov_regularity_general_1}
Let $X=\theta+\mathfrak{G}$, where $\theta$ is an $\mathbb{F}$-adapted process and $\mathfrak{G}$ is a Gaussian-Volterra process which is $H$-locally nondeterministic w.r.t. $\mathbb{F}$. Let $u\geq 0$ and consider a $\sigma$-algebra $\mathcal{G}\subset \cF_u$.
Further assume that
\begin{equation}\label{eq:besov_regularity_condition_theta}
	\Big\| \EE[|\theta_t-\EE_s \theta_t| \big\vert \mathcal{G}] \Big\|_{L^\infty_\Omega} \leq w_	\theta(s,t)^{\kappa_1} (t-s)^{\kappa_2},\quad \forall \,s<t\text{ with } 
	s\geq u, 
\end{equation}
for some parameters ${\kappa_1}\geq 0$, ${\kappa_2}>H$, and some control $w_\theta$. Let $\tilde q\in [1,\infty]$ and $\tilde \gamma<0$ be parameters satisfying
\begin{equation}\label{eq:besov_regularity_general_condition}
	\tilde\gamma > - \min\Big\{ \frac{{\kappa_1}+{\kappa_2}}{H}-1, \frac{1}{\tilde q H} \Big\}.
\end{equation}
Then there exists $C>0$ such that for any $f\in L^{{\tilde{q}^\prime}}_{[0,T]}\mathcal{C}^\infty_b$ and any $(s,t)\in\simp{u}{T}$ it holds
\begin{equation}\label{eq:besov_regularity_law_general}
	\Big\|\EE\Big[\int_s^t f_r(X_r) \dd r\Big\vert \mathcal{G}\Big]\Big\|_{L^\infty_\Omega}
	\leq C \|f\|_{L^{{\tilde{q}^\prime}}_{[s,t]}\mathcal{B}^{\tilde{\gamma}}_\infty}(t-s)^{\frac{1}{\tilde q}+\tilde{\gamma}H} \Big(1+ w_\theta(s,t)^{\kappa_1} (t-s)^{\kappa_2-H}\Big).
\end{equation}
\end{lemma}

\begin{proof}
	By properties of conditional expectation, in order to verify \eqref{eq:besov_regularity_law_general}, it suffices to show that for any $\Gamma\in \mathcal{G}$ it holds
	\begin{equation}\label{eq:besov_regularity_general_proof1}
		\EE\Big[\int_s^t f_r(X_r) \dd r \, \mathbbm{1}_\Gamma \Big]
		\leq C \|f\|_{L^{{\tilde{q}^\prime}}_{[s,t]}\mathcal{B}^{\tilde{\gamma}}_\infty}(t-s)^{\frac{1}{\tilde q}+\tilde{\gamma}H} \Big(1+ w_\theta(s,t)^{\kappa_1} (t-s)^{\kappa_2-H}\Big) \PP(\Gamma)
	\end{equation}
	for a constant $C$ which does not depend on $\Gamma$. To verify \eqref{eq:besov_regularity_general_proof1}, we will apply Lemma~\ref{lem:deterministic_shifted_sewing}; to this end, for any $(s,t) \in \Dminus{u}{T}$, define
	\begin{align*}
	A_{s,t} = \EE\left[\int_s^t f_r(\mathfrak{G}_r + \EE^{s-(t-s)} \theta_r) \dd r\, \mathbbm{1}_\Gamma \right], \quad
	\cA_t-\cA_s = \EE\left[\int_s^t f_r(\mathfrak{G}_r+\theta_r) \dd r \mathbbm{1}_\Gamma\right].
	\end{align*}
	Since $f$ is smooth and $\theta$ satisfies \eqref{eq:besov_regularity_condition_theta}, it is easy to see that $\cA$ is the sewing of $A$. Hence to show \eqref{eq:besov_regularity_general_proof1} it remains to verify that $A$ satisfies \eqref{eq:deterministic_shifted_assumption_1} and \eqref{eq:deterministic_shifted_assumption_2}.
	For any $(s,t) \in \Dminus{u}{T}$, using $\EE=\EE \EE^s$ and $\mathcal{G}\subset\cF_s$, it holds
	\begin{align*}
		A_{s,t}
		= \EE\left[ \int_s^t  \EE^s f_r(\mathfrak{G}_r + \EE^{s-(t-s)} \theta_r) \dd r\, \mathbbm{1}_\Gamma \right]
		&= \EE\left[ \int_s^t  G_{\sigma^2_{s,r}} f_r(\EE^s \mathfrak{G}_r + \EE^{s-(t-s)} \theta_r) \dd r \, \mathbbm{1}_\Gamma\right],
	\end{align*}
	where we recall that $G$ denotes the Gaussian semigroup and $\sigma^2_{s,r} = \var(\mathfrak{G}_r-\EE^s \mathfrak{G}_r)\geq c_H (r-s)^{2H}~ \text{I}_{d}$ (see \eqref{eq:LND}).
	Thus the heat kernel estimate Lemma~\ref{A.3}\ref{A.3.4} with $\beta=0$ yields
	\begin{align*}
		|A_{s,t}|
		&\lesssim \PP(\Gamma)\, \int_s^t \|G_{\sigma^2_{s,r}} f_r\|_{L^\infty_x} \dd r\\
		& \lesssim \PP(\Gamma)\, \int_s^t \|f_r\|_{\mathcal{B}^{\tilde{\gamma}}_\infty} (r-s)^{\tilde{\gamma}H} \dd r
		\lesssim \PP(\Gamma)\, \| f\|_{L^{\tilde q'}_{[s,t]} \cB^{\tilde \gamma}_\infty} (t-s)^{\tilde\gamma H + \frac{1}{\tilde{q}}},
	\end{align*}
using that $\tilde{\gamma}>-1/(\tilde{q}H)$; hence \eqref{eq:deterministic_shifted_assumption_1} holds.

To check \eqref{eq:deterministic_shifted_assumption_2}, consider $(s,v,t)\in \DDminus{u}{T}$. We have
\begin{align*}
	\delta A_{s,v,t}
	& = \EE\Big[\int_s^v \big[ f_r(\EE^{s-(v-s)} \theta_r+\mathfrak{G}_r) - f_r(\EE^{s-(t-s)}\theta_r+\mathfrak{G}_r)\big] \mathbbm{1}_\Gamma\, \dd r \Big]\\
	&\quad +\EE\Big[\int_v^t \big[ f_r(\EE^{v-(t-v)}\theta_r+\mathfrak{G}_r)-f_r(\EE^{s-(t-s)}\theta_r+\mathfrak{G}_r) \big] \mathbbm{1}_\Gamma\, \dd r \Big]\\
&\eqqcolon  I_1+I_2.
\end{align*}
We only provide estimates for $I_1$, $I_2$ being similar. By using $\EE=\EE\,\EE^{s-(v-s)}$, $\mathcal{G}\subset \cF_{u}\subset \cF_{s-(v-s)}$ and the fact that $\mathfrak{G}$ is a Gaussian-Volterra process w.r.t. $\mathbb{F}$, one finds
\begin{align*}
	|I_1|
	& = \Big| \EE\Big[\int_s^v \big[ G_{\sigma^2_{s-(v-s),r}} f_{r}(\EE^{s-(v-s)} \theta_r+\EE^{s-(v-s)} \mathfrak{G}_r) - G_{\sigma^2_{s-(v-s),r}} f_r(\EE^{s-(t-s)}\theta_r+\EE^{s-(v-s)}\mathfrak{G}_r)\big] \mathbbm{1}_\Gamma\, \dd r \Big]\Big|\\
	& \leq \EE\Big[\int_s^v \| G_{\sigma^2_{s-(v-s),r}} f_{r}\|_{\cC^1_b} |\EE^{s-(v-s)} \theta_r-\EE^{s-(t-s)}\theta_r|\, \mathbbm{1}_\Gamma\, \dd r \Big]\\
	& \leq \PP(\Gamma) \int_s^v \| G_{\sigma^2_{s-(v-s),r}} f_{r}\|_{\cC^1_b} \Big\| \EE\big[ |\EE^{s-(v-s)} \theta_r-\EE^{s-(t-s)}\theta_r|\big| \mathcal{G}\big]  \Big\|_{L^\infty_\Omega}\, \dd r\\
	& \leq \PP(\Gamma) \int_s^v \| G_{\sigma^2_{s-(v-s),r}} f_{r}\|_{\cC^1_b} \Big( \big\| \EE\big[ |\EE^{s-(v-s)} \theta_r-\theta_r|\big| \mathcal{G}\big]  \big\|_{L^\infty_\Omega} + \big\| \EE\big[ |\theta_r-\EE^{s-(t-s)}\theta_r|\big| \mathcal{G}\big]  \big\|_{L^\infty_\Omega}\Big) \dd r\\
	& \leq \PP(\Gamma)\, w_\theta(s-(t-s),t)^{\kappa_1}\,(t-s)^{\kappa_2} \int_s^v \| G_{\sigma^2_{s-(v-s),r}} f_{r}\|_{\cC^1_{b}} \dd r,
\end{align*}
	where in the last steps we used the triangular inequality and Assumption \eqref{eq:besov_regularity_condition_theta}.
	By the local nondeterminism property \eqref{eq:LND} of $\mathfrak{G}$, the heat kernel estimate Lemma~\ref{A.3}\ref{A.3.4}, Assumption \eqref{eq:besov_regularity_condition_theta} and H\"older's inequality, it holds
\begin{align*}
	\int_s^v \| G_{\sigma^2_{s-(v-s),r}} f_{r}\|_{\cC^1_b} \dd r
	& \lesssim \int_s^v (r-s+(v-s))^{(\tilde \gamma-1) H} \|f_r\|_{\mathcal{B}^{\tilde{\gamma}}_{\infty}}\, \dd r\\
	& \leq (v-s)^{(\tilde \gamma-1) H} \int_s^v \|f_r\|_{\mathcal{B}^{\tilde{\gamma}}_{\infty}}\, \dd r\\
	& \leq (v-s)^{(\tilde \gamma-1)H + \frac{1}{\tilde{q}}} \| f\|_{L^{\tilde q'}_{[s,v]} \cB^{\tilde\gamma}_\infty}\\
	& \lesssim (t-s)^{(\tilde \gamma-1)H + \frac{1}{\tilde{q}}} \| f\|_{L^{\tilde q'}_{[s,t]} \cB^{\tilde\gamma}_\infty} ,
\end{align*}
where we used the fact that $v-s\sim t-s$ since $(s,v,t)\in\DDminus{u}{T}$.
Combining the above estimates with similar ones for $I_2$, one overall finds
\begin{align*}
	|\delta A_{s,v,t}|\lesssim \PP(\Gamma)\,\| f\|_{L^{\tilde q'}_{[s,t]} \cB^{\tilde\gamma}_\infty}\, w_\theta(s-(t-s),t)^{\kappa_1}\,(t-s)^{\kappa_2-H + \tilde \gamma H + \frac{1}{\tilde{q}}}.
\end{align*}
Notice that $\| f\|_{L^{\tilde q'}_{[s,t]} \cB^{\tilde\gamma}_\infty}$ is itself a control to the power $1/q'$, and by Assumption \eqref{eq:besov_regularity_general_condition} we have
\begin{align*}
	& \kappa_2-H+\tilde\gamma H+\frac{1}{\tilde q}\geq \tilde\gamma H+\frac{1}{\tilde q} >0,\\
	& \kappa_1 +\frac{1}{\tilde q'} + \kappa_2-H + \tilde\gamma H+\frac{1}{\tilde q}
	= 1 + \kappa_1 + \kappa_2 + \tilde\gamma H - H > 1.
\end{align*}

	Therefore condition \eqref{eq:deterministic_shifted_assumption_2} is satisfied and we can apply Lemma~\ref{lem:deterministic_shifted_sewing}, yielding estimate \eqref{eq:besov_regularity_general_proof1} (cf. estimate \eqref{eq:deterministic_shifted_conclusion_2} therein).
\end{proof}

\begin{corollary}\label{cor:besov_regularity_general_2}
	Let $X$, $\mathcal{G}$, $u$, $\kappa_1$, $\kappa_2$ be as in Lemma \ref{lem:besov_regularity_general_1}. Then for any $(\alpha,\tilde q)\in (0,\infty)\times [1,\infty]$ such that
	\begin{align} \label{cond:corollary}
		\alpha< \min\Big\{ \frac{{\kappa_1}+{\kappa_2}}{H}-1, \frac{1}{\tilde q H} \Big\},
	\end{align}
	it holds $\law(X_\cdot\vert \mathcal{G})\in L^{\tilde q}([u,T];\cB^\alpha_{1,1})$ $\PP$-a.s. and for any $(s,t)\in\simp{u}{T}$ we have the estimate
	\begin{equation}\label{eq:besov_regularity_general_2}
		\Big\| \|\law(X_\cdot \vert \mathcal{G}) \|_{L^{\tilde q}_{[s,t]}\cB^\alpha_{1,1}} \Big\|_{L^\infty_\Omega} \lesssim
		(t-s)^{\frac{1}{\tilde q}-\alpha H} \Big(1+ w_\theta(s,t)^{\kappa_1} (t-s)^{{\kappa_2}-H}\Big). 
	\end{equation}
\end{corollary}

\begin{proof}
	Fix $(s,t)\in \simp{u}{T}$ and for $r\in [s,t]$, set $\mu_r=\law(X_r|\mathcal{G})$.
	Similarly to the proof of Lemma \ref{lem:smoothinglemma}, we can introduce a mollification $\mu^\eps_r=\chi^\eps\ast \mu_r$ to guarantee that $\mu^\eps\in L^{\tilde q}([u,T];\cB^\alpha_{1,1})$, and proceed to derive uniform-in-$\eps$ bounds; the conclusion then follows by Fatou's lemma and lower semicontinuity of Lebesgue-Besov norms (Lemma \ref{lem:uniformBesov}). For notational simplicity, let us omit the presence of the mollifier in the following.
	
	For any given $f\in \cC^\infty_b([s,t]\times\R^d)$, we can apply Lemma \ref{lem:besov_regularity_general_1} with $\gamma=-\alpha$. Hence, we have the $\PP$-a.s. bound
	\begin{align*}
		\Big| \int_s^t \langle f_r,\mu_r \rangle\, \dd r\Big|
		= \Big| \EE\Big[ \int_s^t f(X_r) \dd r \Big| \mathcal{G} \Big]\Big|
		\leq C \| f\|_{L^{\tilde q'}_{[s,t]} \cB^{-\alpha}_{\infty}} (t-s)^{\frac{1}{\tilde q}-\alpha H} \Big(1+ w_\theta(s,t)^{\kappa_1} (t-s)^{{\kappa_2}-H}\Big) .
	\end{align*}
	Taking a countable family $\{f_n\}_n$ which is dense in $\cE=\cC^\infty_b([s,t]\times\R^d)\cap L^{\tilde q'}([s,t]; \cB^{-\alpha}_\infty)$, w.r.t. the $L^{\tilde q'}([s,t]; \cB^{-\alpha}_\infty)$-norm, we can find a set of full probability where the estimate holds for all $f_n$, and then extend by density to being $\PP$-a.s. valid for all $f\in\cE$. By Lemma~\ref{lem:dualspace} we conclude that $\PP$-a.s.
	\begin{align*}
	\lVert \law(X_\cdot\vert \mathcal{G})\|_{L^{\tilde{q}}_{[s,t]}\mathcal{B}^{\alpha}_{1,1}}
	&\lesssim 	\sup_{f \in \cE: \| f\|_{L^{q'}_{[s,t]} \cB^{-\alpha}_\infty} \leq 1} \Big| \int_s^t \langle f_r, \mu_r \rangle \dd r \Big|\\
	&\leq C (t-s)^{\frac{1}{\tilde q}-\alpha H} \Big(1+ w_\theta(s,t)^{\kappa_1} (t-s)^{{\kappa_2}-H}\Big). \qedhere
	\end{align*}
\end{proof}

\begin{proof}[Proof of Theorem~\ref{thm:Besovregularity_better}]

Let $(u,t)\in [0,T]^2_\leq$ with $|t-u|\leq 1$ fixed.

First assume $\gamma<0$. 
For $\varphi=X-B$, combining \eqref{eq:compcondexp} with the a priori estimate \eqref{eq:apriori-gamma<0simple} and the tower property $\EE^u=\EE^u \EE^s$, one finds
\begin{align*}
\big\| \EE^u|\varphi_{t} - \EE^s \varphi_{t}| \big\|_{L^\infty_\Omega}
\leq \big\| \EE^s|\varphi_{t} - \EE^s \varphi_{t}| \big\|_{L^\infty_\Omega}
\leqslant C \|b\|_{L^q_{[s,t]}\mathcal{B}^\gamma_\infty} (t-s)^{\gamma H +\frac{1}{q^\prime}}\Big(1+\|b\|_{L^q_{[s,t]}\mathcal{B}^\gamma_\infty}\Big)
\end{align*}
for any $s\in [u,t]^2_\leq$. Therefore the assumptions of Corollary~\ref{cor:besov_regularity_general_2} are satisfied for $\theta=\varphi$, $\mathcal{G}=\cF_u$ and 
\begin{align*}
w_{\varphi}(s,t)=\|b\|^q_{L^q_{[s,t]}\mathcal{B}^\gamma_\infty}\Big(1+\|b\|_{L^q_{[s,t]}\mathcal{B}^\gamma_\infty}\Big)^{q},\quad
\kappa_1=\frac{1}{q}, \quad
\kappa_2=\gamma H+\frac{1}{q^\prime}>H ,
\end{align*}
where the last inequality follows from Assumption \ref{ass:strong}.
In particular, Assumption~\eqref{eq:cond_besov_regularity} ensures the validity of \eqref{cond:corollary}, since $\kappa_1+\kappa_2/H-1=1/H+\gamma-1$.
Taking $s=u$ in \eqref{eq:besov_regularity_general_2} yields \eqref{eq:besovregularity1_conditional}.

For $\gamma>0$, the proof of estimate \eqref{eq:besovregularity2_conditional} follows by identical arguments, up to invoking the a priori bound \eqref{eq:apriori-gamma>0} instead, resulting in a slightly different $w_\varphi$. 

Since $T$ is finite, we can always cover $[u,T]$ by finitely many intervals of length at most $1$, so that $\law(X_\cdot\vert \cF_u)\in L^\infty_\Omega L^{\tilde{q}}([u,T];\mathcal{B}^{\alpha}_{1,1})$ follows.
\end{proof}

\section{Existence and uniqueness for McKean--Vlasov equations}\label{sec:McKeanVlasov}

In this section, we prove two results on McKean--Vlasov equations~\eqref{eq:McKV}: first, Theorem~\ref{thm:mckeanvlasov} provides an existence result under weak regularity assumptions on $\mathfrak{b}$, relying on the regularity of the law provided by Theorem~\ref{thm:Besovregularity_better}. Second, Theorem~\ref{Mainresult-McKV-uniqueness-full} is a uniqueness result under slightly stronger conditions, in particular on the regularity of the initial condition.

\begin{theorem}\label{thm:mckeanvlasov}
Let $H\in (0,\infty)\setminus \N$ and $\mathfrak{b}\in L^\infty([0,T];\mathcal{B}^\theta_{\infty})$ for $\theta\in \R$ satisfying
\begin{equation}\label{eq:hypMcKV}
\theta>1-\frac{1}{H}.
\end{equation}
Let $Y_{0}$ an $\mathcal{F}_{0}$-measurable, $\R^d$-valued random variable with law denoted by $\mu_{0}$;
then there exists a strong solution $Y$ to the McKean--Vlasov equation \eqref{eq:McKV} starting at $Y_0$.

More precisely, setting $\mu_t \coloneqq \mathcal{L}(Y_t)$, we have the following:
\begin{enumerate}[label=(\alph*)]
\item  \label{en:reg1}for any $\rho \in [1,\infty)$ and $\alpha>0$ with
\begin{equation}\label{eq:condalpha}
\alpha< \frac{1}{\rho H},
\end{equation}
we have $\mu\in L^{\rho}([0,T];\mathcal{B}^\alpha_{1,1})$;

\item \label{en:reg2}in particular, $\mathfrak{b}\ast \mu \in L^{1}([0,T];\mathcal{C}^{1}_{b})$, so that the (linear) {\renewcommand\SDE{SDE($\mathfrak{b}\ast \mu$,$Y_0$)}\ref{eq:SDEintro}\renewcommand\SDE{SDE($b$,$X_0$)}} admits a unique strong solution, which is given by the process $Y$.
\end{enumerate}
\end{theorem}

Theorem~\ref{thm:mckeanvlasov} above is a more detailed version of Theorem~\ref{Mainresult-McKV} from the introduction.
As for uniqueness, 
we have Theorem~\ref{Mainresult-McKV-uniqueness}, which we recall here for the reader's convenience.

\begin{theorem}\label{Mainresult-McKV-uniqueness-full}
Let $H\in (0,+\infty)\setminus \N$ and $\mathfrak{b}\in L^\infty([0,T];\cB^\theta_p)$ for some $\theta\in (-\infty,1)$, $p\in [1,\infty]$ satisfying
	\begin{equation}\label{eq:condition-uniqueness-full}
		\theta>1-\frac{1}{2H}, \quad \theta-\frac{d}{p}>1-\frac{1}{H} .
	\end{equation}
Further assume that $\law(Y_0)\in L^\infty_x$.
Then there exists a strong solution, in the sense specified in Theorem~\ref{thm:mckeanvlasov}, to the McKean--Vlasov equation \eqref{eq:McKV} starting from $Y_{0}$, such that
\begin{equation}\label{eq:condition_uniqueness2-full}
	\mathfrak{b}\ast \mu\in L^1([0,T];\cC^1_b);
\end{equation}
moreover, pathwise uniqueness and uniqueness in law hold in the class of solutions satisfying \eqref{eq:condition_uniqueness2-full}.
\end{theorem}

The proofs of Theorem~\ref{thm:mckeanvlasov} and Theorem~\ref{Mainresult-McKV-uniqueness-full} are given respectively in Section~\ref{sec:proof-existence-McKV} and Section~\ref{sec:proofMcKVuniq} below.

\begin{remark}\label{rk:Coulomb-kernel}
	Let $\fkb$ be a compactly supported\footnote{The compact support assumption is merely technical, as it would be natural to measure the regularity of $\fkb$ away from the origin in $\cB^\theta_\infty$ scales, which however creates a mismatch in the parameter $p$ compared to the one used to measure regularity around $0$. We expect the same conclusion to hold even without this assumption.} interaction kernel of Riesz type, which is only singular around the origin with $|\fkb(x)|\sim |x|^{-s}$ for some $s\in (0,d)$. By taking $\theta=0$, $p= (d/s)^-$ (meaning $p=d/s-\varepsilon$ for any small $\varepsilon>0$), Theorem~\ref{Mainresult-McKV-uniqueness-full} ensures well-posedness of the McKean--Vlasov equation for
	\begin{align*}
		\law(Y_0)\in L^\infty_x, \quad H\in (0,1/2), \quad s<\frac{1}{H}-1.
	\end{align*}
In particular, the case of Coulomb-type interaction kernels in dimesion $d\geq 2$ corresponds to $s=d-1$, and the previous condition then implies well-posedness for any $H<1/d$. For instance in $d=2$, Coulomb-type kernels can only be reached for $H<1/2$: this is coherent with known difficulties that occur when $H=1/2$ with attractive interactions such as the Keller-Segel kernel $\fkb(x) = -\chi \frac{x}{|x|^2},\ \chi>0$, for which blow-up occur for large enough values of $\chi$ \cite{BlanchetDolbeaultPerthame,FournierJourdain}.

If additionally $s\in (0,d-1)$ and $|\nabla \fkb(x)|\sim |x|^{-s-1}$, then by taking $\theta=1^-$, $p= (\frac{d}{s+1})^-$, we similarly obtain well-posedness of the McKean--Vlasov equation for
\begin{align*}
		\law(Y_0)\in L^\infty_x, \quad H\in [1/2,1), \quad s<\frac{1}{H}-1.
	\end{align*}
\end{remark}

\subsection{Proof of Theorem~\ref{thm:mckeanvlasov}}\label{sec:proof-existence-McKV}

The proof of Theorem~\ref{thm:mckeanvlasov} follows by a priori estimates and compactness. Therefore we first give an a priori estimate for the law of the solution to \eqref{eq:McKV}. This is done via a local Gr\"onwall-type argument borrowed from rough path theory (see \cite[Section 8.4]{FrizVictoir}).

\begin{lemma} \label{lem:globalapriori}
Let $\theta>1-1/H$. Let $\mathfrak{b} \in L^\infty([0,T];\mathcal{C}^\infty_b)$ and denote by $Y$ the unique strong solution to \eqref{eq:McKV}. Then for any $\rho \in (1,2]$ and any $\alpha<\frac{1}{\rho H}$, there exists a constant $C=C(\alpha,\rho,T,H,\theta,d,\|\mathfrak{b}\|_{L^\infty_{T}\mathcal{B}^\theta_\infty})$, which is non-decreasing w.r.t. $\|\mathfrak{b}\|_{L^\infty_{T}\mathcal{B}^\theta_\infty}$, such that
\begin{equation*}
\lVert \law(Y_\cdot)\|_{L^{\rho}_{T}\mathcal{B}^{\alpha}_{1,1}}
\leq C.
\end{equation*}
\end{lemma}

\begin{proof}
Set $\mu_t=\mathcal{L}(Y_t)$.
For $\rho\sim 1$, the idea of the proof is loosely the following: since $\fkb$ is smooth, we know that $\mu\in L^{1+}([0,T]; \cB^{1/H-}_{1,1})$ (case (a) of Remark \ref{rem:regularity}). Therefore by Young convolutional inequalities and the assumption on $\theta$, it holds
\begin{align*}
	\| \fkb\ast \mu\|_{L^{1+}_T \cC^1_b} 
	\lesssim \| \fkb\ast \mu\|_{L^{1+}_T \cB^{\theta+1/H-}_b}
	\lesssim \| \fkb\|_{L^\infty_T \cB^\theta_\infty} \| \mu\|_{L^{1+}_T \cB^{1/H-}_{1,1}}
\end{align*}
which suggests the possibility to bound $\| \mu\|_{L^{1+}_T \cB^{1/H-}_{1,1}}$ in function of itself.

To make the argument rigorous, we need to work on sufficiently short time intervals and introduce a few additional parameters.
To this end, first notice that by Besov embeddings, we are always allowed to decrease $\theta$ and increase $\alpha$ as needed (this will only make the result stronger). Therefore we may assume them w.l.o.g. to be of the form
\begin{align*}
	\theta=1-\frac{1}{H}+2\delta, \quad \alpha=\frac{1}{\rho H}-\delta
\end{align*}
for some arbitrarily small parameter $\delta>0$. Set 
\begin{align*}
	\gamma \coloneqq \theta + \alpha = 1-\frac{1}{H}+\frac{1}{\rho H}+\delta= 1-\frac{1}{\rho' H}+\delta.
\end{align*}
W.l.o.g. $\gamma\neq 0$.
By H\"older's inequality and a convolution inequality in Besov spaces (see Lemma~\ref{lem:Besovconvolution}), for any $(s,t) \in \simp{0}{T}$, it holds that
\begin{align} \label{eq:convineq}
	\| \fkb \ast \mu\|_{L^\rho_{[s,t]} \mathcal{B}^\gamma_\infty} \lesssim \| \fkb\|_{L^\infty_{[s,t]} \mathcal{B}^\theta_\infty} \lVert \mu\|_{L^\rho_{[s,t]} \mathcal{B}^{\alpha}_{1,1}}.
\end{align}
By construction, $\gamma$ fulfills Assumption~\ref{ass:strong} for $q=\rho$.
Hence, we can apply Theorem~\ref{thm:Besovregularity_better} for {\renewcommand\SDE{SDE($b$,$Y_0$)}\ref{eq:SDEintro}\renewcommand\SDE{SDE($b$,$X_0$)}} with drift $b=\mathfrak{b}\ast \mu \in L^\rho([0,T];\mathcal{B}^\gamma_\infty)$.
Depending on whether $\gamma<0$ or $\gamma \in (0,1)$, we use \eqref{eq:besovregularity1} or \eqref{eq:besovregularity2}. For $\gamma<0$ and $|t-s|\leqslant 1$, we get
\begin{align*}
	\lVert \mu\|_{L^{\rho}_{[s,t]} \mathcal{B}^{\alpha}_{1,1}}
	\lesssim (t-s)^{\delta H} + \| \fkb\|_{L^\infty_{[s,t]} \mathcal{B}^\theta_\infty} \lVert \mu\|_{L^{\rho}_{[s,t]} \mathcal{B}^{\alpha}_{1,1}} (1 +\| \fkb\|_{L^\infty_{[s,t]} \mathcal{B}^\theta_\infty}\lVert \mu\|_{L^{\rho}_{[s,t]} \mathcal{B}^{\alpha}_{1,1}}) (t-s)^{2 \delta H},
\end{align*}
since by our choice of parameters we have
\begin{align*}
	\frac{1}{\tilde q}-H\alpha = \delta H, \quad \frac{1}{\tilde q}+\frac{1}{q'}+H(\gamma-\alpha-1)=2\delta H.
\end{align*}
Instead for $\gamma\in (0,1)$ and $|t-s|\leqslant 1$, using that $\gamma H + \frac{1}{q'}= (1+\delta)H$ and $\llbracket \cdot\rrbracket_{\mathcal{C}^{\gamma}}\lesssim \|\cdot\|_{\mathcal{B}^\gamma_\infty}$, we get
\begin{align*}
	\lVert \mu\|_{L^{\rho}_{[s,t]} \mathcal{B}^{\alpha}_{1,1}}
	\lesssim (t-s)^{\delta H} + \| \fkb\|_{L^\infty_{[s,t]} \mathcal{B}^\theta_\infty} \lVert \mu\|_{L^{\rho}_{[s,t]} \mathcal{B}^{\alpha}_{1,1}} (1 +\| \fkb\|_{L^\infty_{[s,t]} \mathcal{B}^\theta_\infty}^{\eta}\lVert \mu\|_{L^{\rho}_{[s,t]} \mathcal{B}^{\alpha}_{1,1}}^{\eta}) (t-s)^{(1+\delta)H}.
\end{align*}
We can merge the two regimes due to the following: First $(1+\delta)H\geqslant 2\delta H$ and w.l.o.g. $\eta\geqslant 1$, as for $\eta \in (0,1)$ it holds $x^{1+\eta}\leqslant x+x^2$ for all $x\geqslant 0$.
Hence, for any $|t-s|\leqslant 1$,
\begin{equation}\label{ineq:apriorifirst}\begin{split} 
	\lVert \mu\|_{L^{\rho}_{[s,t]} \mathcal{B}^{\alpha}_{1,1}}
	\leq C (t-s)^{\delta H} & + C \| \fkb\|_{L^\infty_{[s,t]} \mathcal{B}^\theta_\infty} \lVert \mu\|_{L^{\rho}_{[s,t]} \mathcal{B}^{\alpha}_{1,1}} (t-s)^{2\delta H}\\
	& + C \| \fkb\|_{L^\infty_{[s,t]} \mathcal{B}^\theta_\infty}^{1+\eta}\lVert \mu\|_{L^{\rho}_{[s,t]} \mathcal{B}^{\alpha}_{1,1}}^{1+\eta} (t-s)^{2 \delta H}.
\end{split}\end{equation}

Choose $\bar{h}\leqslant 1$ small enough such that
\begin{equation}\label{eq:apriori_MKV_hbar}
	C \| \fkb\|_{L^\infty_{T} \mathcal{B}^\theta_\infty}\bar{h}^{2\delta H} \leq \frac12, \quad
	C \| \fkb\|_{L^\infty_{T} \mathcal{B}^\theta_\infty}^{1+\eta}\bar{h}^{\delta H} \leq \frac12, \quad
	4C \bar{h}^{\delta H (1+\frac{1}{\eta})}< \frac13
\end{equation}
(the last condition will be used later); then, for $|t-s|\leq \bar{h}$, from \eqref{ineq:apriorifirst} we get that
\begin{equation*}
	\lVert \mu\|_{L^{\rho}_{[s,t]} \mathcal{B}^{\alpha}_{1,1}}
	\leq 2C (t-s)^{\delta H} + \lVert \mu\|_{L^{\rho}_{[s,t]} \mathcal{B}^{\alpha}_{1,1}}^{1+\eta} (t-s)^{\delta H}.
\end{equation*}
It remains to close estimates for $\lVert \mu\|_{L^{\rho}_{[s,t]} \mathcal{B}^{\alpha}_{1,1}}$. 
Note that
\begin{align}\label{eq:apriori_MKV_eq4}
\lVert \mu\|_{L^{\rho}_{[s,t]} \mathcal{B}^{\alpha}_{1,1}} (t-s)^{\frac{\delta H}{\eta}}\leqslant 2C (t-s)^{\delta H+\frac{\delta H}{\eta}}+\lVert \mu\|^{1+\eta}_{L^{\rho}_{[s,t]} \mathcal{B}^{\alpha}_{1,1}}(t-s)^{\delta H+\frac{\delta H}{\eta}}.
\end{align}
Define $g(s,t)\coloneqq \lVert \mu\|_{L^{\rho}_{[s,t]} \mathcal{B}^{\alpha}_{1,1}}(t-s)^{\frac{\delta H}{\eta}}$.
On one hand it holds $\lim_{|t-s|\to 0} g(s,t)=0$, on the other hand, by \eqref{eq:apriori_MKV_eq4},
\begin{equation}\label{eq:apriori_MKV_eq5}
	g(s,t)\leq 2C (t-s)^{\delta H+\frac{\delta H}{\eta}} + g(s,t)^{1+\eta}.
\end{equation}
Equation \eqref{eq:apriori_MKV_eq5} implies that for $g\leq 1/2$ and $|t-s|<\bar{h}$, it holds that $g(s,t)\leq 4C(t-s)^{\delta H+\frac{\delta H}{\eta}}<1/3$ (if $g\leq 1/2$, then $g \leq 2(g-g^{1+\eta})$, using that $\eta\geqslant 1$), where we used the last assumption on $\bar h$ from \eqref{eq:apriori_MKV_hbar}.
However, since $(s,t)\mapsto g(s,t)$ is continuous, it cannot jump. Combined with the information that $g(s,t)<1/2$ whenever $|t-s|\ll 1$, this implies that
\begin{align*}
	g(s,t)\leq 2C(t-s)^{\delta H+\frac{\delta H}{\eta}}, \quad \forall\, |t-s|\leq \bar{h}.
\end{align*}
So by definition of $g$ one gets that
\begin{equation*}
	\lVert \mu\|_{L^{\rho}_{[s,t]} \mathcal{B}^{\alpha}_{1,1}} \leq 2C (t-s)^{\delta H}, \quad  \forall\, |t-s|\leq \bar{h}.
\end{equation*}
In light of \eqref{eq:apriori_MKV_hbar}, $\bar{h}$ is a positive parameter which depends on $H$, $\theta$, $d$ and $\| \fkb\|_{L^\infty_{T} B^\theta_\infty}$ but nothing else, while the  the constant $C$ additionally depends on $\alpha$ and $\rho$.
Since $T$ is finite, we conclude by triangular inequality that 
\begin{equation*}
\lVert \mu\|_{L^{\rho}_{T}\mathcal{B}^\alpha_{1,1}} \leq C(\alpha,\rho,T,H,\theta,d,\|\mathfrak{b}\|_{L^\infty_{T} \mathcal{B}^\theta_\infty}). \qedhere
\end{equation*}
\end{proof}

Having the a priori estimate at hand, we can now proceed with the compactness-stability argument.

\begin{proof}[Proof of Theorem~\ref{thm:mckeanvlasov}]
It suffices to prove the result for $\rho \in (1,2]$ sufficiently close to $1$. Indeed, the case $\rho=1$ follows by the finiteness of $T$, as for any $\alpha<\frac{1}{H}$ we can choose $\tilde{\rho}>1$ small enough such that $\alpha<\frac{1}{\tilde{\rho}H}$.
Instead, the values $\tilde\rho>\rho$ follow by interpolating between the bound for $\rho$ and the trivial bound in $L^\infty_T \cB^{0-}_{1,1}$, cf. Remark \ref{rem:regularity}.

As in the proof of Lemma \ref{lem:globalapriori},  we may assume $\theta$ and $\alpha$ to be given by
\begin{align*}
	\theta=1-\frac{1}{H}+2\delta, \quad \alpha=\frac{1}{\rho H}-\delta,
\end{align*}
for some small $\delta>0$, where now additionally $\rho>1$ is chosen small enough such that
\begin{align*}
	\gamma \coloneqq \theta + \alpha = 1-\frac{1}{H}+\frac{1}{\rho H}+\delta= 1-\frac{1}{\rho' H}+\delta>1.
\end{align*}
Let $\mathfrak{b}^n$ be a sequence with $\mathfrak{b}^n \in L^\infty([0,T];\mathcal{C}^\infty_b)$ converging to $\mathfrak{b}$ in $L^\infty([0,T];\mathcal{B}^{\theta -}_\infty)$. Let $Y^n$ be the corresponding strong solution to the McKean--Vlasov equation \eqref{eq:McKV}. 

We aim at the identification of the limit points of $(Y^n)_{n\in \N}$ and $(\mu^n)_{n\in \N}$. The procedure is the following: first we show convergence of $(\mu^n)_{n\in \N}$ to a limit $\mu^\infty$ fulfilling $\mathfrak{b}\ast \mu^\infty \in L^{\rho}([0,T],\mathcal{C}^\gamma)$; next we prove that {\renewcommand\SDE{SDE($\mathfrak{b}\ast \mu^\infty$,$Y_0$)}\ref{eq:SDEintro}\renewcommand\SDE{SDE($b$,$X_0$)}} has a unique strong solution with $\mathcal{L}(Y)=\mu^\infty$.

Note that by Lemma~\ref{lem:globalapriori} applied to the solution of {\renewcommand\SDE{SDE($\mathfrak{b}^n\ast \mu^n$,$Y_0$)}\ref{eq:SDEintro}\renewcommand\SDE{SDE($b$,$X_0$)}} and a convolution inequality for Besov spaces (see \eqref{eq:convineq}), we have
\begin{align}\label{eq:unifbounds}
\sup_n \|\mu^n\|_{L^{\rho}_{T}\mathcal{B}^\alpha_{1,1}}<\infty \quad \text{and} \quad \sup_n\|\mathfrak{b}^n \ast \mu^n\|_{L^{\rho}_{T} \mathcal{B}^{\gamma}_\infty}<\infty.
\end{align}

By H\"older's inequality, for $m\geqslant 1$,
\begin{align*}
\EE |Y^n_{t} - Y^n_{s}|^m &\lesssim \EE \left| \int_{s}^t \mathfrak{b}^n_{r}\ast \mu^n_{r}(Y^n_{r}) \, \dd r \right|^m + \EE |B_{t} - B_{s}|^m \\
&\lesssim (t-s)^{\frac{m(\rho-1)}{\rho}} \| \mathfrak{b}^n\ast \mu^n \|_{L^{\rho}_{T} L^\infty_{x}}^m + (t-s)^{Hm}.
\end{align*} 
Thus applying Kolmogorov's tightness criterion, $(\mu^n)_{n\in \N}$ is tight in a space of H\"older continuous functions on $[0,T]$ and therefore it is also tight on $\mathcal{C}_{[0,T]}$  by Arzel\`a-Ascoli's Theorem. Hence, up to extracting a (not relabelled) subsequence, $(\mu^n)_{n\in \N}$ converges weakly in the space of probability measures on $\mathcal{C}_{[0,T]}$ to some $\mu^\infty$. Combining this with \eqref{eq:unifbounds} and Lemma~\ref{lem:uniformBesov} gives that 
\begin{align}\label{eq:aandb}
\|\mu^\infty\|_{L^{\rho}_{T}\mathcal{B}^\alpha_{1,1}}<\infty \quad \Rightarrow\quad b^\infty \coloneqq \mathfrak{b}\ast \mu^\infty\in L^{\rho}([0,T];\mathcal{B}^{\gamma}_{\infty}) \text{ with } \gamma>1.
\end{align}
Hence, there exists a pathwise unique strong solution $Y$ to {\renewcommand\SDE{SDE($b^\infty$,$Y_0$)}\ref{eq:SDEintro}\renewcommand\SDE{SDE($b$,$X_0$)}}, defined on the same probability space. It remains to show that $\law(Y)=\mu^\infty$, which we do via a (pathwise) stability argument.

By construction, the difference $Y^n-Y$ satisfies
\begin{align*}
	 |Y^n_t-Y_t|
	 & = \bigg| \int_0^t \big[ \mathfrak{b}^n_s \ast \mu^n_s(Y^n_s) - \mathfrak{b}^n_s \ast \mu^n_s(Y_s)\big] + \big[ \mathfrak{b}^n_s \ast \mu^n_s(Y_s) - \mathfrak{b}_s \ast \mu^\infty_s(Y_s)\big]\, \dd s\bigg|\\
	 & \leq \int_0^t \| \mathfrak{b}^n_s \ast \mu^n_s\|_{\cC^1_b} |Y^n_s-Y_s|\, \dd s + \int_0^t |\mathfrak{b}^n_s \ast \mu^n_s(Y_s) - \mathfrak{b}_s \ast \mu^\infty_s(Y_s)| \, \dd s,
\end{align*}
so that by Gr\"onwall's lemma
\begin{align*}
	\sup_{t\in [0,T]} |Y^n_t-Y_t|
	& \leq e^{T \sup_n \| \mathfrak{b}^n_s \ast \mu^n_s\|_{L^1_T \cC^1_b}} \int_0^T |\mathfrak{b}^n_s \ast \mu^n_s(Y_s) - \mathfrak{b}_s \ast \mu^\infty_s(Y_s)|\, \dd s\\
	& \lesssim  \int_0^T |(\mathfrak{b}^n_s-\mathfrak{b}_s) \ast \mu^n_s(Y_s)\,| \dd s + \int_0^T |(\mathfrak{b}_s \ast (\mu^n_s-\mu^\infty_s))(Y_s)|\, \dd s \eqqcolon  I^n + J^n.
\end{align*}
In order to conclude, it suffices to show that $I^n + J^n\to 0$ in probability.

Let $\varepsilon \in (0,1)$. By a convolution inequality, the uniform estimate in \eqref{eq:unifbounds} and the assumption $\mathfrak{b}^n\to \mathfrak{b}$ in $L^\infty_T \cB^{\theta-}_\infty$, we have
\begin{align*}
	I^n \leq \|(\mathfrak{b}^{n}-\mathfrak{b})\ast \mu^n\|_{L^1_T \cC^0_{b}}
	& \leq \| (\mathfrak{b}^n-\mathfrak{b})\ast \mu^n\|_{L^{\rho}_{T}\mathcal{B}^{\theta+\alpha-\varepsilon}_{\infty}}\\
	&\lesssim  \|\mathfrak{b}^{n}-\mathfrak{b}\|_{L^\infty_{T}\mathcal{B}^{\theta-\varepsilon}_{\infty}}\, \sup_n \|\mu^n \|_{L^{\rho}_{T} \mathcal{B}^{\alpha}_{1}} \underset{n \rightarrow 
    \infty}{\longrightarrow} 0.
\end{align*}
To prove convergence for $J^n$, let us fix $h>0$ and consider $\varphi\in L^\infty_T \cC^\infty_b$ such that $\| \mathfrak{b}-\varphi\|_{L^\infty_T \cB^{\theta-\vep}_\infty}\leq h$.
Since $\mu^n$ converges weakly to $\mu$, $\varphi_s\ast \mu^n_s(x)\to \varphi_s\ast\mu_s(x)$ for Lebesgue a.e. $s\in [0,T]$ and every $x\in\mathbb{R}^d$, so that by dominated convergence
\begin{align*}
	\lim_{n\to\infty}\int_0^T |(\varphi_s\ast (\mu^n_s-\mu^\infty_s))(Y_s)|\, \dd s =0;
\end{align*}
on the other hand, we have
\begin{align*}
	\Big| J^n - \int_0^T (\varphi_s\ast (\mu^n_s-\mu^\infty_s))(Y_s) \, \dd s \Big|
	& \leq \int_0^T |(\mathfrak{b}_s-\varphi_s)\ast(\mu^n_s-\mu^\infty_s)(Y_s)|\, \dd s\\
	& \leq \int_0^T \|(\mathfrak{b}_s-\varphi_s)\ast(\mu^n_s-\mu^\infty_s)\|_{\cC^0_b} \, \dd s\\
	& \lesssim \| \mathfrak{b}-\varphi\|_{L^\infty_T \cB^{\theta-\vep}_\infty} \sup_n \| \mu^n\|_{L^1_T \cB^\alpha_{1}} \lesssim h.
\end{align*}
Combining these estimates we deduce that
\begin{align*}
	\limsup_{n\to\infty} |J^n| \lesssim h
\end{align*}
and so, by the arbitrariness of $h>0$, $J^n\to 0$ as well. Overall this shows that $\PP$-a.s. $Y^n\to Y$ in $\cC_{[0,T]}$; it follows that $\mu^\infty_t = \law(Y_t)$, proving that $Y$ is a strong solution to the McKean--Vlasov equation \eqref{eq:McKV}. Finally, \ref{en:reg1} and \ref{en:reg2} follow from \eqref{eq:aandb}.
\end{proof}

\subsection{Proof of Theorem~\ref{Mainresult-McKV-uniqueness-full}}\label{sec:proofMcKVuniq}

Having established strong existence in Theorem~\ref{thm:mckeanvlasov}, we can now exploit it to obtain pathwise uniqueness and uniqueness in law under additional conditions on $\fkb$ and $\mathcal{L}(Y_0)$. To this end, we start by noticing that any $L^p_x$-regularity of $\mathcal{L}(Y_0)$ is propagated by the dynamics, readapting arguments from \cite{GaHaMa2022}. With a slight abuse, given a probability measure $\mu$, in what follows we write $\mu\in L^p_x$ to mean that $\mu$ is absolutely continuity w.r.t. the Lebesgue measure, with density belonging to $L^p_x$; we use $\| \mu\|_{L^p_x}$ to denote the $L^p_x$-norm of its density.

\begin{lemma}\label{lem:propagation_integrability}
Let $H\in (0,\infty)\setminus\N$, $b\in L^1([0,T];\cC^1_b)$ and consider \ref{eq:SDEintro}.
If $\law(X_0)\in L^p_x$ for some $p\in [1,\infty]$, then the same holds for $\law(X_t)$ and
\begin{align*}
	\lVert \law(X_t)\|_{L^p_x} \leq \lVert \law(X_0) \rVert_{L^p_{x}} \exp\Big(\frac{1}{p'} \int_0^t \| {\rm div} b_s\|_{L^\infty_x} \dd s \Big),\quad \forall\, t\in [0,T].
\end{align*}
\end{lemma}

\begin{proof}
Since $b\in L^1([0,T];\cC^1_b)$ and the noise is additive, the SDE admits a classical (pathwise defined) stochastic flow coming from ODE theory, denote it by $\Phi_{s\to t}(x;\omega)$; let $\Phi_{s\leftarrow t}(\cdot;\omega)$ be the inverse of $\Phi_{s\to t}(\cdot;\omega)$, as a map from $\R^d$ to itself.
Since $X_0$ is $\cF_0$-measurable, it is independent of the driving noise $B$; by assumption, $\law(X_0)(\dd x) = \rho_0(x) \dd x$ for some $\rho_0\in L^p_x$.
Arguing as in \cite[Prop. 4.3]{GaHaMa2022}, it follows that $\law(X_t)$ is also absolutely continuous w.r.t. Lebesgue measure, with an explicit formula for its density $\rho_t$ given by
\begin{equation}\label{eq:density_formula}
	\rho_t(x) = \EE\Bigg[ \rho_0(\Phi_{0\leftarrow t}(x;\omega)) \exp\Big( -\int_0^t {\rm div} b_s(\Phi_{s\leftarrow t} (x;\omega))\, \dd s\Big) \Bigg].
\end{equation}
From here, similarly to \cite[Prop. 4.3]{GaHaMa2022}, one can estimate the $L^p_x$-norm of $\rho_t$ by using Jensen's inequality, Fubini and the pathwise formula $\det (D \Phi_{0\to t} (x;\omega)) = \exp \big( \int_0^t {\rm div} b_s(\Phi_{0\rightarrow s} (x;\omega))\, \dd s \big)$ to find
\begin{align*}
	\int_{\R^d} |\rho_t(x)|^p\, \dd x
	& \leq \EE\Bigg[\int_{\R^d} |\rho_0(\Phi_{0\leftarrow t}(x;\omega))|^p \exp\Big( - p \int_0^t {\rm div} b_s(\Phi_{s\leftarrow t} (x;\omega))\, \dd s\Big) \dd x \Bigg]\\
	& = \EE\Bigg[\int_{\R^d} |\rho_0(y)|^p \exp\Big( (1- p) \int_0^t {\rm div} b_s(\Phi_{0\to s} (y;\omega))\, \dd s\Big) \dd y \Bigg]
\end{align*}
which yields the conclusion.
\end{proof}

Let now $Y$ be a solution to the McKean--Vlasov SDE, so that it also solves a linear SDE with effective drift $b_t=\mathfrak{b}_t\ast \mu_t$, for $\mu_t=\law(Y_t)$.
Leveraging on Lemma \ref{lem:propagation_integrability} and the convolutional structure, we can expect $b$ to inherit higher integrability from $\mu$; to obtain uniqueness, we refine the strategy first developed in \cite[Section 5]{GHM}, based on carefully combining this information with stability estimates for linear SDEs and Wasserstein distances.

\begin{proof}[Proof of Theorem~\ref{Mainresult-McKV-uniqueness-full}]
	By condition \eqref{eq:condition-uniqueness} and the Besov embedding $\cB^\theta_p\hookrightarrow \cB^{\theta-d/p}_\infty$, $\fkb$ satisfies the assumptions of Theorem~\ref{thm:mckeanvlasov}, so that strong existence of a solution to the McKean--Vlasov SDE such that $\fkb\ast \mu\in L^1_T \cC^1_b$ follows.

	Next we focus on pathwise uniqueness. To this end, first note that we may assume $p>1$ without loss of generality; indeed, for $p=1$, by Besov embedding $\cB^\theta_1\hookrightarrow \cB^{\tilde \theta}_{\tilde p}$ for $\tilde\theta= \theta-d(1-1/\tilde p)$, where we can choose $p>1$ sufficiently small so that $(\tilde \theta,\tilde p)$ still satisfies condition \eqref{eq:condition-uniqueness}.
	Suppose now we are given two solutions $Y^i$, $i=1,2$, defined on the same probability space and with same data $(Y_0,B)$, such that $b^i\coloneqq\fkb \ast \mu^i\in L^1_T \cC^1_b$, where $\mu^i_t\coloneqq\law(Y^i_t)$. Since each $Y^i$ solves the linear SDE with drift $b^i$, Lemma \ref{lem:propagation_integrability} applies, and so $t\mapsto \law(Y^i_t)\in L^\infty_T L^\infty_x$. Since moreover $ \lVert \law(Y^i_t)\rVert_{L^1_x}=1$ for all $t\geq 0$, by interpolation and Young inequalities we deduce that
	\begin{align*}
		\| b^i\|_{L^\infty_T \cB^\theta_\infty}
		= \| \fkb\ast \mu^i\|_{L^\infty_T \cB^\theta_\infty}
		\leq \| \fkb\|_{L^\infty_T \cB^\theta_p} \| \mu^i\|_{L^\infty_T L^{p'}_x}
		\leq \| \fkb\|_{L^\infty_T \cB^\theta_p} \| \mu^i\|_{L^\infty_T (L^1_x\cap L^{\infty}_x)} < \infty.
	\end{align*}	
	In particular, since $b^i\in L^\infty_T \cB^\theta_\infty$ with $\theta>1-1/(2H)$, we can invoke the stability estimates from \cite[Thm. 1.4-ii)]{GaleatiGerencser}, for $m=p'\in [1,\infty)$, $\alpha=\theta$ and $q=2$. Together with Corollary \ref{cor:convol_wasserstein} from Appendix \ref{app:besov}, for any $t\in [0,T]$, this yields
	\begin{align*}
		\EE\Big[\sup_{s\in [0,t]} |Y^1_s-Y^2_s|^{p'}\Big]^{\frac{2}{p'}}
		& \lesssim \int_0^t \| \mathfrak{b}_s\ast(\mu^1_s-\mu^2_s)\|_{\cB^{\theta-1}_\infty}^2 \dd s\\
		& \lesssim \int_0^t \| \mathfrak{b}_s\|_{\cB^\theta_p}^2 (\|\mu^1_s\|_{L^\infty_x}^{\frac{2}{p}} + \|\mu^2_s\|_{L^\infty_x}^{\frac{2}{p}}) \mathbb{W}_{p'}(\mu^1_s,\mu^2_s)^2\, \dd s\\
		& \lesssim \| \mu_0\|_{L^\infty_x}^{\frac{2}{p}} \| \mathfrak{b}\|_{L^\infty_T \cB^\gamma_p}^2 \int_0^t  \EE\big[ |Y^1_s-Y^2_s|^{p'}\big]^{\frac{2}{p'}}\, \dd s.
	\end{align*}
By Gr\"onwall's lemma, it follows that $\EE[\sup_{s\in [0,T]} |Y^1_s-Y^2_s|^{p'}]=0$, yielding pathwise uniqueness.

Finally, under condition \eqref{eq:condition_uniqueness2-full}, we know that any weak solution we consider is actually a strong solution. In particular, given $Y^1$, $Y^2$ solutions defined on different probability spaces, we can always construct a coupling of them, namely two processes $\tilde Y^1$ and $\tilde Y^2$ defined on the same probability space and driven by the same data $(\tilde Y_0,\tilde B)$, with $\tilde Y^i$ solving {\renewcommand\SDE{SDE($b^i$,$\tilde Y_0$)}\ref{eq:SDEintro}\renewcommand\SDE{SDE($b$,$X_0$)}} and such that $\mathcal{L}(Y^i)=\mathcal{L}(\tilde Y^i)$. In particular, since $\mu^i_t=\mathcal{L}(Y^i_t)=\mathcal{L}(\tilde Y^i_t)$ for all $t\in [0,T]$, $\tilde Y^i$ are also solutions to the McKean--Vlasov equation \eqref{eq:McKV}, still satisfying \eqref{eq:condition_uniqueness2-full}. By the previous step, $\tilde Y^1=\tilde Y^2$ $\PP$-a.s., so that $\mathcal{L}(Y^1)=\mathcal{L}(Y^2)$, proving uniqueness in law under condition \eqref{eq:condition_uniqueness2-full}.
\end{proof}

\begin{remark}\label{rem:propagation_higher_regularity}
In light of Lemma \ref{lem:propagation_integrability}, one may wonder whether higher regularity of $\law(Y_0)$ is propagated by the McKean--Vlasov dynamics as well. We expect this to be non-trivial but feasible, possibly under further assumptions on $\text{div} \fkb$.
Let us consider the case of linear SDEs first. Proceeding a bit formally, further differentiating both sides of formula \eqref{eq:density_formula} yields
\begin{equation}\label{eq:formula.derivative.density}\begin{split}
	\nabla \rho_t(x) = \EE\Bigg[ \Bigg( \nabla\rho_0 & (\Phi_{0\leftarrow t}(x;\omega))\, D\Phi_{0\leftarrow t}(x;\omega) - \int_0^t D {\rm div} b_s(\Phi_{s\leftarrow t}(x;\omega)) \, D\Phi_{s\leftarrow t}(x;\omega)\dd s\Bigg)\\ 
	& \times \exp\Big( -\int_0^t {\rm div} b_s(\Phi_{s\leftarrow t} (x;\omega)) \dd s\Big) \Bigg].
\end{split}\end{equation}
Here, the term inside the exponential is always uniformly bounded, and similarly one has uniform pathwise bounds on $D\Phi_{0\leftarrow t}(x;\omega)$ thanks to condition $b\in L^1([0,T];\cC^1_b)$.
The possibly challenging term is the one related to $\int_0^t D {\rm div} b_s(\Phi_{s\leftarrow t}) \dd s$, which in lack of better assumptions would be a singular integral associated to a distributional function, $D {\rm div} b$, evaluated along the solution itself.
Still, it might be possible to control this term by stochastic sewing techniques; in the MKV case, one could try to leverage on the (expected) regularity of $\nabla\rho_t$ in a bootstrap fashion to close a Gr\"onwall-type estimate.
Let us also mention that \eqref{eq:formula.derivative.density} was obtained by simply differentiating the quantity inside expectation, but expectation itself might yield further cancellations, in the style of Bismut-Li-Elworthy formulas.
We leave this problem for future research.
\end{remark}

\section{Upper and Lower Gaussian bounds}\label{sec:GaussianBounds}

In this section, we prove Gaussian bounds for the density of the solution to \ref{eq:SDEintro}, when $H\in (0,1/2)$ and $b$ is a suitable distributional drift.
For the reader's convenience, we recall here the statement of Theorem~\ref{mainresult:gaussianbound}.

\begin{theorem}\label{thm:Gaussiantails}
Let $T\in (0,\infty)$, $H \in (0,1/2)$, $\gamma>1-1/(2H)$ and $b \in L^\infty([0,T];\mathcal{B}^\gamma_\infty)$. Assume that $X_{0} = x_{0}\in \R^d$ is deterministic. 
Then the unique admissible solution to~{\renewcommand\SDE{SDE($b$,$x_0$)}\ref{eq:SDEintro}\renewcommand\SDE{SDE($b$,$X_0$)}} has a density $p(t,\cdot)$ for any $t\in (0,T]$ and there exists $C>0$ (depending on $H$, $\gamma$, $T$, $\|b\|_{L^\infty_{T}\mathcal{B}^\gamma_\infty}$, but independent of $x_0$ and $t$) such that, for all $y\in \R^d$,
\begin{align}\label{eq:Gaussianbound}
C^{-1} t^{-dH} \exp\{-C t^{-2H}|y-x_0|^2\} \leqslant p(t,y)\leqslant C t^{-dH} \exp\{-C^{-1} t^{-2H}|y-x_0|^2\}.
\end{align}
\end{theorem}

\begin{remark}
In \cite{LukasThesis}, for $d=1$, a different approach using Malliavin calculus is explored. Under the stronger condition $\gamma>2-1/(2H)$ it yields \eqref{eq:Gaussianbound} and additionally proves that $p(t,\cdot)$ is $\alpha$-H\"older continuous for any $\alpha<1/2$.
\end{remark}

In the following, we assume without loss of generality that $x_{0}=0$, as our estimates will be invariant under shifting the initial condition. Up to time rescaling, one may further reduce to the case $t=T$ (see the end of Section~\ref{sec:proofbounds}).
Similarly to \cite{LiPanloupSieber}, the main idea is to apply Girsanov's theorem and then conditioning on the terminal position $B_T$; in this way, for $X$ being the unique admissible solution to~{\renewcommand\SDE{SDE($b$,$0$)}\ref{eq:SDEintro}\renewcommand\SDE{SDE($b$,$X_0$)}} and $f \in \mathcal{C}_b(\mathbb{R}^d)$, one finds
\begin{multline*}
 \EE[f(X_T)]\\
	 = \EE\Big[ f(B_T) \EE\Big[\exp\Big(\int_0^T \Big(K_H^{-1}\Big(\int_0^{\cdot}b_r(B_r) \dd r\Big)\Big)_s \dd W_s  
	- \frac12 \int_0^T \Big|\Big(K_H^{-1} \Big(\int_0^{\cdot}b_r(B_r) \dd r\Big)\Big)_s\Big|^2 \dd s \Big) \, \Big|\, B_T\Big]\Big],
\end{multline*}
where $K_{H}^{-1}$ is the inverse of the 
linear operator with kernel $K_{H}$ from the representation of fBm \eqref{eq:fBmrepresentation} (see beginning of Section~\ref{sec:expmoments}).
The conditioning leads to the problem of getting exponential moments for (the Cameron-Martin norm of) the singular integral $\int_0^{\cdot}b_r(B_r) \dd r$; note that, conditional on $B_T=x$, the process $(B_r)_{r\in [0,T]}$ is now distributed as a fractional Brownian bridge (fBb). The regularization effect of the fBb is captured by its (partially degenerate) local nondeterminism property
(see Corollary~\ref{cor:LND-bridge}). We deduce this result from of a more general one for Gaussian-Volterra processes (see Proposition~\ref{prop:LND}). 
This allows to apply stochastic sewing leading to finiteness of exponential moments of the singular integral in Lemma~\ref{lem:expmoments1}. 
This is, together with another result on exponential integrability of the finite variation part of a fractional Brownian bridge, stated and proven in 
Section~\ref{sec:expmoments}. 
First however, we present the crucial local nondeterminism property of the bridge. 
Since we believe this result to be of independent interest, we formulate the result for Gaussian bridges for general Volterra kernels.

\subsection{Local nondeterminism of Gaussian bridges} \label{sec:Gaussianbridges}

Throughout the section, let $K$ be a Volterra kernel (see \eqref{eq:defK})
and assume that $K$ satisfies \eqref{eq:propertiesK} and the additional property
\begin{equation*}
K(T,s) \neq 0 ~\text{for almost all}~ s\in [0,T].
\end{equation*}
Consider the Gaussian process $\Volt$ with the following Volterra representation:
\begin{align} \label{eq:Xprocess}
\Volt_t=\int_0^t K(t,s) \, \dd W_s
\end{align}
for some Brownian motion $W$. We call a bridge from $0$ to $y$ with length $T$ over $\Volt$, any process whose law equals the law of $\Volt$ conditioned to hit $y$ at time $T$. Several representations of Gaussian bridges have been studied under various assumptions \cite{BaudoinCoutin,GasbarraEtAl}, and in particular for the Riemann-Liouville fractional Brownian motion \cite{LiPanloupSieber}. 

We collect here some facts on Gaussian bridges, whose proof is postponed to Appendix~\ref{app:bridge}; although such results are well-known (see the aforementioned references), they are often established under quite restrictive assumptions on $K$, which is why we prefer to provide a self-contained proof. By Proposition~\ref{prop:bridgerep}, the law of the Brownian motion $W$ conditioned by the event $\Volt_{T}=y$ is given by
\begin{align*}
\law \left((W_{t})_{t\in [0,T]} \vert \Volt_{T} = y\right) = 
\law \left( (Y^{T,y}_t)_{t\in [0,T]}\right),
\end{align*}
where $(Y^{T,y}_t)_{t\in [0,T)}$ is the semimartingale given by
\begin{equation}\label{eq:semimartingale-underlying-bridge}
	\dd Y^{T,y}_t = V^{T,y}_t \dd t + \dd W_t \coloneqq K(T,t) \bigg( \frac{y}{\int_0^T K(T,u)^2\, \dd u} - \int_0^t \frac{K(T,s)}{\int_s^T K^2(T,u)\, \dd u} \dd W_s \bigg) \dd t + \dd W_t.
\end{equation}
As a consequence, the following process $P^{T,y}$ is a bridge from $0$ to $y$ with length $T$ over $\Volt$:
\begin{align}\label{eq:defBridgeP}
P_t^{T,y}&= \int_{0}^t K(t,s)\, \dd Y^{T,y}_{s}\nonumber\\
&=D_t^{T,y} + \int_0^t K(t,s) \, \dd W_s - \int_0^t \frac{K(T,s)}{\smallint_s^T K(T,u)^2 \, \dd u} \int_s^t K(t,r) K(T,r) \, \dd r \, \dd W_s,
\end{align}
where the deterministic function $D^{T,y}_t$ reads
\begin{align*}
D^{T,y}_{t} = y \int_{0}^t \frac{K(t,s) K(T,s)}{\smallint_{0}^T K(T,u)^2\, \dd u} \, \dd s.
\end{align*}
With these preparations, we can now study the local nondeterminism property of $P^{T,y}$.

\begin{proposition}\label{prop:LND}
The bridge $P_t^{T,y}$ defined in \eqref{eq:defBridgeP} fulfills for $(\xi,t) \in \simp{0}{T}$,
\begin{equation}\label{eq:LND-bridge-general}
	\var( P^{T,y}_t | \mathcal{F}_\xi) \geq \int_\xi^t K(t,s)^2 \, \dd s\  \frac{\smallint_t^T K(T,r)^2 \, \dd r}{\int_\xi^T K(T,u)^2 \, \dd u}.
\end{equation}
\end{proposition}

\begin{proof}
Using It\^o's isometry, from \eqref{eq:defBridgeP} we have that
\begin{align*}
\var(P^{T,y}_t | \mathcal{F}_{\xi})=\int_\xi^t \tilde{K}^T(t,s)^2 \, \dd s,
\end{align*}
where 
\begin{align*}
\tilde{K}^T(t,s)=K(t,s)-K(T,s)\int_s^t \frac{K(t,r) K(T,r)}{\smallint_s^T K(T,u)^2\, \dd u} \, \dd r \eqqcolon  K(t,s)-\hat{K}^T(t,s).
\end{align*}
Integration by parts gives
\begin{align*}
-2 \int_\xi^t K(t,s) \hat K^T(t,s) \dd s 
& = -2 \int_\xi^t K(t,s)K(T,s)\int_s^t \frac{K(t,r) K(T,r)}{\smallint_s^T K(T,u)^2 \, \dd u} \, \dd r \dd s\\
&=\left[\frac{(\smallint_s^t K(t,r) K(T,r)  \, \dd r)^2}{\smallint_s^T K(T,u)^2  \, \dd u}\right]_{s=\xi}^{s=t}-\int_\xi^t \left(\frac{K(T,s)\smallint_s^t K(t,r) K(T,r)  \, \dd r}{\smallint_s^T K(T,u)^2  \, \dd u}\right)^2  \, \dd s\\
&=-\frac{(\smallint_\xi^t K(t,r) K(T,r)  \, \dd r)^2}{\smallint_\xi^T K(T,u)^2  \, \dd u}-\int_\xi^t \hat{K}^T(t,s)^2  \, \dd s.
\end{align*}
Therefore, expanding the square and applying Cauchy-Schwarz inequality, we find
\begin{align*}
\int_\xi^t \tilde{K}^T(t,s)^2  \, \dd s
& = \int_\xi^t \tilde{K}(t,s)^2  \, \dd s - 2 \int_\xi^t K(t,s) \hat K^T(t,s) \dd s + \int_\xi^t \hat{K}^T(t,s)^2\\
&= \int_\xi^t K(t,s)^2 \, \dd s-\frac{(\smallint_\xi^t K(t,r) K(T,r)  \, \dd r)^2}{\smallint_\xi^T K(T,u)^2  \, \dd u}\\
&\geqslant \int_\xi^t K(t,s)^2 \, \dd s-\frac{\smallint_\xi^t K(t,r)^2  \, \dd r \smallint_\xi^t K(T,r)^2  \, \dd r}{\smallint_\xi^T K(T,u)^2  \, \dd u}\\
&=\int_\xi^t K(t,s)^2  \, \dd s \left(1-\frac{\smallint_\xi^t K(T,r)^2  \, \dd r}{\smallint_\xi^T K(T,u)^2  \, \dd u}\right)\\
&=\int_\xi^t K(t,s)^2  \, \dd s\ \frac{\smallint_t^T K(T,r)^2  \, \dd r}{\smallint_\xi^T K(T,u)^2  \, \dd u}. \qedhere
\end{align*}
\end{proof}

Applying Proposition~\ref{prop:LND} and using \eqref{eq:fBm_cond_var}, we immediately deduce the following LND property for the fractional Brownian bridge. 
\begin{corollary} \label{cor:LND-bridge}
Let $K_H$ be the kernel such that $\Volt$ defined in \eqref{eq:Xprocess} is a fractional Brownian motion with Hurst parameter $H \in (0,1)$. Then there exists $c_H>0$ such that, for $(\xi,t) \in \simp{0}{T}$ and $y \in \mathbb{R}$,
\begin{align}\label{eq:LNDbridge}
	\var( P^{T,y}_t | \mathcal{F}_\xi) \geq c_H (t-\xi)^{2H} \frac{(T-t)^{2H}}{(T-\xi)^{2H}}.
\end{align}
\end{corollary}

\subsection{Key exponential estimates}  \label{sec:expmoments}

The main point of this section is to show square exponential integrability of the terms appearing after Girsanov and conditioning. To see the precise terms that have to be estimated, we rigorously lay out in Proposition~\ref{prop:reformulation} the argument formally outlined at the beginning of Section~\ref{sec:GaussianBounds}. From now on, by $P^{T,y}_t$ we solely denote the fractional Brownian bridge, namely the kernel $K$ is now specifically $K_{H}$, the Volterra kernel of fBm, which for $H\in (0,1/2)$ reads
\begin{equation}\label{eq:defKH}
K_H(t,s) = \mathbbm{1}_{\{0<s<t\}}\, \sigma_H\ \left\{ \left(\frac{t(t-s)}{s}\right)^{H-\frac{1}{2}} +\Big( \frac12 - H\Big) s^{\frac12-H} \displaystyle\int_s^t (r-s)^{H-\frac12} r^{H-\frac32}\, \dd r\right\} , 
\end{equation}
with 
\begin{equation*}
\sigma_H = \left( \frac{2H\ \Gamma(3/2-H)}{\Gamma(H+\tfrac{1}{2})\ 
\Gamma(2-2H)}\right)^{\tfrac{1}{2}}.
\end{equation*}
Recall that the fBm admits the Volterra representation \eqref{eq:fBmrepresentation} associated to $K_H$. With a slight abuse, we also denote by $K_H$ the associated operator, so that $(K_H f)_t = \int_0^t K_H(t,s) f(s)\, \dd s$.
As shown in \cite{NualartOuknine}, $K_H$ is invertible on the Cameron-Martin space $K_H(L^2([0,T]))$, so that we can denote by $K_H^{-1}$ its inverse. It can be written explicitly in terms of fractional derivatives, see \cite[eq.~(12)]{NualartOuknine} and equation~\eqref{eq:NualartOuknine} in Appendix~\ref{app:fBm}.

Recall that, without loss of generality, we assume $x_{0}=0$.

\begin{proposition} \label{prop:reformulation}
Let $X$ be the unique solution to~{\renewcommand\SDE{SDE($b$,$0$)}\ref{eq:SDEintro}\renewcommand\SDE{SDE($b$,$X_0$)}}, with $H$, $\gamma$ and $b$ satisfying the same assumptions as in Theorem~\ref{thm:Gaussiantails}. The density $p(T,\cdot)$ of $X_T$ is given by
\begin{equation}
p(T,y)=(2\pi T)^{-dH} \exp\left(-\tfrac{|y|^2}{2T^{2H}}\right) \Psi(T,y),
\end{equation}
where 
\begin{align}
	\Psi(T,y)\coloneqq \EE\Big[\exp\Big(\int_0^T  (K_H^{-1}Z)_s \cdot \dd W_s - \frac12 \int_0^T |(K_H^{-1}Z)_s|^2\, \dd s \Big)\Big\vert B_T=y\Big]
\end{align}
and $Z_\cdot \coloneqq \int_0^\cdot b_r(B_r)\, \dd r$.
Furthermore, $\Psi$ can be rewritten as 
\begin{equation}\label{eq:conditional-density-bridge}
	\Psi(T,y) = \EE\Big[\exp\Big(\int_0^T (K_H^{-1}\tilde{Z}^{T,y})_s \cdot V^{T,y}_s\, \dd s + \int_0^T (K_H^{-1}\tilde{Z}^{T,y})_s\cdot \dd W_s  - \frac12 \int_0^T |(K_H^{-1}\tilde{Z}^{T,y})_s|^2\, \dd s \Big)\Big],
\end{equation}
where $V^{T,y}$ is defined as in~\eqref{eq:semimartingale-underlying-bridge} and
\begin{equation}\label{eq:defZtilde}
\tilde{Z}^{T,y}_\cdot\coloneqq\int_0^\cdot b_r(P^{T,y}_r)\, \dd r.
\end{equation}
\end{proposition}

\begin{proof}
In order to apply Girsanov's theorem for fBm (see \cite{NualartOuknine}), we need to verify Novikov's condition
\begin{align*}
\EE \left[\exp \left(\frac{1}{2} \int_{0}^T \big| (K_H^{-1} \int_{0}^\cdot b_{r}(X_{r})\, \dd r)_s \big|^2\, \dd s\right)\right] <\infty.
\end{align*}
Under Assumption~\ref{ass:strong} and $\gamma<0$, this indeed holds thanks to \cite[Lemma C.3]{GaleatiGerencser}.
Hence, for any $f \in \mathcal{C}_b(\mathbb{R}^d)$, we have
\begin{align*}
	\EE[f(X_T)]
	& = \EE\Big[ f(B_T) \exp\Big(\int_0^T (K_H^{-1}Z)_s\cdot \dd W_s - \frac12 \int_0^T |(K_H^{-1}Z)_s|^2\, \dd s \Big)\Big]\\
	& = \EE\Big[ f(B_T) \EE\Big[\exp\Big(\int_0^T (K_H^{-1}Z)_s\cdot \dd W_s - \frac12 \int_0^T |(K_H^{-1}Z)_s|^2\, \dd s \Big)\Big\vert B_T\Big]\Big]\\
	& = \int_{\R^d} f(y) \dfrac{1}{(2\pi T)^{dH}} \exp\left(-\frac{|y|^2}{2T^{2H}}\right) \Psi(T,y)\, \dd y.
\end{align*}
Conditional on $B_{T}=y$ for $y\in \R^d$, by Proposition~\ref{prop:bridgerep}, $B$ is distributed as a fractional Brownian bridge $P^{T,y}$ and therefore $Z$ is distributed as $\tilde{Z}^{T,y}$. 
Moreover, $W$ is distributed as the process $Y^{T,y}$ given by \eqref{eq:semimartingale-underlying-bridge}, i.e. $\dd Y^{T,y}_t= V^{T,y}_t \dd t + \dd W_t$.  Hence, \eqref{eq:conditional-density-bridge} follows.
\end{proof}
In view of Proposition~\ref{prop:reformulation}, the goal of obtaining Gaussian tails boils down to showing that $\Psi(T,y)$ is bounded from above and below. For the upper bound, thanks to the properties of exponential martingales, it suffices to show that
\begin{equation*}
	\EE\Big[\exp\Big(\lambda \int_0^T |\mathfrak{L}^{T,y}_s| |V^{T,y}_s|\, \dd s + \lambda \int_0^T |\mathfrak{L}^{T,y}_s|^2\, \dd s\Big)\Big]<\infty, \quad \forall\, \lambda>0,
\end{equation*}
where, for $\tilde{Z}^{T,y}$ given in \eqref{eq:defZtilde}, we define
\begin{align*}
\mathfrak{L}^{T,y}_{s} \coloneq (K_H^{-1}\tilde{Z}^{T,y})_{s} = \Big(K_{H}^{-1} \int_{0}^\cdot b_{r}(P^{T,y}_{r})\, \dd r\Big)_{s}. 
\end{align*}
The rest of this section is dedicated to proving square exponential estimates for $V^{T,y}$ (Lemma~\ref{lem:expmoments2}) and $\mathfrak{L}^{T,y}$ (Lemma~\ref{lem:expmoments1}), with respect to suitable $L^p([0,T])$-norms. To estimate $\mathfrak{L}^{T,y}$, it is convenient to establish estimates of $\tilde{Z}^{T,y}$ in a fractional Sobolev space and then to use the mapping properties of the operator $K_H^{-1}$ 
(Lemma~\ref{lem:relation-sobolev-cameron-martin}).

\paragraph*{Estimating $V^{T,y}$.}

\begin{lemma}\label{lem:expmoments2}
Let $y \in \mathbb{R}^d$ and $p<2$. Let $V^{T,y}$ be defined via \eqref{eq:semimartingale-underlying-bridge} with the kernel $K\equiv K_H$ associated to a fractional Brownian motion with Hurst parameter $H<1/2$. Then there exists $c_p>0$ such that, uniformly in $y\in\R^d$, 
\begin{equation} \label{eq:Vy2}
\EE\big[\exp(c_p\|V^{T,y}\|_{L^{p}_{[0,T]}}^2)\big]<\infty.
\end{equation}
\end{lemma}

\begin{proof}
By its definition as a stochastic integral with respect to a deterministic kernel, $V^{T,y}$ is a Gaussian process. Therefore, it suffices to show that 
\begin{equation}\label{eq:goal-V}
	\EE \int_0^T |V^{T,y}_s|^p\, \dd s <\infty, \quad \forall\, p<2.
\end{equation}
Indeed, by \cite[Example 2.3.16]{Bogachev}, \eqref{eq:goal-V} implies that $\law(V^{T,y})$ defines a Gaussian measure on $L^p([0,T])$. 
Property \eqref{eq:Vy2} then follows by Fernique's Theorem \cite[Theorem 2.8.5]{Bogachev}.

Recall from \eqref{eq:semimartingale-underlying-bridge} that 
\begin{align*}
	V^{T,y}_t  = y\frac{K_H(T,t)}{\int_0^T K_H(T,u)^2\, \dd u} - K_H(T,t) \int_0^t \frac{K_H(T,s)}{\int_s^T K_H(T,u)^2\, \dd u}\, \dd W_s\eqqcolon V^1_t - V^2_t.
\end{align*}
It is immediate to see that $V^1\in L^2([0,T])$ since
\begin{align}\label{eq:estimateV1}
	\|V^1\|_{L^2_{[0,T]}}^2 = \frac{\| K_H(T,\cdot)\|_{L^2_{[0,T]}}^2}{\| K_H(T,\cdot)\|_{L^2_{[0,T]}}^4} |y|^2 = \frac{|y|^2}{T^{2H}} . 
\end{align}
Fix $p<2$. By the Burkholder-Davis-Gundy inequality, it holds
\begin{align*}
	\EE\Big[ \int_0^T |V^2_t|^p\, \dd t\Big]
	&  \lesssim  \int_0^T K_H(T,t)^p \bigg( \int_0^t \frac{K_H(T,s)^2}{\big(\int_s^T K_H(T,u)^2\, \dd u\big)^2}\, \dd s \bigg)^{\frac{p}{2}} \dd t\\
	& = \int_0^T K_H(T,t)^p \bigg[ \Big( \int_t^T K_H(T,u)^2\, \dd u\Big)^{-1} - \Big( \int_0^T K_H(T,u)^2\, \dd u\Big)^{-1} \bigg]^{\frac{p}{2}} \dd t\\
	& \lesssim \int_0^T K_H(T,t)^p\, (T-t)^{-Hp}\, \dd t = (\ast),
\end{align*}
where  in the last inequality we used the LND property of fBm~\eqref{eq:fBm_cond_var}. Applying Lemma \ref{lem:kernel-fbm}, we find 
\begin{align}\label{eq:estimateV2}
	(\ast)
	\lesssim \int_0^T \big( (T-t)^{H-\frac12} + t^{H-\frac12} \big)^p (T-t)^{-Hp}\, \dd t
	\lesssim \int_0^T (T-t)^{-\frac{p}{2}} + t^{p(H-\frac12)} (T-t)^{-Hp}\, \dd t
\end{align}
and the last integral is finite since $p<2$ and $H<1/2$.
\end{proof}

\paragraph*{Estimating $\mathfrak{L}^{T,y}$.}

As discussed earlier, to prove exponential estimates for $\mathfrak{L}^{T,y}$, we will estimate a fractional Sobolev norm of $\tilde Z^{T,y}$, see the upcoming Lemma~\ref{lem:expmoments1}.
The arguments including stochastic sewing are similar to previous works (see e.g. \cite{GaleatiGerencser}), however one has to carefully track the singularity at $t=T$ in the local nondeterminism property of the bridge coming from Corollary~\ref{cor:LND-bridge}. Hence, we first prove the supporting results Lemma~\ref{lem:expsewing} and Corollary~\ref{cor:sewingexp}.

In the following, for $\beta>0$, we define the fractional Sobolev space $H^\beta([0,T]) \coloneqq \{f \in L^2([0,T];\mathbb{R}^d): \, \|f\|_{H_{[0,T]}^\beta}<\infty\}$, where the norm is given by
\begin{align*}
	\| f\|_{H^\beta_{[0,T]}}^2 = \| f\|_{L^2_{[0,T]}}^2 + \llbracket f \rrbracket_{H^\beta_{[0,T]}}^2, \quad
	\llbracket f \rrbracket_{H^\beta_{[0,T]}}^2 \coloneqq \int_{[0,T]^2} \frac{|f(t)-f(s)|^2}{|t-s|^{1+2\beta}} \, \dd s \dd t.
\end{align*}
For $\beta>1/2$, thanks to the embedding $H^\beta([0,T])\hookrightarrow \cC_{[0,T]}$, 
we can meaningfully define the closed subspace of $H^\beta([0,T])$ given by
\begin{equation}\label{eq:defHs0space}
H^\beta_0([0,T]) \coloneqq \left\{f \in L^2([0,T];\mathbb{R}^d): f(0)=0,\, \|f\|_{H_{[0,T]}^\beta}<\infty  \right\} .
\end{equation}

\begin{lemma} \label{lem:expmoments1}
Let $y \in \mathbb{R}^d$, $H \in (0,1/2)$, $\gamma \in (1-1/(2H),0)$ and $b \in L^\infty([0,T];\mathcal{B}^\gamma_\infty)$. Then, for $\tilde{Z}^{T,y}_\cdot=\int_0^\cdot b_s(P_s^{T,y})\, \dd s$, there exists $\delta>0$ such that 
\begin{equation*}
	\EE\Big[ \exp\Big( \lambda \| \tilde{Z}^{T,y}\|_{H_{[0,T]}^{H+1/2+\delta}}^2 \Big) \Big]<\infty, \quad \forall\, \lambda>0
\end{equation*}
where the estimate is uniform over $y\in\R^d$.
\end{lemma}

In order to prove Lemma \ref{lem:expmoments1}, we need two auxiliary results.

\begin{lemma} \label{lem:expsewing}
Let $y \in \mathbb{R}^d$, $H \in (0,1/2)$, $\gamma \in (1-1/(2H),0)$. Then there exists $\bar{\lambda}>0$ independent of $y$ such that for 
$f \in L^\infty([0,T];\mathcal{B}^\gamma_\infty) \cap L^\infty([0,T];\mathcal{C}^\infty_b)$ and $0<s<t<T$,
\begin{equation}\label{eq:exp-integrability-1}
	\EE\Bigg[\exp\Big( \frac{\bar{\lambda}}{\| f\|_{L^\infty_{[0,T]}\cB^\gamma_\infty}^2} \frac{|\int_s^t f_r(P_r^{T,y})\, \dd r|^2}{|t-s|^{2\gamma H + 2}} \Big(\frac{T-s}{T-t}\Big)^{2\gamma H}\Big)\Bigg]<\infty.
\end{equation}
\end{lemma}

\begin{proof}
Fix $0<s<t<T$. For $(u,v) \in \simp{s}{t}$, define $A_{u,v}\coloneqq\EE^u \int_u^v f_r(P^{T,y}_r)\, \dd r$. We aim to apply~\cite[Lemma 2.6]{GaleatiGerencser}: provided that there exists a constant $C=C(H,\gamma)$ such that for all $(u,v) \in \simp{s}{t}$,
\begin{enumerate}[label=(\alph*)]
\item \label{cond:sewing1}$\EE^u \delta A_{u,\xi,v}=0$ for $\xi \in [u,v]$,
\item \label{cond:sewing2}$\|A_{u,v}\|_{L^\infty_\Omega}\leqslant C \|f\|_{L^\infty_{[u,v]} \cB^\gamma_\infty} (v-u)^{\gamma H +1} \Big(\frac{T-t}{T-s}\Big)^{\gamma H},$
\end{enumerate}
we can conclude that the sewing of the two-parameter process $A$ is a one-parameter process 
$\mathcal{A}$, which satisfies 
\begin{equation}\label{eq:exp-integrability-A}
	\EE\Bigg[\exp\Big( \frac{\bar{\lambda}}{\| f\|_{L^\infty_{[0,T]} \cB^\gamma_\infty}^2} \frac{|\mathcal{A}_{t}-\mathcal{A}_{s}|^2}{|t-s|^{2\gamma H + 2}} \Big(\frac{T-s}{T-t}\Big)^{2\gamma H}\Big)\Bigg]<\infty,
\end{equation}
for some $\bar\lambda>0$ that does not depend on $f$, $s$ and $t$. Note that \ref{cond:sewing1} clearly holds.
For \ref{cond:sewing2}, by heat kernel estimates (recall $\gamma<0$) and  Corollary~\ref{cor:LND-bridge},
\begin{align*}
	\| A_{u,v}\|_{L^\infty_\Omega}
	\leqslant C \int_u^v \| f_r\|_{\cB^\gamma_\infty} (r-u)^{\gamma H} \Big(\frac{T-r}{T-u}\Big)^{\gamma H} \dd r
	\leqslant C \| f\|_{L^\infty_{[u,v]} \cB^\gamma_\infty} (v-u)^{\gamma H +1} \Big(\frac{T-t}{T-s}\Big)^{\gamma H},
\end{align*}
where we used the fact that $(u,v) \in \simp{s}{t}$ in the second inequality.

Besides, from sewing (see \cite[Lemma 2.6]{GaleatiGerencser}), $\mathcal{A}$ is unique in the class of adapted process such that $\mathcal{A}_s=0$ and
\begin{align*}
  \|\mathcal{A}_v-\mathcal{A}_u-A_{u,v}\|_{L^\infty_\Omega}&\lesssim (v-u)^{\gamma H +1} \Big(\frac{T-t}{T-s}\Big)^{\gamma H},\\
   \|\EE^u[\mathcal{A}_v-\mathcal{A}_u-A_{u,v}]\|_{L^\infty_\Omega}&=0.
\end{align*}
To identify $\mathcal{A}$ with $\int_0^\cdot f_r(P^{T,y}_r)\dd r$, we verify the above. For the first inequality, note that as $f \in L^\infty([0,T];\mathcal{C}_b)$,
\begin{align*}
\|\mathcal{A}_v-\mathcal{A}_u-A_{u,v}\|_{L^\infty_\Omega}\leqslant 2 \|f\|_{L^\infty_{[u,v]}L^\infty_x} (v-u).
\end{align*}
The equality on the conditional expectation follows directly from the definition. 
Hence \eqref{eq:exp-integrability-A} holds with $\mathcal{A}_\cdot = \int_s^\cdot f_r(P_r^{T,y})\, \dd r$, yielding \eqref{eq:exp-integrability-1}.
\end{proof}

Applying an interpolation trick, the previous lemma implies that \eqref{eq:exp-integrability-1} holds for any $\bar{\lambda}$ after \textit{sacrificing} 
a small amount of regularity, i.e. $\gamma-\varepsilon$ instead of $\gamma$.
Choosing $\varepsilon$ small enough still ensures that 
\(\gamma - \varepsilon>1-1/(2H)\). 

\begin{corollary}\label{cor:sewingexp}
Assume that the conditions from Lemma~\ref{lem:expsewing} hold. Then, for any $\varepsilon>0$, there exists an increasing function $\kappa(\lambda)$ (independent of $\|f\|_{L^\infty_{T}\mathcal{B}^\gamma_\infty}$ and $y$) such that
\begin{equation*}
	\EE\Big[\exp\Big(\frac{\lambda}{\| f\|_{L^\infty_{T} \mathcal{B}^{\gamma}_\infty}^2} \frac{|\int_s^t f_r(P^{T,y}_r)\, \dd r|^2}{|t-s|^{2(\gamma-\varepsilon) H + 2}} \Big(\frac{T-s}{T-t}\Big)^{2(\gamma-\varepsilon) H}\Big)\Big]\leqslant \kappa(\lambda), \quad \forall \, \lambda>0.
\end{equation*}
\end{corollary}

\begin{proof}
The argument is almost the same as in \cite[p. 273]{GHM}. By Lemma~\ref{lem:decomposition}, for any $\vep>0$ and $N \in \mathbb{N}$, one can decompose $f\in L^\infty([0,T]; \cB^\gamma_\infty)$ 
as $f=f^{1,N}+f^{2,N}$ with
\begin{align*}
\|f^{1,N}\|_{L^\infty_{T} \mathcal{B}^{\varepsilon}_{\infty}}\leqslant C 2^{N(\varepsilon-\gamma)}\|f\|_{L^\infty_{T}\mathcal{B}^{\gamma}_\infty},\quad \|f^{2,N}\|_{L^\infty_{T}\mathcal{B}^{\gamma-\varepsilon}_\infty}\leqslant C 2^{-N\varepsilon}\|f\|_{L^\infty_{T}\mathcal{B}^{\gamma}_\infty}.
\end{align*}
Let $\lambda>0$ and choose $\tilde{N} \in \mathbb{N}$ large enough such that $\lambda \leqslant C^{-2}2^{2\tilde{N} \varepsilon-1} \bar{\lambda}$, for $\bar\lambda$ corresponding to $\gamma\equiv\gamma-\varepsilon$ in Lemma~\ref{lem:expsewing}. Then
\begin{align*}
\|f\|^{-2}_{L^\infty_{T}\mathcal{B}^\gamma_\infty}\lambda \Big|\int_s^t f_r(P_r^{T,y})\, \dd r\Big|
^2&\leqslant \|f\|^{-2}_{L^\infty_{T}\mathcal{B}^\gamma_\infty}C^{-2}2^{2\tilde{N}\varepsilon}\bar{\lambda}\left(\Big|\int_s^t f_r^{1,\tilde{N}}(P_r^{T,y})\, \dd r\Big|^2
+ \Big|\int_s^t f_r^{2,\tilde{N}}(P_r^{T,y})\, \dd r\Big|^2\right)\\
&\leqslant 2^{\tilde{N}(2\varepsilon-\gamma)}\bar{\lambda}\|f^{1,\tilde{N}}\|^{-2}_{L^\infty_{T}\mathcal{B}^\varepsilon_\infty}\Big|\int_s^t f_r^{1,\tilde{N}}(P_r^{T,y})\, \dd r\Big|^2\\
&\quad + \bar{\lambda} \|f^{2,\tilde{N}}\|^{-2}_{L^\infty_{T}\mathcal{B}^{\gamma-\varepsilon}_\infty}\Big|\int_s^t f_r^{2,\tilde{N}}(P_r^{T,y})\, \dd r\Big|^2\\
& \eqqcolon  I_1+I_2.
\end{align*}
 Since $\mathcal{B}^\varepsilon_\infty \hookrightarrow \mathcal{C}_b$, note that $I_1\leq C_{1} 2^{\tilde{N}(2\varepsilon-\gamma)}(t-s)^2 \bar\lambda$, for some $C_{1}>0$. Using this, we get
\begin{align*}
\EE\Bigg[\exp&\Big(\frac{\lambda}{\| f\|_{L^\infty_{T} \mathcal{B}^{\gamma}_\infty}^2} \frac{\Big|\int_s^t f_r(P_r^{T,y})\, \dd r\Big|^2}{|t-s|^{2(\gamma-\varepsilon) H + 2}} \Big(\frac{T-s}{T-t}\Big)^{2(\gamma-\varepsilon) H}\Big)\Bigg]\\
&\leq \exp\Big(C_{1} 2^{\tilde{N}(2\varepsilon-\gamma)}(t-s)^2 \bar\lambda\Big)\, \EE\Bigg[\exp\Big(\frac{I_2}{|t-s|^{2(\gamma-\varepsilon) H + 2}} \Big(\frac{T-s}{T-t}\Big)^{2(\gamma-\varepsilon) H}\Big)\Bigg]
\end{align*}
and the result follows from Lemma~\ref{lem:expsewing}.
\end{proof}

\begin{proof}[Proof of Lemma~\ref{lem:expmoments1}]
	Let $\varepsilon>0$ such that $\gamma-\varepsilon>1-\frac{1}{2H}$ and $\delta>0$ such that $1+(\gamma-\varepsilon)H = 1/2 + H +\delta$. Define 
\begin{equation*}
c_\delta\coloneqq\int_{\overline{[0,T]}_\leq^2} |t-s|^{-1+\delta}\left(\frac{T-t}{T-s}\right)^{2 (\gamma-\varepsilon) H} \dd t \dd s;
\end{equation*}
the above integral is finite, since the singularity $|t-s|^{-1+\delta}$ along the diagonal $s=t$ is integrable, similarly $2(\gamma-\varepsilon) H>2H-1>-1$ implies that $t\mapsto (T-t)^{2(\gamma-\eps)H}$ is integrable as well, while $(T-s)^{-2(\gamma-\eps)H}\leq T^{-2(\gamma-\eps)H}$.
	Note that, in the definition of $\llbracket f\rrbracket_{H^\beta_{[0,T]}}$, by symmetry we may replace the domain of integration to be $\overline{[0,T]}_\leq^2$ instead of $[0,T]^2$ (up to a multiplicative constant $2$). With this in mind,	
	For any $\lambda>0$, by Jensen's inequality it holds
	\begin{align*}
		\EE\Big[\exp\Big(\lambda &\llbracket \tilde{Z}^{T,y}\rrbracket_{H_{[0,T]}^{H+1/2+\delta/2}}^2\Big)\Big]
		= \EE\bigg[ \exp\Big(2\lambda c_\delta \frac{1}{c_\delta} \int_{\overline{[0,T]}_\leq^2} \frac{1}{|t-s|^{1-\delta}} \frac{|\tilde{Z}^{T,y}_t-\tilde{Z}^{T,y}_s|^2}{|t-s|^{1+2H+2\delta}}\dd s\dd t\Big)\bigg]\\
		& \leq \frac{1}{c_\delta} \int_{\overline{[0,T]}_\leq^2} \frac{1}{|t-s|^{1-\delta}} \Big(\frac{T-t}{T-s}\Big)^{2(\gamma-\varepsilon) H} \underbrace{\EE\Big[\exp\Big(2\lambda c_\delta \frac{|\tilde{Z}^{T,y}_t-\tilde{Z}^{T,y}_s|^2}{|t-s|^{1+2H+2\delta}} \Big(\frac{T-s}{T-t}\Big)^{2(\gamma-\varepsilon) H}\Big) \Big]}_{\Cat} \dd s \dd t
	\end{align*}
	where $\Cat<C(\delta,\lambda,\|b\|_{L^\infty_{T} \mathcal{B}^{\gamma}_\infty})$ by Corollary~\ref{cor:sewingexp}.
	Hence
	\begin{equation}\label{eq:expmomentseminorm}
	\EE\Big[ \exp\Big( \lambda \llbracket \tilde{Z}^{T,y}\rrbracket_{H_{[0,T]}^{H+1/2+\delta}}^2 \Big) \Big]<\infty, \quad \forall\, \lambda>0.
\end{equation}
Finally, note that the following Poincar\'e inequality holds, for $f\in H_{0}^\beta([0,T])$ with $\beta>1/2$:
\[\sup_{t\in [0,T]} |f(t)| \lesssim \llbracket f \rrbracket_{H^\beta_{[0,T]}},\]
see Equations (8.4) and (8.8) in \cite{H2G2}. 
	Thus $\llbracket \cdot \rrbracket_{H^\beta_{[0,T]}}$ is an equivalent seminorm to $\| \cdot \|_{H^\beta_{[0,T]}}$ on $H_{0}^{H+1/2+\delta/2}([0,T])$ and the conclusion follows from \eqref{eq:expmomentseminorm}.
\end{proof}

\subsection{Proof of Theorem~\ref{thm:Gaussiantails}}\label{sec:proofbounds}

As before, we work in the setup of Theorem~\ref{thm:Gaussiantails}, i.e. $b\in L^\infty([0,T];\cB^\gamma_\infty)$ with $H\in (0,1/2)$ and $\gamma>1-1/(2H)$, where w.l.o.g. $\gamma<0$ and $X_0=0$. Following Proposition~\ref{prop:reformulation} and the discussion thereafter, the following is crucial.

\begin{proposition} \label{prop:expmoments}
Let $y \in \mathbb{R}^d$. Let $V^{T,y}$ be as in \eqref{eq:semimartingale-underlying-bridge} for the kernel $K_H$ and $\tilde{Z}^{T,y}$ given in \eqref{eq:defZtilde}. Then for $\mathfrak{L}^{T,y} = K_H^{-1}\tilde{Z}^{T,y}$,  we have
\begin{equation*}
	\EE\Big[\exp\Big(\lambda \int_0^T |\mathfrak{L}^{T,y}_s| |V^{T,y}_s|\, \dd s + \lambda \int_0^T |\mathfrak{L}^{T,y}_s|^2\, \dd s\Big)\Big]<\infty, \quad \forall\, \lambda>0.
\end{equation*}
\end{proposition}

\begin{proof}
For any $\varepsilon, \nu>0$, by H\"older's and Young's inequalities we have
\begin{align}\label{eq:LyVy}
\lambda \int_0^T |\mathfrak{L}^{T,y}_s| |V^{T,y}_s|\, \dd s
&\leqslant \frac{\lambda}{ \sqrt{\nu}} \|\mathfrak{L}^{T,y}\|_{L^{2+\varepsilon}_{[0,T]}} \sqrt{\nu} \,\|V^{T,y}\|_{L^{(2+\varepsilon)/(1+\varepsilon)}_{[0,T]}}\nonumber\\
&\leqslant \frac{\lambda^2}{2\nu} \|\mathfrak{L}^{T,y}\|_{L^{2+\varepsilon}_{[0,T]}}^2 + \frac{\nu}{2} \|V^{T,y}\|^2_{L^{(2+\varepsilon)/(1+\varepsilon)}_{[0,T]}}.
\end{align}
By Lemma~\ref{lem:relation-sobolev-cameron-martin}, for any $\delta>0$, we can choose $\varepsilon$ small enough such that
\begin{equation}\label{eq:embeddingKH}
	\|K_H^{-1}\cdot\|_{L^{2+\vep}_{[0,T]}}\lesssim  \| \cdot\|_{H_{[0,T]}^{H+1/2+\delta}}.
\end{equation}
Combining \eqref{eq:LyVy} and \eqref{eq:embeddingKH}, we see that in order to conclude it is sufficient to show that for some $\delta>0$ and any $\lambda>0$
\begin{equation}\label{eq:goal-L}
	\EE\Big[ \exp\Big( \lambda \| \tilde{Z}^{T,y}\|_{H_{[0,T]}^{H+1/2+\delta}}^2 \Big) \Big]<\infty,
\end{equation}
and that for some $c>0$ and any $p<2$
\begin{equation} \label{eq:Vy1}
\EE\Big[\exp\Big(c\|V^{T,y}\|_{L^{p}_{[0,T]}}^2\Big)\Big]<\infty.
\end{equation}
Note that \eqref{eq:goal-L} holds by Lemma~\ref{lem:expmoments1} and \eqref{eq:Vy1} holds by Lemma~\ref{lem:expmoments2}.
\end{proof}

\begin{proof}[Proof of Theorem~\ref{thm:Gaussiantails}]
We first prove \eqref{eq:Gaussianbound} for $t=T$. We know by Proposition~\ref{prop:reformulation} that
\begin{equation}\label{eq:gaussian_tails_proof1}
p(T,y)=C \exp\Big(-\frac{|y|^2}{2T^{2H}}\Big) \Psi(T,y)
\end{equation}
with
\begin{equation*}
\Psi(T,y) = \EE\Big[\exp\Big(\int_0^T (K_H^{-1}\tilde{Z}^{T,y})_s\cdot V^{T,y}_s\, \dd s + \int_0^T (K_H^{-1}\tilde{Z}^{T,y})_s\cdot \dd W_s  - \frac12 \int_0^T |(K_H^{-1}\tilde{Z}^{T,y})_s|^2\, \dd s \Big)\Big].
\end{equation*}
Recall that $\mathfrak{L}^{T,y} \coloneqq K_H^{-1}\tilde{Z}^{T,y}$. We start by proving the lower bound. Using Jensen's and H\"older's inequalities, and the fact that 
$\int_0^\cdot \mathfrak{L}^{T,y}_s \cdot \dd W_s$ is a martingale, we get for $\varepsilon>0$,
\begin{align*}
\Psi(T,y)&\geqslant \exp\left(-\EE\left[\int_0^T |\mathfrak{L}^{T,y}_s||V_s^{T,y}| \dd s +\frac{1}{2}\int_0^T |\mathfrak{L}^{T,y}_s|^2 \dd s\right]\right)\\
&\geqslant \exp\Big(-\EE\Big[ \tfrac{3}{2}\|\mathfrak{L}^{T,y}\|_{L^{2+\varepsilon}_{[0,T]}}^2 + \|V^{T,y}\|^2_{L^{(2+\varepsilon)/(1+\varepsilon)}_{[0,T]}}\Big]\Big).
\end{align*}
Using Lemma~\ref{lem:expmoments1} and the estimates \eqref{eq:estimateV1} and \eqref{eq:estimateV2} from the proof of Lemma~\ref{lem:expmoments2} , choosing $\varepsilon>0$ small enough, there exists a constant $C>0$ such that 
\begin{align*}
\exp\Big(-\EE\Big[ \tfrac{3}{2}\|\mathfrak{L}^{T,y}\|_{L^{2+\varepsilon}_{[0,T]}}^2 + \|V^{T,y}\|^2_{L^{(2+\varepsilon)/(1+\varepsilon)}_{[0,T]}}\Big]\Big)\geqslant \exp(-C(1+|y|^2)),
\end{align*}
which together with \eqref{eq:gaussian_tails_proof1} gives the lower bound in \eqref{eq:Gaussianbound}.
Proposition~\ref{prop:expmoments} readily implies that $\sup_{y\in\R^d} \Psi(T,y)<\infty$, from which the upper bound follows by virtue of \eqref{eq:gaussian_tails_proof1}.

The case $t \in (0,T]$ follows by the scaling from Lemma~\ref{lem:scaling}, using in particular \eqref{eq:scaling} and that all bounds are given in terms of the norm of the rescaled drift, which can be uniformly controlled by \eqref{eq:scalingdrift}. 
\end{proof}

\begin{remark}
One possible question is whether one can actually allow for $b \in L^2([0,T];\cB^\gamma_\infty)$ with $\gamma>1-1/(2H)$ (this is where one can apply Girsanov with standard fBm, see \cite[Lemma~C.3]{GaleatiGerencser}). The problem is that the local nondeterminism of the fractional Brownian bridge creates a singularity $(1-t)^{2\gamma H}$ and a priori it is not obvious if one can still ensure good integrability when the increments lose partial H\"older regularity.
\end{remark}

\begin{appendices}

\section{Besov spaces and functional inequalities}\label{app:besov}

We recall in this appendix several useful analysis results concerning Besov spaces $\cB^s_{p,q}$, including embeddings, duality relations, scaling and convolution inequalities.
Following \cite{BaDaCh}, we adopt the definition of (inhomogeneous) Besov spaces based on the inhomogeneous Littlewood--Paley (LP) blocks $\{\Delta_j\}_{j\geq -1}$ associated to a radial dyadic partition of unity; in particular this implies that
\begin{align*}
	\Delta_{-1}f=\phi\ast f, \quad
	\Delta_j f = \psi_j\ast f \coloneqq \big(2^{jd}\psi(2^j\cdot)\big)\ast f ~ \text{ for }j\geq 0,
\end{align*}
for some radially symmetric Schwartz functions $\phi, \psi:\R^d\to\R$.
For $n\in\N$, we define the truncation operator $S_n\coloneqq\sum_{j\leq n} \Delta_j$.

\begin{lemma}[Besov embeddings]
For $s \in \mathbb{R}$, $p,q \in [1,\infty]$, $1\leqslant p_1 \leqslant p_2 \leqslant \infty$, $1\leqslant q_1 \leqslant q_2 \leqslant \infty$ and any $\varepsilon>0$,
\begin{align}
\mathcal{B}^{s+\varepsilon}_{p}&\hookrightarrow \mathcal{B}^s_{p,q}\hookrightarrow \mathcal{B}^s_p,\label{eq:embeddingepsilon}\\
\mathcal{B}_{p_1,q_1}^s &\hookrightarrow \mathcal{B}^{s-d(p_1^{-1}-p_2^{-1})}_{p_2,q_2}. \label{eq:embedding2}
\end{align}
\end{lemma}
\begin{proof}
See \cite[page 47, Prop. 2]{Triebel} for \eqref{eq:embeddingepsilon} and \cite[Prop.~2.71]{BaDaCh} for \eqref{eq:embedding2}.
\end{proof}

\begin{lemma}[Scaling]\label{lem:besov_scaling}
	For any $s\in (0,+\infty)$, there exists a constant $C=C(s,d)$ such that
	\begin{equation}\label{eq:besov_scaling_infty}
		\| f(\lambda \, \cdot\, )\|_{\cB^{-s}_\infty} \leq C( 1 + \lambda^{-s})\| f\|_{\cB^{-s}_\infty} \quad \forall\, \lambda>0, \, f\in \cB^{-s}_\infty.
	\end{equation}
	Similarly,
	\begin{equation}\label{eq:besov_scaling_L1}
		\| \lambda^d f(\lambda \, \cdot\, )\|_{\cB^s_1} \leq C( 1 + \lambda^s)\| f\|_{\cB^s_1} \quad \forall\, \lambda>0, \, f\in \cB^s_1.
	\end{equation}
\end{lemma}

\begin{proof}
	We give the proof for $s\in (0,1]$, as the general case follows by induction on higher derivatives. Observe that for any $f\in \cB^{-s}_\infty$ and any $g \in \mathcal{C}^\infty_b$ it holds
	\begin{align*}
		\langle f(\lambda\,\cdot\,), g\rangle = \langle f, \lambda^{-d} g(\lambda^{-1}\,\cdot\,)\rangle
	\end{align*}
therefore by duality (cf. \cite[Prop. 2.76]{BaDaCh}) it suffices to show that
	\begin{equation}\label{eq:besov_scaling_goal}
		\| \lambda^{-d} g(\lambda^{-1}\,\cdot\,)\|_{\cB^s_{1,1}}
		 \lesssim \|g\|_{\cB^s_{1,1}} (1 + \lambda^{-s}) \quad\forall\, \lambda>0, \ \forall\, g\in \mathcal{C}^\infty_b\cap \cB^s_{1,1}.
	\end{equation}
First consider $s\in (0,1)$; in this case, by \cite[Section 17]{Leoni2017}, $\cB^s_{1,1}$ admits the equivalent norm
\begin{align*}
	\| g\tilde\|_{\cB^s_{1,1}} \coloneqq \| g\|_{L^1_x} + \llbracket g \tilde \rrbracket_{\cB^s_{1,1}},
	\quad \llbracket g \tilde\rrbracket_{\cB^s_{1,1}} \coloneqq \int_{\R^d\times \R^d} \frac{|g(x+y)-g(x)|}{|y|^{d+s}} \dd x \dd y.
\end{align*}
An immediate computation shows that
\begin{align*}
	\| \lambda^{-d} g(\lambda^{-1}\,\cdot\,)\|_{L^1_x} = \| g\|_{L^1_x}, \quad
	\llbracket \lambda^{-d} g(\lambda^{-1}\,\cdot\,) \tilde\rrbracket_{\cB^s_{1,1}} = \lambda^{-s} \llbracket g \tilde\rrbracket_{\cB^s_{1,1}}
\end{align*}
yielding \eqref{eq:besov_scaling_goal} in this case. For $s=1$, the argument is identical, the only difference being that the equivalent norm is now given by
\begin{align*}
	\| g\tilde\|_{\cB^1_{1,1}} \coloneqq \| g\|_{L^1_x} + \llbracket g \tilde \rrbracket_{\cB^1_{1,1}},
	\quad \llbracket g \tilde\rrbracket_{\cB^1_{1,1}} \coloneqq \int_{\R^d\times \R^d} \frac{|g(x+y)-2 g(x) + g(x-y)|}{|y|^{d+1}} \dd x \dd y.
\end{align*}
The proof of estimate \eqref{eq:besov_scaling_L1} follows exactly the same pattern, as for $s>0$ the norm $\| g\|_{\cB^s_1}$ is equivalent to $\| g\tilde\|_{\cB^s_1}\coloneqq \| g\|_{L^1_x} + \llbracket g\tilde\rrbracket_{\cB^s_1}$, where
\begin{align*}
		\llbracket g\tilde\rrbracket_{\cB^s_1} = \sup_{0<|h|\leq 1} \frac{\| g(\cdot+h)-g\|_{L^1_x}}{|h|^s}\ \text{ for } s\in (0,1), \quad
		\llbracket g\tilde\rrbracket_{\cB^1_1} = \sup_{0<|h|\leq 1} \frac{\| g(\cdot+h)-2g + g(\cdot-h)\|_{L^1_x}}{|h|^s};
\end{align*}
we leave the details to the reader.
\end{proof}

Throughout the paper we employed a duality statement concerning mixed Lebesgue-Besov spaces $L^r_{[s,t]} \cB^\gamma_{p,q}$. Although such a result is natural in view of the standard one given by \cite[Prop. 2.76]{BaDaCh}, we couldn't find a reference in the literature; so we include here a proof.
We recall that $\langle \cdot,\cdot\rangle$ stands for the duality between Schwartz function and tempered distributions on $\mathbb{R}^d$.

\begin{lemma}[Duality]\label{lem:dualspace}
Let $0\leqslant s <t\leq +\infty$, $\gamma \in \mathbb{R}$ and $p,q,r \in [1,\infty]$. Define the space
\begin{align*}
	\cE = \mathcal{C}^\infty_b([s,t]\times \R^d)\cap L^{r^\prime}([s,t];\cB^{-\gamma}_{p^\prime,q^\prime}).
\end{align*}
Then there exists a constant $C>0$ such that, for all $f \in L^r([s,t];\mathcal{B}^\gamma_{p,q})$, it holds
\begin{align}\label{eq:dualtesting}
	\|f\|_{L^r_{[s,t]}\cB^\gamma_{p,q}}
	\leqslant C \sup \bigg\{ \Big| \int_s^t \langle f(u),\phi(u)\rangle\, \dd u \Big| \ : \ \phi \in \cE, \ \| \phi\|_{L^{r'}_{[s,t]}\cB^{-\gamma}_{p',q'}} \leq 1\bigg\} .
\end{align}
\end{lemma}

\begin{proof}
Without loss of generality we can take $[s,t]=[0,1]$, the other cases being similar.
Moreover, up to considering both $\phi$ and $-\phi$, we may neglect the modulus appearing inside the supremum in \eqref{eq:dualtesting}.
By \cite[Prop. 2.76]{BaDaCh}, $\cB^{-\gamma}_{p',q'}$ is a closed norming space for $\cB^\gamma_{p,q}$ (in the sense of the definition given at \cite[p. 522]{Hytonen2016});
therefore by \cite[Prop. 1.3.1]{Hytonen2016}, $L^{r'}([0,1];\cB^{-\gamma}_{p',q'})$ is a norming space for $L^r([0,1];\cB^\gamma_{p,q})$, namely
\begin{equation}\label{eq:dualtesting_proof_eq1}
	\|f\|_{L^r_{[0,1]}\mathcal{B}^\gamma_{p,q}}
	\lesssim \sup \bigg\{ \int_s^t \langle f(u),\phi(u)\rangle\, \dd u \ : \ \phi \in L^{r^\prime}([0,1],\cB^{-\gamma}_{p^\prime,q^\prime}), \ \| \phi\|_{L^{r'}_{[0,1]}\cB^{-\gamma}_{p',q'}} \leq 1\bigg\} .
\end{equation}

In order to deduce \eqref{eq:dualtesting} from \eqref{eq:dualtesting_proof_eq1}, it suffices to show that, for any fixed $f\in L^r([0,1];\cB^\gamma_{p,q})$ and $\phi \in L^{r^\prime}([0,1];\cB^{-\gamma}_{p^\prime,q^\prime})$ with $\| \phi\|_{L^{r'}_{[0,1]}\cB^{-\gamma}_{p',q'}}\leq 1$, there exists a sequence $(\phi^n)_n\subset \mathcal{C}^\infty_b([0,1]\times \R^d)\cap L^{r^\prime}([0,1];\cB^{-\gamma}_{p^\prime,q^\prime})$, such that
\begin{equation}\label{eq:dualtesting_proof_eq2}
	\langle f, \phi^n\rangle \to \langle f,\phi\rangle, \quad
	\| \phi^n\|_{L^{r^\prime}_{[0,1]}\cB^{-\gamma}_{p^\prime,q^\prime}} \leq \kappa
\end{equation}
for some constant $\kappa>0$ independent of $f$, $\phi$.

We will prove \eqref{eq:dualtesting_proof_eq2} by a two-step approximation procedure, regularizing first in space and then in time.
For $n\in\N$, set $\phi^n\coloneqq S_n \phi$, where $S_n=\sum_{j\leq n} \Delta_j$ is the truncation operator in the $x$-variable. Then it holds 
\begin{align*}
	\langle f(t), S_n\phi(t)\rangle=\langle S_n f(t), \phi(t)\rangle, \quad
	\|\phi^n(t)\|_{\cB^{-\gamma}_{p',q'}} \lesssim \| \phi(t)\|_{\cB^{-\gamma}_{p',q'}};
\end{align*}
moreover, since $\min\{q,q'\}<\infty$, by \cite[Lemma 2.73]{BaDaCh} either $S_n f(t)\to f(t)$ in $\cB^\gamma_{p,q}$ or $\phi^n (t)\to \phi(t)$ in $\cB^{-\gamma}_{p',q'}$. Combining these considerations with dominated convergence theorem, we can conclude that
\begin{align*}
	\int_0^1 \langle  f(t), \phi^n(t)\rangle \dd t \to \int_0^1 \langle  f(t),  \phi(t)\rangle \dd t
\end{align*}
and that \eqref{eq:dualtesting_proof_eq2} holds, at least with $(\phi^n)_n\subset L^{r'}([0,1];\cB^{-\gamma}_{p',q'}\cap \mathcal{C}^\infty_b)$.

So we can now assume that $\phi\in L^r([0,1];\cB^{-\gamma}_{p',q'}\cap \mathcal{C}^\infty_b)$ and it remains to achieve regularity in time as well. The argument is similar: take a symmetric mollification in time, namely $\phi^\vep\coloneqq h^\vep\ast \phi$ with $h^\vep(t)=\vep^{-1} h(t\vep^{-1})$, similarly for $f$ (after possibly having extended $f$, $\phi$ to take value $0$ outside $[0,1]$). Then $\langle f, \phi^\vep	\rangle= \langle f^\vep, \phi \rangle$ and $\min\{r,r'\}<\infty$, so that either $f^\vep \to f$ in $L^r([0,1];\cB^\gamma_{p,q})$ or $\phi^\vep\to \phi$ in $L^{r'}([0,1];\cB^{-\gamma}_{p',q'})$. Arguing as before, we can deduce that \eqref{eq:dualtesting_proof_eq2} holds, concluding the proof.
\end{proof}

As a consequence of Lemma \ref{lem:dualspace}, we deduce that $L^r_{[s,t]} \cB^\gamma_{p,q}$ satisfies a version of the \emph{Fatou property}.

\begin{lemma}\label{lem:uniformBesov}
Let $\gamma \in \mathbb{R}$ and $p,q,r \in [1,\infty]$. Let $(\mu^n)_n$ be a bounded sequence in $L^r_{[s,t]} \cB^\gamma_{p,q}$, namely $\sup_n \| \mu^n\|_{L^r_{[s,t]} \cB^\gamma_{p,q}}<\infty$, such that $\mu^n$ converges weakly to some $\mu^\infty\in \mathcal{S}'([s,t]\times\mathbb{R}^d)$, in the sense that
\begin{equation*}
	\lim_{n\to\infty} \int_s^t \langle \mu^n(u), \phi(u)\rangle\,\dd u 
	= \int_s^t \langle \mu^\infty(u),\phi(u)\rangle\, \dd u,\quad \forall\, \phi\in \cC^\infty_b([s,t]\times\mathbb{R}^d) \cap L^{r'}_{[s,t]} \cB^{-\gamma}_{p',q'}.
\end{equation*}
Then $\mu^\infty \in L^r_{[s,t]} \cB^\gamma_{p,q}$ and $\| \mu^\infty\|_{L^r_{[s,t]} \cB^\gamma_{p,q}} \lesssim \liminf_{n\to\infty} \| \mu^n\|_{L^r_{[s,t]} \cB^\gamma_{p,q}}$.
\end{lemma}

\begin{proof}
Up to refining the subsequence, we may assume that $\lim_{n\to\infty} \| \mu^n\|_{L^r_{[s,t]} \cB^\gamma_{p,q}}$ exists.
By Lemma \ref{lem:dualspace} and the assumption, it holds
\begin{align*}
	\Big|\int_s^t \langle \mu^\infty(u),\phi(u)\rangle \dd u \Big| = \lim_{n\to\infty} \Big|\int_s^t \langle \mu^n(u),\phi(u)\rangle \dd u \Big| \lesssim \lim_{n\to\infty} \| \mu^n\|_{L^r_{[s,t]} \cB^\gamma_{p,q}} \, \| \phi\|_{L^{r'}_{[s,t} \cB^{-\gamma}_{p',q'}}.
\end{align*}
As the estimate holds for any such $\phi$, applying again \eqref{eq:dualtesting} the conclusion follows.
\end{proof}

\begin{lemma} \label{lem:decomposition}
Let $s \in \mathbb{R}$ and $\alpha_1,\alpha_2>0$. Then, there exists $C=C(s)>0$ so that for any $N \in \mathbb{N}$ we can decompose $f\in\cB^s_\infty$ into $f=f^{1,N}+f^{2,N}$ with
\begin{align*}
\|f^{1,N}\|_{\mathcal{B}^{s+\alpha_1}_\infty}\leqslant C 2^{N\alpha_1}\|f\|_{\mathcal{B}^s_\infty},\quad \|f^{2,N}\|_{\mathcal{B}^{s-\alpha_2}_\infty}\leqslant C 2^{-N\alpha_2}\|f\|_{\mathcal{B}^s_\infty}.
\end{align*}
\end{lemma}
\begin{proof}
Let $f^{1,N}=\sum_{i \leqslant N} \Delta_i f$ and $f^{2,N}=\sum_{i>N}\Delta_i f$. Then
\begin{align*}
\|f^{1,N}\|_{\mathcal{B}^{s+\alpha_1}_{\infty}}
& =\sup_{j \in \mathbb{N}} 2^{j(s+\alpha_1)} \|\Delta_j \sum_{i\leqslant N} \Delta_i f\|_{L^\infty_x}\\
&\leqslant \sup_{j \leqslant N} 2^{j(s+\alpha_1)} \sum_{|i-j|\leqslant 1} \|\Delta_i f\|_{L^\infty_x}
\leqslant C 2^{(N+1)\alpha_1} \|f\|_{\mathcal{B}^s_\infty}.
\end{align*}
We omit the proof for $f^{2,N}$ as it follows by the same argument.
\end{proof}

Besov spaces behave nicely under convolution; the next statement is taken from \cite[Thm. 2.2]{KuhnSchilling}.

\begin{lemma}[Convolution]\label{lem:Besovconvolution}
Let $s_1,s_2 \in \mathbb{R}$ and $p,p_1,p_2,q,q_1,q_2 \in [1,\infty]$. Assume that
\begin{align*}
\frac{1}{q}\leqslant \frac{1}{q_1}+\frac{1}{q_2},\quad 1+\frac{1}{p}=\frac{1}{p_1}+\frac{1}{p_2}.
\end{align*}
Then there exists a constant $C>0$, depending on the above parameters, such that for all $f\in \mathcal{B}^{s_1}_{p_1,q_1}$ and $g\in \mathcal{B}^{s_2}_{p_2,q_2}$ it holds that
\begin{equation}\label{eq:Besovconvolution}
\|f\ast g\|_{\mathcal{B}^{s_1+s_2}_{p,q}}\leqslant C \|f\|_{\mathcal{B}^{s_1}_{p_1,q_1}}\|g\|_{\mathcal{B}^{s_2}_{p_2,q_2}}.
\end{equation}
\end{lemma}

We need some useful relations between Besov norms and Wasserstein inequalities.
In the next statement, for $r\in [1,\infty)$, we denote by $\mathbb{W}_r(\mu,\nu)$ the $r$-Wasserstein distance between two probability measures $\mu$ and $\nu$; see \cite[Chap. 6]{Villani2009} for more details. Whenever a probability measure $\mu$ on $\R^d$ is absolutely continuous w.r.t. the Lebesgue measure $\cL^d$ and its density $\dd \mu/\dd \cL^d$ belongs to $L^q_x$, with a slight abuse we identify $\mu$ with its density and simply write $\mu\in \mathcal{P}(\R^d)\cap L^q_x$.

\begin{lemma}\label{lem:wasserstein_negBesov}
For any $r\in [1,\infty)$ and any $\mu,\nu\in\mathcal{P}(\mathbb{R}^d)\cap L^\infty_x$ it holds that
\begin{equation}\label{eq:ineq_wasserstein_negBesov}
	\| \mu-\nu\|_{\cB^{-1}_{r}} \lesssim (\|\mu\|_{L^\infty_x}^{\frac{1}{r'}} + \|\nu\|_{L^\infty_x}^{\frac{1}{r'}}) \mathbb{W}_{r}(\mu,\nu).
\end{equation}
\end{lemma}

\begin{proof}
Without loss of generality we can assume $\mathbb{W}_r(\mu,\nu)<\infty$, the other case being trivial.
We show how to estimate $\| \Delta_j (\mu-\nu)\|_{L^r_x}$ for $j\geq 0$, the case $j=-1$ being similar; recall that $\Delta_j f= \psi_j \ast f$ for $\psi_j(x)=2^{jd} \psi(2^j x)$.

Consider first the case $r=1$. Since $\mu-\nu$ is a finite measure and $\psi_j\in L^1_x$, $\psi_j\ast(\mu-\nu)\in L^1_x$ and so by duality
\begin{align*}
	\| \Delta_j (\mu-\nu)\|_{L^1_x}
	= \sup_{\| f\|_{L^\infty_x}\leq 1} |\langle \psi_j\ast(\mu-\nu),f\rangle|
	= \sup_{\| f\|_{L^\infty_x}\leq 1} |\langle \mu-\nu, \psi_j\ast f\rangle|.
\end{align*}
By Kantorovich--Rubinstein duality (cf. \cite[Rem. 6.5]{Villani2009}), for any such $f$ we have
\begin{align*}
	|\langle \mu-\nu, \psi_j\ast f\rangle|
	\leq \mathbb{W}_1(\mu,\nu) \| \nabla(\psi_j\ast f)\|_{L^\infty_x}
	\leq \mathbb{W}_1(\mu,\nu) \| \nabla \psi_j\|_{L^1} \| f\|_{L^\infty_x}
	\sim 2^j\, \mathbb{W}_1(\mu,\nu) ,
\end{align*}
where in the intermediate estimate we used Young's inequality. Therefore
\begin{align*}
	\sup_{j\geq 0} 2^{-j} \| \Delta_j (\mu-\nu)\|_{L^1_x} \lesssim \mathbb{W}_1(\mu,\nu) ,
\end{align*}
which proves the statement for $r=1$, up to running the same estimate for $\Delta_{-1}$.

Consider now the case $r\in (1,\infty)$; as before, we can employ duality to find
\begin{align*}
	\| \Delta_j (\mu-\nu)\|_{L^1_x}
	= \sup_{\| f\|_{L^{r'}_x}\leq 1} |\langle \mu-\nu, \psi_j\ast f\rangle|
\end{align*}
where now
\begin{align*}
	\| \nabla (\psi_j\ast f)\|_{L^{r'}_x} \leq \| \nabla \psi_j\|_{L^1_x} \| f\|_{L^{r'}_x} \sim 2^j.
\end{align*}
Therefore in order to conclude, it suffices (up to taking $g=f\ast \psi_j$) to show that
\begin{equation}\label{eq:wasserstein_claim}
	|\langle \mu-\nu, g\rangle| \lesssim \| \nabla g\|_{L^{r'}_x} (\|\mu\|_{L^\infty_x}^{\frac{1}{r'}} + \|\nu\|_{L^\infty_x}^{\frac{1}{r'}}) \mathbb{W}_{r}(\mu,\nu), \quad \forall\, g\in W^{1,r}.
\end{equation}
The claim follows by readopting the arguments from \cite[Lem.~A.8]{GaHaMa2022}, but let us give a self-contained proof.
Let $\pi\in \mathcal{P}(\R^d\times \R^d)$ be an optimal coupling for $(\mu,\nu)$, then by Hajlasz's inequality (cf. \cite[Lem.~A.3]{CriDeL2008}) we have
\begin{align*}
	|\langle \mu-\nu, g\rangle|
	& \leq \int_{\R^d\times \R^d} |g(x)-g(y)|\, \pi(\dd x,\dd y)\\
	& \lesssim \int_{\R^d\times \R^d} |x-y| (\mathcal{M}\nabla g(x)+\mathcal{M}\nabla g(y))\, \pi(\dd x,\dd y)\\
	& \lesssim \bigg( \int_{\R^d\times \R^d} |x-y|^r\, \pi(\dd x,\dd y) \bigg)^{\frac{1}{r}} \bigg( \int_{\R^d\times \R^d} (\big|\mathcal{M}\nabla g(x)\big|^{r'}+\big|\mathcal{M}\nabla g(y)\big|^{r'})\, \pi(\dd x,\dd y) \bigg)^{\frac{1}{r'}}\\
	& = \mathbb{W}_r(\mu,\nu) \bigg( \int_{\R^d} \big|\mathcal{M}\nabla g(x)\big|^{r'} \mu(\dd x) + \int_{\R^d} \big|\mathcal{M}\nabla g(y)\big|^{r'} \nu(\dd y) \bigg)^{\frac{1}{r'}}\\
	& \lesssim \mathbb{W}_r(\mu,\nu) \| \mathcal{M}\nabla g\|_{L^{r'}_x} (\|\mu\|_{L^\infty_x}^{\frac{1}{r'}} + \|\nu\|_{L^\infty_x}^{\frac{1}{r'}}) ,
\end{align*}
which finally yields \eqref{eq:wasserstein_claim}, since the maximal operator $\mathcal M$ is bounded on $L^{r'}_x$ for $r\in(1,\infty)$.
\end{proof}

\begin{corollary}\label{cor:convol_wasserstein}
	For any $p\in (1,\infty]$, $\alpha\in\R$ it holds that
\begin{equation}\label{eq:convol_wasserstein}
	\| f\ast (\mu-\nu)\|_{\cB^{\alpha-1}_{\infty}} \lesssim \| f\|_{\cB^{\alpha}_{p}} (\|\mu\|_{L^\infty_x}^{\frac{1}{p}} + \|\nu\|_{L^\infty_x}^{\frac{1}{p}}) \mathbb{W}_{p'}(\mu,\nu)
\end{equation}
for all $f\in \cB^{\alpha}_{p}$ and $\mu,\nu\in\mathcal{P}(\R^d)\cap L^\infty_x$.
\end{corollary}

\begin{proof}
	It follows from combining inequalities \eqref{eq:ineq_wasserstein_negBesov} and \eqref{eq:Besovconvolution}, for the choice $s_1=\alpha$, $s_2=-1$, $q_1=q_2=\infty$, $p_1=p$ and $p_2=p'$.
\end{proof}

We conclude this appendix by recalling some facts about the heat kernel semigroup $\{G_t\}_{t\geq 0}$ defined by \eqref{Gaussiansemigroup}.

\begin{lemma}\label{A.3}
Let $\gamma \in \mathbb{R}$ 
and $p \in [1,\infty]$. Then for any $f \in \mathcal{B}_p^\gamma$ it holds that:
\begin{enumerate} [label=(\alph*)]
    \item $\sup_{t>0}\|G_tf\|_{\mathcal{B}_p^\gamma}\leq \|f\|_{\mathcal{B}_p^\gamma}$; \label{A.3.2}
    \item $\| G_t f - f\|_{\cB^{\gamma-\varepsilon}_p}\lesssim t^{\varepsilon/2} \| f\|_{\cB^{\gamma}_p}$ for all $t\in [0,1]$ and $\varepsilon\in [0,1]$; \label{A.3.3}
    \item $\|G_tf\|_{\mathcal{C}^\beta_b}\lesssim \|f\|_{\mathcal{B}_p^\gamma} (1+t^{(\gamma-d/p-\beta)/2})$ for all $t>0$, provided that $\gamma-d/p<\beta$. \label{A.3.4}
\end{enumerate}
\end{lemma}

\begin{proof}
	For point \ref{A.3.2} we refer to \cite[Lem.~A.3]{Atetal}.
	Concerning \ref{A.3.3}, note that $G_t f - f = (g_t-\delta_0)\ast f$, where the Gaussian density satisfies $\mathbb{W}_1(g_t,\delta_0)\lesssim t^{1/2}$. Combining this bound with Lemma \ref{lem:wasserstein_negBesov} (for $r=1$) and Lemma \ref{lem:Besovconvolution}, we deduce that
	\begin{align*}
		\| G_t f - f\|_{\cB^{\gamma-1}_p}
		= \| f \ast (g_t-\delta_0)\|_{\cB^{\gamma-1}_p}
		\lesssim \| f\|_{\cB^\gamma_p} \| g_t-\delta\|_{\cB^{-1}_1}
		\lesssim \| f\|_{\cB^\gamma_p} t^{1/2}
	\end{align*}
proving the claim for $\varepsilon=1$; the case $\varepsilon=0$ follows from \ref{A.3.2}.
The case $\varepsilon\in (0,1)$ then follows by interpolation of the previous ones, cf. \cite[Thm.~2.80]{BaDaCh}.

To prove \ref{A.3.4}, note that we can reduce to $p=\infty$ by the embedding $\cB^\gamma_p\hookrightarrow \cB^{\gamma-d/p}_\infty$. Then by Besov embeddings and Lemma \ref{lem:Besovconvolution}, we have $\| G_t f\|_{\cC^\beta_b} \lesssim \| G_t f\|_{\cB^\beta_{\infty,1}} \lesssim \| f\|_{\cB^\gamma_\infty} \| g_t\|_{\cB^{\beta-\gamma}_{1,1}}$, so that we only need to show that 
\begin{align*}
	\| g_t\|_{\cB^{\beta-\gamma}_{1,1}}\lesssim (1+ t^{-\frac{\beta-\gamma}{2}}) \quad\forall\, t >0.
\end{align*}
But since $g_t(x)=t^{-d/2}  g_1(t^{-1/2} x)$, this is nothing else than a scaling estimate like the ones already shown in Lemma \ref{lem:besov_scaling}.
\end{proof}

\section{Auxiliary results for fBm}\label{app:fBm}

We collect here some results we need on fBm in the regime $H\in (0,1/2)$; we need to introduce some notation. Recall that fBm admits the representation \eqref{eq:fBmrepresentation} associated to the Volterra kernel $K_H$ given in \eqref{eq:defKH}; with a slight abuse, we also denote by $K_H$ the associated operator, so that $(K_H f)_t = \int_0^t K_H(t,s) f(s) \dd s$.
Let us fix $T\in (0,+\infty)$; as shown in \cite{NualartOuknine}, $K_H$ is invertible on $K_H(L^2([0,T]))$, so that we can denote by $K_H^{-1}$ its inverse. It is given by the formula (cf. \cite[eq.~(12)]{NualartOuknine})
\begin{equation}\label{eq:NualartOuknine}
		K^{-1}_H f = s^{\frac12 -H} D_{0+}^{\frac12-H} s^{H-\frac12} D^{2H}_{0+} f ,
\end{equation}
where for $\alpha\in (0,1)$, $D_{0+}^\alpha$ denotes the fractional Riemann--Liouville derivative given by
\begin{equation}\label{eq:fractional_derivative}
	(D_{0+}^\alpha f)_t = c_\alpha  \bigg(t^{-\alpha} f(t) + \alpha\int_0^t \frac{f(t)-f(s)}{(t-s)^{1+\alpha}}\, \dd s \bigg), \quad \forall\, t\in (0,T],
\end{equation}
for $c_\alpha\coloneqq 1/\Gamma(1-\alpha)$. Lemma \ref{lem:relation-sobolev-cameron-martin} below, which is closely related to \cite[Prop.~C.1]{GaleatiGerencser}, provides connections between the Sobolev regularity of $f$ and the integrability of $K^{-1}_H f$. 
To this end we recall that for $\beta\in (1/2,1)$, $H^\beta_0([0,T])$ is the subspace of the fractional Sobolev space $H^\beta([0,T])$ given by functions $f$ such that $f(0)=0$, see \eqref{eq:defHs0space}.

\begin{lemma}\label{lem:relation-sobolev-cameron-martin}
	Let $H\in (0,1/2)$, $\delta\in (0,1/2-H)$  and $f\in H^{H+1/2+\delta}_0([0,T])$. Then $f\in K_H(L^2([0,T]))$ and there exists $\vep=\vep(\delta,H)>0$ such that $K_H^{-1}f\in L^{2+\vep}([0,T])$ with
	\begin{equation}\label{eq:relation-sobolev-cameron-martin}
	\| K_H^{-1} f\|_{L^{2+\vep}_{[0,T]}} \lesssim \| f\|_{H^{H+1/2+\delta}_{[0,T]}}.
	\end{equation}
\end{lemma}

\begin{proof}
	The fact that $f\in K_H(L^2([0,T]))$ follows from \cite[Prop.~C.1]{GaleatiGerencser}, so we only need to show the estimate \eqref{eq:relation-sobolev-cameron-martin}.
	We may rewrite formula \eqref{eq:NualartOuknine} as
	\begin{align*}
		K^{-1}_H f = \Gamma_H D^{2H}_{0+} f, \quad \text{ for } \quad \Gamma_H g \coloneqq s^{\frac12 -H} D_{0+}^{\frac12-H} s^{H-\frac12} g.
	\end{align*}
	Since $H^{H+1/2+\delta}_0 \hookrightarrow I_{0+}^{H+1/2+\delta/2}(L^2)$ (see e.g. \cite[Prop. 5]{Decreusefond}), we have $f = I_{0+}^{H+1/2+\delta/2} \tilde{f}$ for some $\tilde{f}\in L^2([0,T])$.
Then by \cite[Thm 2.4 and 2.5]{SaKiMa1993}, $D^{2H}_{0+} f = I_{0+}^{1/2-H+\delta/2} \tilde{f}$; again by \cite{Decreusefond}, it follows that $D^{2H}_{0+} f\in H^{1/2-H+\delta/4}([0,T])$.
	 Hence our task reduces to show that $\Gamma_H$ maps $H^{1/2-H+\delta/4}([0,T])$ into $L^{2+\vep}([0,T])$.
	Set $\alpha=1/2-H$, so by the definition \eqref{eq:fractional_derivative}, it holds that
	\begin{align*}
		(s^{\alpha} D_{0+}^{\alpha} s^{-\alpha} g)_t
		= (D_{0+}^\alpha g)_t + c_\alpha t^\alpha \int_0^t \frac{t^{-\alpha}-s^{-\alpha}}{(t-s)^{\alpha+1}} g(s)\, \dd s
		\eqqcolon \big( D_{0+}^\alpha g + \tilde \Gamma_\alpha g\big)_t.
	\end{align*}
	The operator $D^\alpha_{0+}$ maps $H^{\alpha+\delta/4}$ into $H^{\delta/4}$, and consequently in $L^{2+\vep}$ for suitable $\vep=\vep(\delta)>0$ by Sobolev embedding. To handle $\tilde\Gamma_\alpha$, we do the change of variables $s=tu$ to write it as
	\begin{align*}
		(\tilde \Gamma_\alpha g)_t
		= c_\alpha t^{-\alpha} \int_0^1 \frac{1-u^{-\alpha}}{(1-u)^{1+\alpha}}\, g(tu) \, \dd u
		\eqqcolon  t^{-\alpha} \int_0^1 F_\alpha(u) g(tu)\, \dd u .
	\end{align*}
	The function $F_\alpha$ is only unbounded at the points $u=0$ and $u=1$, where it behaves asymptotically respectively as $-u^{-\alpha}$ and $(1-u)^{-\alpha}$. In particular, it belongs to $L^1([0,1])\cap L^2([0,1])$. By Sobolev embedding, $g\in H^{\alpha+\delta/4}([0,T])\hookrightarrow L^q([0,T])$ for $1/q = 1/2 - \alpha+\delta/4$, so that H\"older's inequality yields
	\begin{align*}
		|(\tilde \Gamma_\alpha g)_t|
		\leq t^{-\alpha} \| F_\alpha\|_{L^{q'}_{[0,1]}} \| g(t\,\cdot)\|_{L^q_{[0,1]}}
		\lesssim t^{-\alpha -\frac{1}{q}} \| g\|_{L^q_{[0,T]} } \lesssim t^{\frac{\delta}{4}-\frac12} \| g\|_{H^{\alpha+\delta/4}_{[0,T]}}.
	\end{align*}
	The last pointwise bound implies that, for any $\vep>0$ such that $(2+\vep)(1/2-\delta/4)<1$, it holds
	\begin{align*}
		\| \tilde\Gamma_\alpha g\|_{L^{2+\vep}_{[0,T]}} \lesssim \| g\|_{H^{\alpha+\delta/4}_{[0,T]}},
	\end{align*}
	yielding the conclusion.	
\end{proof}

\begin{lemma}\label{lem:kernel-fbm}
	Let $H\in (0,1/2)$ and recall the Volterra kernel $K_H(t,s)$ associated to fBm from \eqref{eq:defKH}.
	Then there exists a constant $C_H>0$ such that
	\begin{equation*}
		\sigma_H (t-s)^{H-\frac12} \leq K_H(t,s) \leq C_H \Big( (t-s)^{H-\frac12} + s^{H-\frac12} \Big), \quad \forall\, s<t.
	\end{equation*}
\end{lemma}

\begin{proof}
	The first inequality is proved in \cite[Prop. B.2(ii)]{ButLeMyt}, so we only need to show the second one. Since $H<1/2$ and $s<t$, it holds
	\begin{align*}
		\Big(\frac{s}{t}\Big)^{\frac12-H} (t-s)^{H-\frac12} \leq (t-s)^{H-\frac12};
	\end{align*}
	by the change of variables $r=s (1+u) $, it holds
	\begin{align*}
		s^{\frac12-H} \int_s^t (r-s)^{H-\frac12} r^{H-\frac32}\, \dd r
		& =  s^{H-\frac12} \int_0^{t/s-1} u^{H-\frac12} (1+u)^{H-\frac32}\, \dd u\\
		& \leq s^{H-\frac12} \int_0^{+\infty} u^{H-\frac12} (1+u)^{H-\frac32}\, \dd u
	\end{align*}
	where the last integral is finite. Inserting these estimates in formula \eqref{eq:defKH} for $K_H(t,s)$ yields the conclusion.
\end{proof}

\section{Proof of the shifted deterministic sewing lemma}\label{app:sewing}

\begin{proof}[Proof of Lemma~\ref{lem:deterministic_shifted_sewing}]
We start by proving \eqref{eq:deterministic_shifted_conclusion_1}; we point out that this conclusion only relies on Assumption \eqref{eq:deterministic_shifted_assumption_2}.
To this end, for any $n\geq 1$, consider the dyadic sequence $t^n_i = s+(t-s)i/2^n$, $i=0,\ldots,2^n$ and define
\begin{align*}
	S_n = \sum_{i=0}^{2^n-1} A_{t^n_i,t^n_{i+1}}
\end{align*}
so that by assumption $A_{s,t}= S_0$, $\cA_t-\cA_s= \lim_{n\to\infty} S_n$ and
\begin{align*}
	\cA_t-\cA_s - A_{s,t} = \sum_{n=0}^\infty S_{n+1}-S_n
	= \sum_{n=0}^\infty \sum_{i=0}^{2^{n}-1} \delta A_{t^{n+1}_{2i}, t^{n+1}_{2i+1}, t^{n+1}_{2i+2}}.
\end{align*}
Define $t^n_{-1}\coloneqq s-2^{-n}(t-s) \geq \cS$ and observe that $t^{n+1}_{2i}-(t^{n+1}_{2i+2}-t^{n+1}_{2i})= t^n_{i-1}$ for all $i\geq 0$; moreover by construction $(t^{n+1}_{2i}, t^{n+1}_{2i+1}, t^{n+1}_{2i+2})\in \DDminus{\cS}{\cT}$.
By Assumption \eqref{eq:deterministic_shifted_assumption_2} and triangular inequality, it holds
\begin{align*}
	\Big\| \sum_{i=0}^{2^{n}-1} \delta A_{t^{n+1}_{2i}, t^{n+1}_{2i+1}, t^{n+1}_{2i+2}} \Big\|_E
	\leq \sum_{i=0}^{2^{n}-1} w(t^n_{i-1},t^n_{i+1})^{\alpha_2} (t^n_{i+1} - t^n_{i})^{\beta_2}
	= 2^{-n \beta} (t-s)^{\beta_2} \sum_{i=0}^{2^{n}-1} w(t^n_{i-1},t^n_{i+1})^{\alpha_2}.
\end{align*}
Observing that $\{t^n_{2j}\}_{j=0}^{2^{n-1}}$ and $\{t^n_{2j-1}\}_{j=0}^{2^{n-1}}$ form two disjoint partitions respectively of subintervals of $[s-(t-s),t]$, we can use the superadditivity of $w$ and Jensen's inequality to find
\begin{align*}
	\sum_{j=0}^{2^{n-1}-1} w(t^n_{2j},t^n_{2j+2})^{\alpha_2}
	\leq \bigg(\sum_{j=0}^{2^{n-1}-1} w(t^n_{2j},t^n_{2j+2}) \bigg)^{\alpha_2} 2^{(n-1)(1-\alpha_2)}
	\leq w(s-(t-s),t)^{\alpha_2}\, 2^{n(1-\alpha)},
\end{align*}
similarly for the odd sequence.
Combining everything, we arrive at
\begin{align*}
	\|\cA_t-\cA_s - A_{s,t}\|_E
	\leq w(s-(t-s),t)^{\alpha_2} (t-s)^{\beta_2}\, 2^{\alpha_2} \sum_{n=0}^\infty 2^{-n(\alpha_2+\beta_2-1)}
\end{align*}
where the last series is finite since $\alpha_2+\beta_2>1$; this yields \eqref{eq:deterministic_shifted_conclusion_1}.

Now consider any $(s,t)\in \simp{S}{T}$ and define $t_n=s+2^{-n}(t-s)$ for $n\geq 0$. Observe that $t_{n+1}-(t_n-t_{n+1})= s\geq \cS$ and $t_n-t_{n+1}= 2^{-n-1} (t-s)$, therefore we can apply \eqref{eq:deterministic_shifted_assumption_1} and \eqref{eq:deterministic_shifted_conclusion_1} to find
\begin{align*}
	\| \cA_t-\cA_s\|_E
	& \leq \sum_{n=0}^\infty \| \cA_{t_{n}}-\cA_{t_{n+1}}\|_E
	\leq \sum_{n=0}^\infty \| \cA_{t_{n}}-\cA_{t_{n+1}}-A_{t_{n+1},t_n}\|_E + \sum_{n=0}^\infty \| A_{t_{n+1},t_n}\|_E\\
	& \leq C_1 \sum_{n=0}^\infty w_2(s,t_n)^{\alpha_2} (2^{-(n+1)}(t-s))^{\beta_2} + \sum_{n=0}^\infty w_1(s,t_n)^{\alpha_1} (2^{-(n+1)}(t-s))^{\beta_1}\\
	& \leq w_2(s,t)^\alpha (t-s)^{\beta_2} C_1 \sum_{n=1}^\infty 2^{-n\beta_2} + w_1(s,t)^{\alpha_1} (t-s)^{\beta_1} C_1 \sum_{n=1}^\infty 2^{-n\beta_1}
\end{align*}
which yields \eqref{eq:deterministic_shifted_conclusion_2} since $\beta_1,\, \beta_2>0$.
\end{proof}

\section{Representation of Gaussian-Volterra bridges}\label{app:bridge}

In the following we prove that the process $P^{T,y}$ given by \eqref{eq:defBridgeP} is a bridge from $0$ to $y$ with length $T$ over the Gaussian-Volterra process $\Volt$ defined in \eqref{eq:Xprocess}. Such representations were initially proposed in \cite[Thm. 3.1]{BaudoinCoutin}, under assumptions that would however not include the fractional Brownian motion with $H>1/2$. Here, we follow closely the proofs of Lemma 4.1 and Corollary 4.2 in \cite{LiPanloupSieber} which are given for the Riemann-Liouville process, and adapt it to more general kernels. We recall that $K$ satisfies the following:
\begin{itemize}
\item $\int_{[0,T]} K(t,s)^2 \, \dd s<\infty$ for all $t\in [0,T]$;
\item $K(t,s)= 0$ for $s>t$;
\item $K(T,s) \neq 0$ for almost all $s\in [0,T]$.
\end{itemize}

Let us recall the notations from Section~\ref{sec:Gaussianbridges}: the Gaussian process with kernel $K$ is defined in~\eqref{eq:Xprocess} by $\Volt_{t} = \int_{0}^t K(t,s)\, \dd W_{s}$, while the process $P^{T,y}$ defined in~\eqref{eq:defBridgeP} is given for $t\in [0,T]$ by
\begin{align*}
P_t^{T,y}=y \int_{0}^t \frac{K(t,s) K(T,s)}{\smallint_{0}^T K(T,u)^2\, \dd u} \, \dd s
+ \int_0^t K(t,s) \, \dd W_s - \int_0^t \frac{K(T,s)}{\smallint_s^T K(T,u)^2 \, \dd u} \int_s^t K(t,r) K(T,r) \, \dd r \, \dd W_s.
\end{align*}

\begin{proposition}\label{prop:bridgerep}
The law of the Brownian motion $W$ conditioned by the event $\Volt_{T}=y$ is given by
\begin{align*}
\law \left((W_{t})_{t\in [0,T]} \vert \Volt_{T} = y\right) = \law \left( (Y^{T,y}_t)_{t\in [0,T]}\right),
\end{align*}
where $Y^{T,y}$ is the semimartingale defined in \eqref{eq:semimartingale-underlying-bridge}.
Moreover, the process $(P_{t}^{T,y})_{t\in [0,T]}$ given by $P_{t}^{T,y} = \int_{0}^t K(t,s)\, \dd Y^{T,y}_{s}$ is a bridge from $0$ to $y$ with length $T$ over $\Volt$, that is
\begin{equation*}
\law\big((P^{T,y}_t)_{t \in [0,T]}\big)=\law\left(\left(\int_0^t K(t,s)\, \dd W_s\right)_{t \in [0,T]}  \bigg\vert \int_0^T K(T,s)\, \dd W_s=y\right).
\end{equation*}
\end{proposition}

\begin{proof}
Let us introduce the martingale $\Volt^T$ and its quadratic variation $\tau$ as
\begin{align*}
\Volt^T_{t} \coloneqq \int_{0}^t K(T,s)\, \dd W_{s}, \quad \tau(t) \coloneqq \int_{0}^t K(T,s)^2\, \dd s, \quad  t\in [0,T].
\end{align*}
From the assumption on $K$, $\tau$ is a bijection from $[0,T]$ to $[0,\tau(T)]$. Hence we also introduce the Brownian motion
\begin{align*}
\widetilde{W}_{t} \coloneqq \Volt^T_{\tau^{-1}(t)} = \int_{0}^{\tau^{-1}(t)} K(T,s)\, \dd W_{s},\quad t\in [0, \tau(T)].
\end{align*}
Let $Z=A+M$ be a continuous semimartingale, with $M$ the continuous martingale part and $A$ the finite variation process, assumed here to be absolutely continuous almost surely. Denote by $\mathcal{I}(Z)$ the class of functions (we will not need to integrate processes) $f$ such that 
 $\int_{0}^T |f(s)|\, \big|\frac{\dd A_{s}}{\dd s}\big|\, \dd s$ and $\EE\int_{0}^T |f(s)|^2\, \dd \langle Z\rangle_{s} <\infty$. A function $f\in \mathcal{I}(Z)$ can be integrated against $Z$ and the chain rule for Lebesgue-Stieltjes integrals yields $\int_{0}^T f(s)\, \dd Z_{s} = \int_{0}^T f(s)\, \frac{\dd A_{s}}{\dd s} \, \dd s + \int_{0}^T f(s)\, \dd M_{s}$. 
As in \cite{LiPanloupSieber}, define a linear map $\Phi$ which takes any semimartingale $Z$ such that $\frac{1}{K(T,\cdot)} \in \mathcal{I}( Z)$
 and maps it to
\begin{align*}
\Phi(Z)_{t} \coloneqq \int_{0}^t \frac{1}{K(T,s)}\, \dd Z_{s},\quad t\in [0,T].
\end{align*}

With these notations, observe that $(W,\Volt^T) = \big(\Phi(\widetilde{W}\circ \tau), \widetilde{W}\circ \tau\big)$, so that
\begin{align}\label{eq:lawbridge}
\law \left((W_{t})_{t\in [0,T]} \vert \Volt^T_{T} = y\right) &= \law \left( (\Phi(\widetilde{W}\circ \tau)_{t})_{t\in [0,T]} \vert \widetilde{W}\circ \tau(T) = y\right) \nonumber\\
& = \law \left((\Phi(\beta^{y}\circ \tau)_{t})_{t\in [0,T]}\right),
\end{align}
where $\beta^{y}$ is a Brownian bridge conditioned to hit $y$ at time $\tau(T)$. One such bridge is classically given as the solution of the following SDE:
\begin{equation*}
\beta^y_{t} = \int_{0}^t \frac{y-\beta^y_{s}}{\tau(T)-s}\, \dd s + \widetilde W_{t}, \quad t\in [0,\tau(T)].
\end{equation*}
Evaluating this process at time $\tau(t)$ and performing the change of variable $s \equiv \tau(s)$ yields, for $t\in [0,T]$,
\begin{align*}
\beta^y_{\tau(t)} &= \int_{0}^{\tau(t)} \frac{y-\beta^y_{s}}{\tau(T)-s}\, \dd s + \int_{0}^t K(T,s)\, \dd W_{s}\\
&= \int_{0}^{t} \frac{y-\beta^y_{\tau(s)}}{\tau(T)-\tau(s)}\, \tau'(s)\, \dd s + \int_{0}^t K(T,s)\, \dd W_{s}\\
&= \int_{0}^{t} \frac{y-\beta^y_{\tau(s)}}{\int_{s}^T K(T,r)^2\, \dd r}\, K(T,s)^2\, \dd s + \int_{0}^t K(T,s)\, \dd W_{s}. 
\end{align*}
The linear equation satisfied by $\beta^y\circ \tau$ can be solved by considering first the equation
\begin{align*}
z'(t) + \frac{K(T,t)^2}{\int_{t}^T K(T,r)^2\, \dd r} z(t)= 0.
\end{align*}
Since
\[ \frac{K(T,t)^2}{\int_{t}^T K(T,r)^2\, \dd r} = \left(-\log \int_{t}^T K(T,r)^2\, \dd r\right)',\]
this yields
\[z(t) = z(0)\, \frac{\int_{t}^T K(T,r)^2\, \dd r}{\int_{0}^T K(T,r)^2\, \dd r}.\]
Now the method of variation of constants gives
\begin{align*}
(\beta^y\circ \tau)_{t} = \int_{t}^T K(T,r)^2\, \dd r \Bigg(y \int_{0}^t \frac{K(T,s)^2}{ \left(\int_{s}^T K(T,r)^2\, \dd r\right)^2}\, \dd s + \int_{0}^t \frac{K(T,s)}{\int_s^T K(T,r)^2\, \dd r}\, \dd W_{s}\Bigg).
\end{align*}
Note that for any $t<T$, $\inf_{s\in [0,t]} \int_s^T K(T,r)^2\, \dd r>0$, thus $s\mapsto \frac{K(T,s)}{\int_s^T K(T,r)^2\, \dd r} \in L^2([0,t])$. The semimartingale decomposition of $\beta^y\circ \tau$ is
\begin{align*}
\dd (\beta^y\circ \tau)_{t} &= K(T,t)^2 \left( -y \int_{0}^t \frac{K(T,s)^2}{ \big(\int_{s}^T K(T,r)^2\, \dd r\big)^2} \dd s + \frac{y}{\int_{t}^T K(T,r)^2\, \dd r} - \int_{0}^t \frac{K(T,s)}{\int_{s}^T K(T,r)^2\, \dd r} \dd W_{s} \right) \dd t \\
&\quad + K(T,t)\, \dd W_{t}.
\end{align*}
Using
\begin{align*}
\int_0^t \frac{K(T,s)^2}{(\int_s^T K(T,r)^2\,  \dd r)^2} \dd s
	= \int_0^t \frac{\dd}{\dd s} \Big(\frac{1}{\int_s^T K(T,r)^2\,  \dd r} \Big) \dd s
	= \frac{1}{\int_t^T K(T,s)^2 \, \dd s} - \frac{1}{\int_0^T K(T,s)^2\,  \dd s},
\end{align*}
it follows that
\begin{align*}
\dd (\beta^y\circ \tau)_{t} &= K(T,t)^2 \left( \frac{y}{\int_{0}^T K(T,r)^2\, \dd r} - \int_{0}^t \frac{K(T,u)}{\int_{u}^T K(T,r)^2\, \dd r} \dd W_{u} \right) \dd t+ K(T,t)\, \dd W_{t}\\
&\eqqcolon  a^{T,y}_{t} \dd t+ K(T,t)\, \dd W_{t}.
\end{align*}
Now for $t\in (0,T)$, recall that $\inf_{u\in [0,t]} \int_u^T K(T,r)^2\, \dd r>0$. So using that $s\mapsto \frac{1}{K(T,s)} a^{T,y}_{s}$ is integrable on $[0,t]$, the chain rule on $[0,t]$ for Lebesgue-Stieltjes integrals gives
\begin{align*}
\Phi(\beta^y\circ \tau)_{t} &= \int_{0}^t \frac{1}{K(T,s)}\, \dd (\beta^y\circ \tau)_{s} \\
&= \int_{0}^t K(T,s) \left( \frac{y}{\int_{0}^T K(T,r)^2\, \dd r} - \int_{0}^s \frac{K(T,u)}{\int_{u}^T K(T,r)^2\, \dd r} \dd W_{u} \right) \dd s + W_{t}.
\end{align*}
Using again the chain rule, one gets that for $t<T$
\begin{align*}
\int_{0}^t K(t,s)\, \dd \Phi(\beta^y\circ \tau)_{s} &= \int_{0}^t  K(t,s) K(T,s) \left( \frac{y}{\int_{0}^T K(T,r)^2\, \dd r} - \int_{0}^s \frac{K(T,u)}{\int_{u}^T K(T,r)^2\, \dd r} \dd W_{u} \right) \dd s\\
&\quad  + \int_{0}^t K(t,s)\, \dd W_{s}.
\end{align*}
After applying the stochastic Fubini Theorem, one recognises that 
\begin{align*}
\int_{0}^t K(t,s)\, \dd \Phi(\beta^y\circ \tau)_{s} = P^{T,y}_{t}.
\end{align*}
By \eqref{eq:lawbridge}, $\law \left((\Phi(\beta^{y}\circ \tau)_{t})_{t\in [0,T]}\right) = \law \left((W_{t})_{t\in [0,T]} \vert \Volt^T_{T} = y\right)$. Thus
\begin{align*}
\law \left((P^{T,y}_{t})_{t\in [0,T]}\right) = \law \bigg( \Big(\int_{0}^t K(t,s)\, \dd W_{s}\Big)_{t\in [0,T]}\, \bigg\vert\, \Volt^T_{T} = y\bigg). \qquad\qquad \qedhere
\end{align*}
\end{proof}

\end{appendices}

\bibliographystyle{amsplainhyper_m}
\bibliography{skew}
\end{document}